\documentclass[12pt]{amsart}
\usepackage {amscd}
\usepackage[german,english]{babel}
\usepackage {amsfonts}
\usepackage{slashed}
\usepackage[all]{xy}
\usepackage{amssymb,amsmath, amsthm,latexsym,verbatim}
\usepackage{stmaryrd}
\usepackage{a4wide}
\usepackage[hypertex]{hyperref}\usepackage{enumitem}

\newcommand{\C}{\mathbb{C}}
\newcommand{\R}{\mathbb{R}}

\newcommand{\cH}{\mathcal{H}}
\newcommand{\N}{\mathbb{N}}

\renewcommand{\Re}{\operatorname{Re}}

\newcommand{\bs}{\backslash}

\newcommand{\s}{\operatorname{s}}

\newcommand{\kL}{\mathfrak{k}}
\newcommand{\mL}{\mathfrak{m}}
\newcommand{\aL}{\mathfrak{a}}
\newcommand{\gL}{\mathfrak{g}}
\newcommand{\pL}{\mathfrak{p}}
\newcommand{\nf}{\mathfrak{n}}
\newcommand{\hL}{\mathfrak{h}}

\newcommand{\fP}{\mathfrak{P}}
\newcommand{\cE}{\mathcal{E}}
\newcommand{\Arg}{\operatorname{arg}}
\newcommand{\rk}{\operatorname{rk}}

\newcommand{\Gl}{\operatorname{GL}}

\newcommand{\Id}{\operatorname{Id}}

\newcommand{\Spin}{\operatorname{Spin}}

\newcommand{\Ind}{\operatorname{Ind}}

\newcommand{\SU}{\operatorname{SU}}
\newcommand{\Ad}{\operatorname{Ad}}

\newcommand{\SO}{\operatorname{SO}}
\newcommand{\Iim}{\operatorname{Im}}
\newcommand{\CC}{\operatorname{C}}

\newcommand{\vol}{\operatorname{vol}}
\newcommand{\Tr}{\operatorname{Tr}}
\newcommand{\tr}{\operatorname{tr}}
\newcommand{\End}{\operatorname{End}}
\newcommand{\Real}{\operatorname{Re}}
\newcommand{\spec}{\operatorname{spec}}
\newcommand{\rel}{\operatorname{reg}}
\newcommand{\sym}{\operatorname{sym}}

\newtheorem {thrm}{Theorem}[section]
\newtheorem {prop}[thrm] {Proposition}
\newtheorem {lem}[thrm] {Lemma}

\theoremstyle{definition}
\newtheorem{defn}[thrm] {Definition}
\theoremstyle{remark}
\newtheorem {bmrk}[thrm] {Remark}

\begin{document}
\title[]
{Analytic torsion of complete hyperbolic manifolds of finite volume}
\date{\today}

\author{Werner M\"uller}
\address{Universit\"at Bonn\\
Mathematisches Institut\\
Endenicher Allee 60\\
D -- 53115 Bonn, Germany}
\email{mueller@math.uni-bonn.de}

\author{Jonathan Pfaff}
\address{Universit\"at Bonn\\
Mathematisches Institut\\
Endenicher Alle 60\\
D -- 53115 Bonn, Germany}
\email{pfaff@math.uni-bonn.de}


\begin{abstract}
In this paper we define the analytic torsion for a complete
oriented hyperbolic manifold of finite volume. It depends on 
a representation of the fundamental group. For manifolds of odd
dimension, we study the asymptotic
behavior of the analytic torsion with respect to certain sequences of 
representations obtained by restriction of irreducible representations of
the group of isometries of the hyperbolic space to the fundamental group. 
\end{abstract}
\maketitle
\section{Introduction}
Let $X$ be an oriented hyperbolic manifold of dimension
$d$. Let $G=\Spin(d,1)$, $K=\Spin(d)$. Then there exists a discrete,
torsion free subgroup $\Gamma\subset G$ such that
$X=\Gamma\backslash\mathbb{H}^{d}$,
where $\mathbb{H}^{d}\cong G/K$ is the d-dimensional hyperbolic space.
First assume that $X$ is compact and $d=2n+1$. 
Let $\tau$ be an irreducible finite
dimensional representation of $G$. Restrict $\tau$ to $\Gamma$ and let $E_\tau$
be the associated flat vector-bundle over $X$. By \cite{MM} one can equip 
$E_\tau$ with a canonical metric, called admissible metric. Let $T_X(\tau)$
be the Ray-Singer analytic torsion with respect to the hyperbolic metric of
$X$ and the admissible metric in $E_\tau$ (see \cite{RS}, \cite{Mu3}). In
\cite{MP} we introduced special sequences $\tau(m)$, $m\in\N$,  of irreducible
representations of $G$ and we  studied the asymptotic behavior of 
$T_X(\tau(m))$ as $m\to\infty$. The representations $\tau(m)$ are defined as 
follows.
Fix natural numbers $\tau_{1}\geq\tau_{2}\geq\dots\geq\tau_{n+1}$. For
$m\in\mathbb{N}$ let $\tau(m)$ be the finite-dimensional irreducible 
representation of $G$ with highest weight 
$(\tau_{1}+m,\dots,\tau_{n+1}+m)$ (see \cite[p. 365]{GW}).
By Weyl's dimension formula there exists a 
constant $C>0$ such that
\begin{equation}\label{weyldim}
\dim(\tau(m))=Cm^{\frac{n(n+1)}{2}}+O(m^{\frac{n(n+1)}{2}-1}),\quad
m\to\infty. 
\end{equation}
One of the main results of \cite{MP} is the following asymptotic formula:
There exists a constant $C(n)>0$, which depends only on $n$, such that
\begin{equation}\label{asymp1}
-\log T_{X}(\tau(m))=C(n)\vol(X)m\cdot\dim(\tau(m))+O(m^{\frac{n(n+1)}{2}})
\end{equation}
as $m\longrightarrow\infty$. The $3$-dimensional case was first treated in
\cite{Mu2}. This result has been used in \cite{MaM} to study
the growth of torsion in the cohomology of arithmetic hyperbolic 3-manifolds.

The main goal of the present paper is to extend \eqref{asymp1} to complete
oriented hyperbolic manifolds  of finite volume. 
Let $\Gamma\bs\mathbb{H}^d$ be such a manifold. 
To simplify some of the considerations we will assume that $\Gamma$ satisfies
the following condition: For every $\Gamma$-cuspidal parabolic subgroups 
$P=M_PA_PN_P$ of $G$ we have
\begin{equation}\label{a1}
\Gamma\cap P=\Gamma\cap N_P.
\end{equation}
We note that this condition is satisfied, if $\Gamma$ is ``neat'', which means 
that the group generated by the eigenvalues of any $\gamma\in\Gamma$ contains 
no roots of unity $\ne1$. We need \eqref{a1} to
eliminate some technical difficulties related to the Selberg trace
formula.

The first problem is to define the
analytic torsion for non-compact hyperbolic manifolds of finite volume. 
The Laplace operator $\Delta_p(\tau)$ on $E_\tau$-valued $p$-forms has then a
continuous spectrum and therefore, the heat operator $\exp(-t\Delta_p(\tau))$ 
is not trace class. So the usual zeta function regularization can not be used
to define the analytic torsion in this case. To overcome this problem we use
a regularization of the trace of the heat operator which is similar to the
$b$-trace of Melrose \cite{Me}. This kind of regularization was also used 
by Park \cite{Pa} in the case of unitary representations of $\Gamma$.

The regularization of the trace of the heat operator is defined as follows. 
Chopping off the cusps at sufficiently high level $Y>Y_0$, we get a compact
submanifold $X(Y)\subset X$ with boundary $\partial X(Y)$. Let
$K^{p,\tau}(t,x,y)$
be the kernel of the heat operator $\exp(-t\Delta_p(\tau))$. Then it follows 
that there exists $\alpha(t)\in\R$ such that 
$\int_{X(Y)} \tr K^{p,\tau}(x,x,t)\,dx -
\alpha(t)\log Y$ has a limit as $Y\to\infty$. Then we put
\begin{equation}\label{regtrace1}
\Tr_{\rel}\left(e^{-t\Delta_p(\tau)}\right):=\lim_{Y\to\infty}\left(
\int_{X(Y)}\tr K^{p,\tau}(t,x,x)\,dx-\alpha(t)\log Y\right).
\end{equation}
We note that one can also use relative traces as in \cite{Mu3} to regularize
the trace of the heat operator. The methods are closely related.

It turns out that the right hand side of \eqref{regtrace1} equals the spectral 
side of the Selberg trace formula applied to the heat operator 
$\exp(-t\Delta_p(\tau))$. Using 
the Selberg trace formula, it follows that 
$\Tr_{\rel}\left(e^{-t\Delta_p(\tau)}\right)$ has asymptotic expansions as 
$t\to +0$ and as $t\to\infty$. This permits to define the spectral zeta
function. Let $(\tau_1,\dots,\tau_{n+1})$ be the highest weight of $\tau$.
If $\tau_{n+1}\neq0$, then it follows that 
$\Tr_{\rel}\left(e^{-t\Delta_p(\tau)}\right)$ is exponentially decreasing as 
$t\to\infty$. In this case the definition of the zeta function is
simplified. It is given by
\begin{equation}\label{speczeta}
\zeta_{p}(s;\tau):=\frac{1}{\Gamma(s)}\int_{0}^{\infty}t^{s-1}
\Tr_{\rel}(e^{-t\Delta_p(\tau)})\,dt.
\end{equation}
The integral converges absolutely and uniformly on compact subsets of 
the half-plane $\Re(s)>d/2$ and admits a meromorphic continuation
to $\C$ which is regular at $s=0$. In analogy to the compact case we now 
define the 
analytic torsion $T_X(\tau)\in\R^+$ with respect to  $E_{\tau}$ by
\begin{equation}\label{def-tor3}
T_X(\tau):=\exp{\left(\frac{1}{2}\sum_{p=1}^{d}(-1)^pp\frac{d}{ds}
\zeta_{p}(s;\tau)\big|_{s=0}\right)}.
\end{equation}
Now we can state the main result of this paper.
\begin{thrm}\label{theo1}
Let $X=\Gamma\bs\mathbb{H}^{2n+1}$ be a $(2n+1)$-dimensional, complete, 
oriented, hyperbolic manifold of finite volume. Assume that $\Gamma$ satisfies
\eqref{a1}.
There exists a constant $C(n)>0$ which depends only on $n$, such that we have
\begin{align*}
\log
T_{X}(\tau(m))=-C(n)\vol(X)m\cdot\dim(\tau(m))+O\left(m^{\frac{
n(n+1)}{2}}\log{m}\right)
\end{align*}
as $m\longrightarrow\infty$.
\end{thrm} 
This result generalizes \eqref{asymp1} to the finite volume case.
The constant $C(n)$ in Theorem \ref{theo1} equals the constant $C(n)$ occurring
in \eqref{asymp1} and can be computed explicitly from the Plancherel
polynomials. It equals
\begin{equation}\label{cn}
C(n)=\frac{(-1)^{n-1} 2^{(n+1)n/2}n!}{2\pi^n\prod_{0\le i<j\le n}(j+i)}.
\end{equation}
We also consider the $L^2$-torsion $T^{(2)}_X(\tau)$. Although $X$ is
noncompact,
it can be defined as in the compact case \cite{Lo}. It can be computed using
the results of \cite{MP}. First of all, we show that there  exists a
polynomial $P_\tau(m)$ of degree $n(n+1)/2+1$ such that
\begin{equation}\label{l2tor3}
\log T^{(2)}_X(\tau(m))=\vol(X)P_\tau(m).
\end{equation}
The polynomial is obtained from the Plancherel polynomials. Its leading term
can be determined as in \cite{MP} and we obtain
\begin{equation}\label{l2tor5}
\log{T_X^{(2)}(\tau(m))}=-C(n)\vol(X)m\cdot\dim(\tau(m))+O(m^{\frac{n(n+1)}{2}}
).
\end{equation}
Compared with Theorem \ref{theo1} we obtain the following Theorem.
\begin{thrm}\label{Theorem2}
Let $X=\Gamma\bs\mathbb{H}^{2n+1}$ be a $(2n+1)$-dimensional complete, oriented,
hyperbolic manifold of finite volume. Assume that $\Gamma$ satisfies \eqref{a1}.
Then we have
\begin{align*}
\log{T_{X}(\tau(m))}=\log{T_X^{(2)}(\tau(m))}+O(m^{\frac{n(n+1)}{2}}\log{m})
\end{align*}
as $m\to\infty$. 
\end{thrm}
Next we turn to the even-dimensional case. First
recall that for a compact manifold of even dimension, the analytic torsion is 
always equal to 1 (see \cite{RS}, \cite[Proposition 1.7]{MP}). This is
not true anymore in the noncompact case. Park \cite[Theorem 1.4]{Pa}
has computed the analytic torsion of a unitary representation of $\Gamma$ in
even dimensions. His formula shows that in the noncompact case, 
the analytic torsion in even dimensions is not trivial in general.
Nevertheless, the torsion has still 
a rather simple behavior as shown by the next proposition. For a hyperbolic 
manifold of finite volume $X$, denote by
$\kappa(X)$ the number of cusps of $X$. Let $\hL$ be the standard Cartan 
subalgebra of $\gL$ and let $\Lambda(G)\subset \hL_\C^*$ be the highest 
weight lattice. For $\lambda\in\Lambda(G)$ let $\tau_\lambda$ be the 
corresponding irreducible representation of $G$.
\begin{prop}\label{toreven}
There exists a function $\Phi\colon \Lambda(G)\to \R$ such that for every
even-dimensional complete oriented hyperbolic manifold $X$ of finite
volume one has
\[
\log T_X(\tau_\lambda)=\kappa(X)\Phi(\lambda),\quad \lambda\in\Lambda(G).
\]
\end{prop} 
The function $\Phi$ can be described as follows. There is a
distribution $J$ which appears on the geometric side of the trace formula.
It is of the form $J=\kappa(X)\cdot \tilde J$, where $\tilde J$ is
defined in terms of weighted characters of principal series representations 
of $G$ (see \eqref{Definition von J}).
Let $k_t^\tau\in\mathcal{C}(G)$ be the function \eqref{ktau6}. 
There is
$c>0$ such that $\tilde J(k_t^\tau)=O(e^{-ct})$ as $t\to\infty$. Moreover 
$\tilde J(k_t^\tau)$
has an asymptotic expansion as $t\to 0$. Thus the Mellin transform 
$\mathcal{M}\tilde{J}(s;\tau)$ of $\tilde J(k_t^\tau)$ is defined for 
$\Re(s)\gg0$ and
admits a meromorphic extension to $\C$ which is regular at $s=0$. Then we have
\[
\Phi(\lambda)=\mathcal{M}\tilde{J}(0;\tau_\lambda)
\]
for all highest weights $\lambda=(k_1,\dots,k_{n+1})$.

Next recall that for a compact manifold $X$, the analytic torsion equals 
the Reidemeister torsion (see \cite{Mu1}). This is the basis 
for the applications of the results of \cite{Mu2} to the cohomology of
arithmetic hyperbolic 3-manifolds in \cite{MaM}. Currently it is not known
if there is an extension of the equality of analytic and Reidemeister torsion
to the noncompact setting. This is an interesting problem and the present 
paper is a first step in this direction. 

We shall now outline our method for the proof of our main result. Let $d=2n+1$.
We assume
that the highest weight of $\tau$ satisfies $\tau_{n+1}\neq0$. Let
\begin{align*}
K(t,\tau):=\sum_{p=0}^{2n+1}(-1)^p
p\Tr_{\rel}(e^{-t\Delta_p(\tau)}).
\end{align*}
By \eqref{speczeta} and \eqref{def-tor3} we need to 
compute the finite part of the Mellin transform of 
$K(t,\tau)$ at 0. Let $\widetilde E_\tau$ be the homogeneous vector bundle over
$\widetilde X=G/K$ associated to $\tau$ and let $\widetilde{\Delta}_p(\tau)$ be
the Laplacian on
$\widetilde E_\tau$-valued $p$-forms on $\widetilde X$. The heat operator
$e^{-t\widetilde{\Delta}_p(\tau)}$ is a convolution operator with kernel
$H_t^{\nu_p(\tau)}\colon G\to \End(\Lambda^p\pL^*\otimes V_\tau)$. Let
$h_t^{\nu_p(\tau)}(g)=\tr H_t^{\nu_p(\tau)}(g)$, $g\in G$, and put
\begin{equation}\label{ktau6}
k_t^\tau=\sum_{p=1}^d (-1)^p p h_t^{\nu_p(\tau)}.
\end{equation}
Let $R_\Gamma$ be the right regular representation
of $G$ on $L^2(\Gamma\backslash G)$. There exists an orthogonal
$R_\Gamma$-invariant
decomposition $L^2(\Gamma\backslash G)=L^2_d(\Gamma\backslash G)\oplus
L^2_c(\Gamma\backslash
G)$. The restriction $R_{\Gamma}^d$ of $R_\Gamma$  to
$L^2_d(\Gamma\backslash G)$ decomposes into
the orthogonal
direct sum of irreducible unitary representations, each of which occurs with
finite multiplicity.
On the other hand, by the theory of Eisenstein series, the restriction 
$R_\Gamma^c$ of $R_\Gamma$ to
$L^2_c(\Gamma\backslash G)$ is isomorphic to the direct integral over all
tempered principle series
representations of $G$. For $\phi\in L^2_d(\Gamma\backslash G)$ let
\begin{align*}
\left(R_{\Gamma}^d(k_{t}^{\tau})\phi\right)(x):=\int_{G}k_{t}^{\tau}
(g)\phi(xg)dg.
\end{align*}
Then $R_{\Gamma}^d(k_{t}^{\tau})$ is a trace class operator and the Selberg
trace formula computes its trace. The right hand side of the trace formula
is the sum of terms associated to the continuous spectrum and orbital
integrals associated to the various conjugacy classes of $\Gamma$. If we move
the spectral terms to the left hand side of the trace formula we end up
with the spectral side $J_{\spec}(k_t^\tau)$ of the trace formula. The key fact
is now that
\[
K(t,\tau)=J_{\spec}(k_t^\tau).
\] 
By the Selberg trace formula, the spectral side equals the geometric side, 
that is, the sum of the orbital integrals. This leads to the following 
fundamental equality:
\begin{equation}\label{trace7}
K(t,\tau)=I(t;\tau)+H(t;\tau)+T(t;\tau)+\mathcal{I}(t;\tau)+J(t;\tau),
\end{equation}
where $I(t;\tau)$ is the contribution of the identity conjugacy class of
$\Gamma$ and $H(t;\tau)$ is the contribution of the hyperbolic conjugacy
classes of
$\Gamma$.
Moreover, $T(t;\tau)$, $\mathcal{I}(t;\tau)$ and $J(t;\tau)$ are
tempered
distributions applied to $k_t^{\tau}$ which
are constructed out of the parabolic conjugacy classes of $\Gamma$. 
Now we evaluate the Mellin transform of each term separately. Here an important
simplification is obtained using a theorem of Kostant on Lie algebra cohomology.

Let $\mathcal{M}I(\tau)$ be the Mellin transform of $I(t;\tau)$ evaluated
at $0$. Then we show that
\[
\log T^{(2)}_X(\tau)=\frac{1}{2}\mathcal{M}I(\tau).
\]
Now consider the representations $\tau(m)$, $m\in\N$. 
Using the results of \cite{MP} we compute $\mathcal{M}I(\tau(m))$ and
prove \eqref{l2tor3} and \eqref{l2tor5}. Thus in order to prove our main
result, we need to show that the Mellin transforms at 0 of all other terms 
are of lower order. It is easy to treat the hyperbolic term and the terms
$T(t;\tau(m))$.  The distribution
$\mathcal{I}(t;\tau(m))$
is invariant and its Fourier transform was computed explicitly by Hoffmann 
\cite{Hoffmann}. Using his results we can estimate the Mellin transform of 
$\mathcal{I}(t;\tau(m))$ at 0.
Finally, the distribution $J(t;\tau(m))$ is non-invariant. However it is 
described in terms of 
Knapp-Stein intertwining operators which are understood completely
in our case. With this information its Mellin transform at 0 can also be
estimated.

In \cite{MP} we have used a different method which does not rely on the
trace formula. It would be interesting to generalize this method to the
finite volume case. Especially the Fourier transform, which we use to deal
with $\mathcal{I}(t;\tau(m))$, is a very heavy machinery and is not
available in the higher rank case. Part of the arguments used in \cite{MP} go
through in the finite volume case as well. The difficult part is to deal with 
the contribution of the parabolic terms. 

This paper is organized as follows. In section \ref{Notations}
we fix notations and collect some basic facts. In section
\ref{secrep} we review some properties of the right regular 
representation of $G$ on $L^2(\Gamma\backslash G)$. In section \ref{secBLO} we
introduce the locally invariant differential operators which act on
locally homogeneous vector bundles over $X$. Section \ref{secrel} is
devoted to the regularized trace which we introduce there and relate it to the 
spectral side of the Selberg trace formula. In section \ref{sectr} we
apply the Selberg trace formula which leads to \eqref{trace7}. Furthermore, we 
study the Fourier transform of the distribution $\mathcal{I}$. Finally we 
derive an asymptotic expansion as $t\to 0$ for the regularized trace of the 
heat operator of a Bochner-Laplace operator. In \ref{sectors} we introduce the
analytic torsion. In section \ref{secvir} we express the test function
$k_t^\tau$
a as combination of functions defined by the heat kernels of certain 
Bochner-Laplace operators. The results of this section are needed to deal with
the Mellin transforms of the various terms on the right hand side of 
\eqref{trace7}. In section \ref{secl2} we study the $L^2$-torsion. 
In the final section \ref{Beweis} we prove the main results.

This paper arose from the PhD-Thesis of the second author under the
supervision of the first author.

\section{Preliminaries}\label{Notations}
\setcounter{equation}{0}

In this section we will establish some notation and recall some basic facts 
about representations of the involved Lie groups.
For $d\in\N$, $d>1$ let  $G:=\Spin(d,1)$. 
Recall that $G$ is the universal covering group 
of $\SO_{0}(d,1)$. Let
$K:=\Spin(d)$. Then $K$ is a maximal compact subgroup of $G$. Put
$\tilde{X}:=G/K$. Let 
\[
G=NAK
\]
 be the standard Iwasawa decomposition of $G$ and let $M$ be the 
centralizer of $A$ in $G$. Then $M=\Spin(d-1)$. 
The Lie algebras of  $G,K,A,M$ and $N$ will be denoted by
$\gL,\kL,\aL,\mL$ and $\nf$, respectively. Define the 
standard Cartan involution $\theta:\gL\rightarrow \gL$ by
\begin{align*}
\theta(Y)=-Y^{t},\quad Y\in\gL. 
\end{align*}
The lift of $\theta$ to $G$ will be denoted by the same letter $\theta$. Let 
\begin{align*}
\mathfrak{g}=\mathfrak{k}\oplus\mathfrak{p}
\end{align*}
be the Cartan decomposition of $\gL$ with respect to $\theta$. Let $x_0=eK\in
\tilde X$. Then we have a canonical isomorphism
\begin{equation}\label{tspace}
T_{x_0}\tilde X\cong \pL.
\end{equation}
Define the symmetric bi-linear form $\langle\cdot,\cdot\rangle$ on $\gL$ by
\begin{equation}\label{killnorm}
\langle Y_1,Y_2\rangle:=\frac{1}{2(d-1)}B(Y_1,Y_2),\quad Y_1,Y_2\in\gL.
\end{equation}
By \eqref{tspace} the restriction of $\langle\cdot,\cdot\rangle$ to $\pL$ 
defines an inner product on $T_{x_0}\tilde X$ and therefore an invariant 
metric on $\tilde X$. This metric has constant curvature $-1$. Then $\tilde X$,
equipped with this metric, is isometric to the hyperbolic space
${\mathbb{H}}^{d}$.

\subsection{}\label{Wurzeln}

Let $\aL$ be the Lie-Algebra of $A$.
Fix a Cartan-subalgebra $\mathfrak{b}$ of $\mathfrak{m}$. 
Then
\begin{align*}
\mathfrak{h}:=\mathfrak{a}\oplus\mathfrak{b}
\end{align*}
is a Cartan-subalgebra of $\mathfrak{g}$. We can identify
$\mathfrak{g}_\C\cong\mathfrak{so}(d+1,\C)$. Let $e_1\in\aL^*$ be the
positive restricted root defining $\mathfrak{n}$.
Then for $d=2n+1$, or $d=2n+2$, we fix $e_2,\dots,e_{n+1}\in
i\mathfrak{b}^*$ such that 
the positive roots $\Delta^+(\mathfrak{g}_\C,\mathfrak{h}_\C)$ are chosen as in
\cite[page 684-685]{Knapp2}
for the root system $D_{n+1}$ resp. $B_{n+1}$. We let
$\Delta^+(\mathfrak{g}_\C,\mathfrak{a}_\C)$ be
the set of roots of $\Delta^+(\mathfrak{g}_\C,\mathfrak{h}_\C)$ which do not
vanish on $\aL_\C$. The positive roots
$\Delta^+(\mathfrak{m}_\C,\mathfrak{b}_\C)$
are chosen such that they are restrictions of elements from
$\Delta^+(\mathfrak{g}_\C,\mathfrak{h}_\C)$.
For $i=1,\dots,n+1$ we let $H_i\in\mathfrak{h}_{\C}$
be such that $e_j(H_i)=\delta_{i,j}$, $j=1,\dots,n+1$.
For $\alpha\in\Delta^{+}(\mathfrak{g}_{\C},\mathfrak{h}_{\C})$ 
there exists a unique $H'_{\alpha}\in\mathfrak{h}_{\C}$ 
such that $B(H,H_{\alpha}^{'})=\alpha(H)$ for all $H\in\mathfrak{h}_{\C}$. 
One has $\alpha(H_{\alpha}^{'})\neq 0$. We let
\begin{align*}
H_{\alpha}:=\frac{2}{\alpha(H_{\alpha}^{'})}H_{\alpha}^{'}.
\end{align*}
One easily sees that 
\begin{align}\label{Explizite Formel fur H(alpha)}
H_{\pm e_{i}\pm e_{j}}=\pm H_{i}\pm H_{j}.
\end{align}
For $j=1,\dots,n+1$ let
\begin{equation}\label{rho}
\rho_{j}:=\begin{cases}n+1-j, & G=\Spin(2n+1,1);\\
n+3/2-j, & G=\Spin(2n+2,1).
\end{cases}.
\end{equation}
Then the half-sum of positive roots $\rho_G$ and $\rho_M$, respectively, are
given by
\begin{align}\label{Definition von rho(G)}
\rho_{G}:=\frac{1}{2}\sum_{\alpha\in\Delta^{+}(\mathfrak{g}_{\mathbb{C}},
\mathfrak{h}_\mathbb{C})}\alpha=\sum_{j=1}^{n+1}\rho_{j}e_{j}
\end{align}
and
\begin{align}\label{Definition von rho(M)}
\rho_{M}:=\frac{1}{2}\sum_{\alpha\in\Delta^{+}(\mathfrak{m}_{\mathbb{C}},
\mathfrak{b}_{\mathbb{C}})}\alpha=\sum_{j=2}^{n+1}\rho_{j}e_{j}.
\end{align}
Let $W_{G}$ be the Weyl-group of $\Delta(\mathfrak{g}_{\C},
\mathfrak{h}_{\mathbb{C}})$. 

\subsection{}

Let ${{\mathbb{Z}}\left[\frac{1}{2}\right]}^{j}$ be the set of all 
$(k_{1},\dots,k_{j})\in\mathbb{Q}^{j}$ such that either all $k_{i}$ are 
integers or all $k_{i}$ are half integers. 
Then the  finite dimensional irreducible representations 
$\tau\in\hat{G}$ of $G$ 
are parametrized by their highest weights
\begin{equation}\label{Darstellungen von G}
\Lambda(\tau)=k_{1}(\tau)e_{1}+\dots+k_{n+1}(\tau)e_{n+1};\:\:
k_{1}(\tau)\geq 
k_{2}(\tau)\geq\dots\geq k_{n}(\tau)\geq \left|k_{n+1}(\tau)\right|,
\end{equation}
if $G=\Spin(2n+1,1)$ resp.
\begin{equation}\label{Darstellungen von G'}
\Lambda(\tau)=k_{1}(\tau)e_{1}+\dots+k_{n+1}(\tau)e_{n+1};\:\:
k_{1}(\tau)\geq 
k_{2}(\tau)\geq\dots\geq k_{n}(\tau)\geq k_{n+1}(\tau)\geq 0,
\end{equation}
if $G=\Spin(2n+2,1)$,
where $(k_{1}(\tau),\dots, k_{n+1}(\tau))\in
{{\mathbb{Z}}\left[\frac{1}{2}\right]}^{n+1}$.\\
Moreover, the finite dimensional representations $\nu\in\hat{K}$ of $K$ are
parametrized by their highest weights
\begin{equation}\label{Darstellungen von K}
\Lambda(\nu)=k_{2}(\nu)e_{2}+\dots+k_{n+1}(\nu)e_{n+1};\:\:
k_{2}(\nu)\geq 
k_{3}(\nu)\geq\dots\geq k_{n}(\nu)\geq k_{n+1}(\nu)\geq 0,
\end{equation} 
if $G=\Spin(2n+1,1)$ resp.
\begin{equation}\label{Darstellungen von K'}
\Lambda(\nu)=k_{1}(\nu)e_{1}+\dots+k_{n+1}(\nu)e_{n+1};\:\:
k_{1}(\nu)\geq 
k_{2}(\nu)\geq\dots\geq k_{n}(\nu)\geq \left|k_{n+1}(\nu)\right|,
\end{equation}
if $G=\Spin(2n+2,1)$,
where $(k_{2}(\nu),\dots, k_{n+1}(\nu)), (k_{1}(\nu),\dots,
k_{n+1}(\nu))\in
{{\mathbb{Z}}\left[\frac{1}{2}\right]}^{n,n+1}$.\\
Finally, the  finite dimensional irreducible representations 
$\sigma\in\hat{M}$ of $M$ 
are parametrized by their highest weights
\begin{equation}\label{Darstellungen von M}
\Lambda(\sigma)=k_{2}(\sigma)e_{2}+\dots+k_{n+1}(\sigma)e_{n+1};\:\:
k_{2}(\sigma)\geq 
k_{3}(\sigma)\geq\dots\geq k_{n}(\sigma)\geq \left|k_{n+1}(\sigma)\right|,
\end{equation}
if $G=\Spin(2n+1,1)$ resp.
\begin{equation}\label{Darstellungen von M'}
\Lambda(\sigma)=k_{2}(\sigma)e_{1}+\dots+k_{n+1}(\sigma)e_{n+1};\:\:
k_{2}(\sigma)\geq\dots\geq k_{n}(\sigma)\geq k_{n+1}(\sigma)\geq 0,
\end{equation}
if $G=\Spin(2n+2,1)$,
where $(k_{2}(\sigma),\dots, k_{n+1}(\sigma))\in
{{\mathbb{Z}}\left[\frac{1}{2}\right]}^{n}$.

\subsection{}
Let $d=2n+1$. For $\tau\in\hat{G}$ let $\tau_{\theta}:=\tau\circ\theta$. Let
$\Lambda(\tau)$ 
denote the highest weight of $\tau$ as in \eqref{Darstellungen von G}.
Then the highest weight $\Lambda(\tau_\theta)$ of $\tau_\theta$ is given by
\begin{equation}\label{Tau theta}
\Lambda(\tau_{\theta})=k_{1}(\tau)e_{1}+\dots+k_{n}(\tau)e_{n}-k_{n+1}(\tau)e_{
n+1}.
\end{equation}
Let $\sigma\in\hat{M}$ with highest weight 
$\Lambda(\sigma)\in\mathfrak{b}_{\C}^{*}$ as in 
\eqref{Darstellungen von M}. By the Weyl dimension formula 
\cite[Theorem 4.48]{Knapp} we have
\begin{equation}\label{Dimension sigma}
\begin{split}
\dim(\sigma)=&\prod_{\alpha\in\Delta^{+}(\mathfrak{m}_{\mathbb{C}},
\mathfrak{b}_{\C})}\frac{\left<\Lambda(\sigma)+\rho_{M},\alpha\right>}
{\left<\rho_{M},\alpha\right>}\\
=&\prod_{i=2}^{n}\prod_{j=i+1}^{n+1}\frac{\left(k_{i}(\sigma)
+\rho_{i}\right)^{2}-\left(k_{j}(\sigma)+\rho_{j}\right)^{2}}
{\rho_{i}^{2}-\rho_{j}^{2}}.
\end{split}
\end{equation}

\subsection{}
Let $M'$ be the normalizer of $A$ in $K$ and let $W(A)=M'/M$ be the 
restricted Weyl-group. It has order two and it acts on the finite-dimensional 
representations of $M$ as follows. Let $w_{0}\in W(A)$ be the non-trivial 
element and let $m_0\in M^\prime$ be a representative of $w_0$. Given 
$\sigma\in\hat M$, the representation $w_0\sigma\in \hat M$ is defined by
\[
w_0\sigma(m)=\sigma(m_0mm_0^{-1}),\quad m\in M.
\] 
If $d=2n+2$ one has $w_0\sigma\cong\sigma$ for every $\sigma\in\hat{M}$. 
Assume that $d=2n+1$. Let
$\Lambda(\sigma)=k_{2}(\sigma)e_{2}+\dots+k_{n+1}(\sigma)e_{n+1}$ be the 
highest weight
of $\sigma$ as in \eqref{Darstellungen von M}. Then the highest weight 
$\Lambda(w_0\sigma)$ of $w_0\sigma$ is given by
\begin{equation}\label{wsigma}
\Lambda(w_0\sigma)=k_{2}(\sigma)e_{2}+\dots+k_{n}(\sigma)e_{n}
-k_{n+1}(\sigma)e_{n+1}.
\end{equation}

\subsection{}
Let $d=2n+1$. 
Let $R(K)$ and $R(M)$ be the representation rings of $K$ and $M$. Let 
$\iota:M\longrightarrow K$ be the inclusion and let $\iota^{*}:R(K)
\longrightarrow R(M)$ be the induced map. If $R(M)^{W(A)}$ is the subring of 
$W(A)$-invariant elements of $R(M)$, then clearly $\iota^{*}$ maps $R(K)$ into 
$R(M)^{W(A)}$. 
\begin{prop}\label{branching}
The map $\iota$ is an isomorphism from $R(K)$ onto $R(M)^{W(A)}$.
Explicitly, let $\sigma\in\hat{M}$ be of highest weight $\Lambda(\sigma)$ as in
\eqref{Darstellungen von M} and assume that $k_{n+1}(\sigma)\geq 0$. Then if
$\nu(\sigma)\in
R(K)$ is such that
\begin{align*}
\iota^*\nu(\sigma)=
\begin{cases} \sigma & \sigma=w_0\sigma\\
\sigma+w_{0}\sigma & \sigma\neq w_0\sigma
\end{cases} 
\end{align*}
one has
\begin{align}\label{BrKM}
\nu(\sigma)=\sum_{\mu\in\{0,1\}^n}(-1)^{c(\mu)}
\nu\left(\Lambda(\sigma)-\mu\right),
\end{align}
where the sum runs over all $\mu\in\{0,1\}^n$ such that $\Lambda(\sigma)-\mu$
is the highest weight of an irreducible representation 
$\nu\left(\Lambda(\sigma)-\mu\right)$ of $K$ and
$c(\mu):=\#\{1\in\mu\}$. 
\end{prop}
\begin{proof} See
\cite[Proposition 1.1]{Bunke}.
\end{proof}
Let $\sigma\in\hat M$ and assume that $\sigma\neq w_0\sigma$. Then by 
Proposition \ref{branching} there exist unique integers $m_\nu(\sigma)\in
\{-1,0,1\}$, which are zero except for finitely many $\nu\in\hat K$, such that
\begin{equation}\label{multipl1}
\sigma+w_0\sigma=\sum_{\nu\in\hat K}m_\nu(\sigma)i^*(\nu).
\end{equation}

\subsection{}

Measures are normalized as follows.  
Every $a\in A$ can be written 
as $a=\exp{\log{a}}$, where $\log{a}\in\mathfrak{a}$ is unique.
For $t\in\mathbb{R}$, we let 
$a(t):=\exp{(tH_{1})}$. If $g\in G$, we define $n(g)\in N$, $
H(g)\in \R$ and $\kappa(g)\in K$ by 
\begin{align*}
g=n(g)\exp{\left(H(g)e_1\right)}\kappa(g). 
\end{align*}
Normalize the Haar-measure on $K$ such that $K$ has volume 1. We let
\begin{align}\label{Def der Metrik}
 \left<X,Y\right>_{\theta}:=-\frac{1}{2(d-1)}B(X,\theta(Y)).
\end{align}
We fix an isometric identification of  $\mathbb{R}^{d-1}$ with $\mathfrak{n}$ 
with respect to the inner product $\left<\cdot,\cdot\right>_{\theta}$. We give 
$\mathfrak{n}$ the measure induced from the Lebesgue measure under this 
identification. Moreover, we identify $\mathfrak{n}$ and $N$ by the 
exponential map and we will denote by $dn$ the Haar measure on $N$ induced 
from the measure on $\mathfrak{n}$ under this identification. We normalize the 
Haar measure on $G$ by setting
\begin{align}\label{Haarmass auf G}
\int_{G}{f(g)dg}
&=\int_{N}\int_{\mathbb{R}}\int_{K}{e^{-(d-1)t}f(na(t)k)dkdtdn}.
\end{align}
The spaces $\tilde{X}$ and $\Gamma\backslash G$,
$\Gamma$ a discrete subgroup, will be equipped with the induced
quotient-measure.

\subsection{}

We parametrize the principal series as follows. Given $\sigma\in\hat{M}$ with
$(\sigma,V_\sigma) \in \sigma$, let $\mathcal{H}^{\sigma}$ denote the space of
measurable functions $f\colon K\to V_\sigma$ satisfying
\[
f(mk)=\sigma(m)f(k),\quad\forall k\in K,\, \forall m\in M,
\quad\textup{and}\quad
\int_K\parallel f(k)\parallel^2\,dk=\parallel f\parallel^2<\infty.
\]
Then for $\lambda\in\mathbb{C}$ and $f\in H^{\sigma}$ let
\begin{align*}
\pi_{\sigma,\lambda}(g)f(k):=e^{(i\lambda +(d-1)/2)H(kg)}f(\kappa(kg)).
\end{align*}
Recall that the representations $\pi_{\sigma,\lambda}$ are unitary iff 
$\lambda\in\mathbb{R}$. Moreover, for $\lambda\in\mathbb{R}-\{0\}$ and 
$\sigma\in\hat{M}$ the representations $\pi_{\sigma,\lambda}$ are irreducible 
and $\pi_{\sigma,\lambda}$ and $\pi_{\sigma',\lambda'}$, $\lambda,
\lambda'\in\mathbb{C}$ are 
equivalent iff either $\sigma=\sigma'$, $\lambda=\lambda'$ or
$\sigma'=w_{0}\sigma$, 
$\lambda'=-\lambda$.
The restriction of $\pi_{\sigma,\lambda}$ to $K$ coincides with the induced 
representation ${\rm{Ind}}_{M}^{K}(\sigma)$. Hence by Frobenius 
reciprocity \cite[p.208]{Knapp} for every $\nu\in\hat{K}$ one has
\begin{align}\label{Frobeniusrez}
\left[\pi_{\sigma,\lambda}:\nu\right]=\left[\nu:\sigma\right].
\end{align}
\subsection{}\label{secPL}
Assume that $d=2n+1$. For $\sigma\in\hat{M}$ and $\lambda\in\mathbb{R}$ let
$\mu_{\sigma}(\lambda)$ 
be the Plancherel measure associated to $\pi_{\sigma,\lambda}$. Then, since 
${\rm{rk}}(G)>{\rm{rk}}(K)$, $\mu_{\sigma}(\lambda)$ is a polynomial in 
$\lambda$ of degree $2n$. Let $\left<\cdot,\cdot\right>$ be the bi-linear form
defined by 
\eqref{killnorm}. Let $\Lambda(\sigma)\in\mathfrak{b}_{\mathbb{C}}^{*}$ be 
the highest weight of $\sigma$ as in \eqref{Darstellungen von M}.
Then by theorem 13.2 in \cite{Knapp} there exists a constant $c(n)$ such that 
one has
\begin{align*}
\mu_{\sigma}(\lambda)=-c(n)\prod_{\alpha\in\Delta^{+}(\mathfrak{g}_{\mathbb{C}},
\mathfrak{h}_{\mathbb{C}})}\frac{\left<i\lambda e_{1}+\Lambda(\sigma)+\rho_{M},
\alpha\right>}{\left<\rho_{G},\alpha\right>}..
\end{align*}
The constant $c(n)$ is computed in \cite{Mi2}. By \cite{Mi2}, theorem 3.1, one 
has $c(n)>0$. 
For $z\in\mathbb{C}$  let
\begin{align}\label{plancherelmass}
P_{\sigma}(z)=-c(n)\prod_{\alpha\in\Delta^{+}(\mathfrak{g}_{\mathbb{C}},
\mathfrak{h}_{\mathbb{C}})}\frac{\left<z e_{1}+\Lambda(\sigma)+\rho_{M},
\alpha\right>}{\left<\rho_{G},\alpha\right>}.
\end{align}
One easily sees that
\begin{align}\label{P-Polynom is W-inv}
P_{\sigma}(z)=&P_{w_{0}\sigma}(z).
\end{align}

\section{The decomposition of the right regular representation}\label{secrep}
\setcounter{equation}{0}

Let $\Gamma$ be a discrete, torsion free subgroup of $G$ with
$\vol(\Gamma\backslash G)<\infty$.  Let $\mathfrak{P}$ be a fixed
set of representatives of $\Gamma$-nonequivalent proper cuspidal parabolic 
subgroups of
$G$. Then $\mathfrak{P}$ is finite. Let $\kappa:=\#\mathfrak{P}$. Without loss 
of generality we will assume that $P_{0}:=MAN\in\mathfrak{P}$. 
For every $P\in\mathfrak{P}$, there exists a $k_{P}\in K$ such that
\[
P=N_{P}A_{P}M_{P}
\]
with $N_{P}=k_{P}Nk_{P}^{-1}$, $A_{P}=k_{P}Ak_{P}^{-1}$, and
$M_{P}=k_{P}Mk_{P}^{-1}$.
We let $k_{P_{0}}=1$. We will assume that for each
$P\in\mathfrak{P}$ one has 
\begin{equation}\label{assumpt}
\Gamma\cap P=\Gamma\cap N_{P}.
\end{equation}
Since $N_{P}$ is
abelian, we have $\Gamma\cap N_{P}\backslash N_P\cong T^{d-1}$, where $T^{d-1}$
is 
the flat $(d-1)$-torus. For
$P\in\mathfrak{P}$ let $a_{P}(t):=k_{P}a(t)k_{P}^{-1}$. If $g\in G$, we define
$n_{P}(g)\in N_{P}$, $H_{P}(g)\in
\R$ and $\kappa_{P}(g)\in K$ by 
\begin{align}\label{coord}
g=n_{P}(g)a_P(H_{P}(g))k_P^{-1}\kappa_{P}(g). 
\end{align}
For each $P\in\mathfrak{P}$ define 
\[
\iota_P\colon \R^+\to A_P
\]
by $\iota_P(t):=a_P(\log(t))$. For $Y>0$, let 
\[
A^{0}_{P}\left[Y\right]:=\iota_P(Y,\infty).
\]
Then there exists a $Y_{0}>0$ and for every $Y\geq Y_{0}$ a compact connected
subset $C(Y)$ of $G$  such that in the sense of a disjoint union one has
\begin{align}\label{Zerlegung des FB}
G=\Gamma\cdot C(Y)\sqcup\bigsqcup_{P\in\mathfrak{P}}\Gamma\cdot
N_{P}A^{0}_{P}\left[Y\right]K
\end{align}
and such that 
\begin{align}\label{Eigenschaft des FB}
\gamma\cdot N_{P}A^0_{P}\left[Y\right]K\cap N_{P}A_{P}^{0}\left[Y\right]K\neq
\emptyset\Leftrightarrow\gamma\in \Gamma_{N}. 
\end{align}
If for $Y\geq Y_{0}$ one lets 
\begin{align}\label{Definition der Spitze}
F_{P,Y}:=A_{P}\left[Y\right]\times\Gamma\cap N_{P}\backslash N_{P}\cong
[Y,\infty)\times
\Gamma\cap N_{P}\backslash N_{P},
\end{align}
it follows from \eqref{Zerlegung des FB} and \eqref{Eigenschaft des FB} that 
there exists a compact manifold $X(Y)$ with smooth boundary such that $X$
has a decomposition as 
\begin{align}\label{Zerlegung X}
X=X(Y)\cup \bigsqcup_{P\in\mathfrak{P}}F_{P,Y}
\end{align}
with
$X(Y)\cap F_{P,Y}=\partial X(Y)=\partial F_{P,Y}$ and $F_{P,Y}\cap
F_{P',Y}=\emptyset$ if $P\neq P'$. \\
Let $R_{\Gamma}$ be the right-regular representation of $G$ on
$L^{2}(\Gamma\backslash G)$. We shall now describe some basic properties
of $R_{\Gamma}$. The main references are \cite{Langlands},
\cite{Harish-Chandra} \cite{Warner}. There exists an orthogonal
decomposition 
\begin{align}\label{Zerlegung von L2}
L^{2}(\Gamma\backslash G)=L^{2}_{d}(\Gamma\backslash G)\oplus
L^{2}_{c}(\Gamma\backslash G)
\end{align}
of $L^{2}(\Gamma\backslash G)$ into closed $R_{\Gamma}$-invariant subspaces.
The restriction of $R_{\Gamma}$ to $L^{2}_{d}(\Gamma\backslash G)$ decomposes
into the orthogonal direct sum of irreducible unitary representations of $G$ and
the multiplicity of each irreducible unitary representation of $G$ in this
decomposition is finite. 
On the other hand, by the theory of Eisenstein series, the restriction 
$R_\Gamma^c$ of
$R_{\Gamma}$ to $L^{2}_{c}(\Gamma\backslash G)$ is isomorphic to the direct
integral over all unitary principle-series representations of $G$. 

Next we recall the definition and some of the basic properties of the 
Eisenstein series.  For $P=M_PA_pN_P\in\fP$ let 
\[
\cE_{P}=L^2((\Gamma\cap P)N_PA_P\bs G).
\]
For each $\lambda\in\C$ there is a representation $\pi_{P,\lambda}$ of
$G$
on $\cE_P$, defined by
\[
(\pi_{P,\lambda}(y)\Phi)(x)=e^{(\lambda +(d-1)/2)(H_P(xy))}e^{
-(\lambda+(d-1)/2)( H_P(x))}
\Phi(xy).
\]
Given $\Phi\in\cE_P$ and $\lambda\in\C$, put
\[
\Phi_\lambda(x)=e^{(\lambda +(d-1)/2)H_P(x)}\Phi(x).
\]
The action of the representation $\pi_{P,\lambda}$ is then given by
\[
(\pi_{P,\lambda}(y)\Phi)_\lambda(x)=\Phi_\lambda(xy).
\] 
and  $\pi_{P,\lambda}$ is unitary for $\lambda\in i\R$. Let $\cE^0_P$ be
the
subspace of $\cE_P$ consisting of of all right $K$-finite and left
$\mathfrak{Z}_M$-finite functions, where $\mathfrak{Z}_M$ denotes the center 
of the
universal enveloping algebra of $\mL_\C$. For $\Phi\in\cE_P^0$ and 
$\lambda\in\C$ the Eisenstein series $E(P,\phi,\lambda,x)$ is defined 
by
\[
E(P,\Phi,\lambda,x)=\sum_{\gamma\in\Gamma\cap P\bs \Gamma}\Phi_\lambda(\gamma
x).
\]
It converges absolutely and uniformly on compact subsets of 
$\{\lambda\in\C\colon \Re(\lambda)>(d-1)/2\}\times G$, and it has a 
meromorphic
extension to $\C$. Let $P^\prime\in\fP$. The constant term 
$E_{P^\prime}(P,\Phi,\lambda)$ of $E(P,\Phi,\lambda)$ along $P^\prime$ is
defined 
by
\begin{equation}\label{c-term1}
E_{P^\prime}(P,\Phi,\lambda,x):=\frac{1}{\vol(\Gamma\cap N_{P^\prime}\bs
N_{P^\prime})}
\int_{\Gamma\cap N_{P^\prime}\bs N_{P^\prime}}E(P,\Phi,\lambda,n^\prime
x)\;dn^\prime.
\end{equation}
Let $W(A_P,A_{P^\prime})$ be the set of all bijections $w\colon A_P
\to A_{P^\prime}$ for which there exists $x\in G$ such that $w(a)=xax^{-1}$,
$a\in A_P$. Then one can identify $W(A_P,A_{P^\prime})$ with
$k_{P'}W(A)k_P^{-1}$.
Thus $W(A_P,A_{P^\prime})$ has order 2. We let $W(A_P,A_{P^\prime})$ act on $\C$
as follows. For
$w=k_{P'}k_{P}^{-1}$ and $\lambda\in\C$ we put $w\lambda:=\lambda$. Let $w_0$ be the non-trivial element of
$W(A)$. Then for $w=k_{P'}w_0k_{P}^{-1}$ and $\lambda\in\C$ we put $w\lambda:=-\lambda$. 
Then one has
\begin{equation}\label{c-term2}
E_{P^\prime}(P,\Phi,\lambda,x)=\sum_{w\in W(A_P,A_{P^\prime})}
e^{(w\lambda+(d-1)/2)(H_{P^\prime}(x))}
\left(c_{P^\prime|P}(w\colon\lambda)\Phi\right)(x),
\end{equation}
where 
\[
c_{P^\prime|P}(w\colon\lambda)\colon \cE_P\to \cE_{P^\prime}
\]
are linear maps which are meromorphic functions of $\lambda\in\C$. 
Put
\[
\boldsymbol{\cE}=\bigoplus_{P\in\fP}\cE_P,\quad \pi_\lambda=\bigoplus_{P\in\fP}
\pi_{P,\lambda}.
\]
Then $\pi_\lambda$ acts in $\boldsymbol{\cE}$ as induced representation. For
$\mathbf{\Phi}=(\Phi_P)\in\boldsymbol{\cE}$ and $\lambda\in\C$ put
\[
{E}(\mathbf{\Phi},\lambda,x)=\sum_{P\in\fP}E(P,\Phi_P,\lambda,x).
\]
Let $\boldsymbol{\cE}^0=\oplus_{P\in\fP}\cE_P^0$. Let $w_0$ be the nontrivial
element
of $W(A)$. Then the operators 
$c_{P^\prime|P}(k_P'w_0k_P^{-1}:\lambda)$ can be combined into a linear operator
\[
\mathbf{C}(\lambda)\colon \boldsymbol{\cE}^0\to \boldsymbol{\cE}^0,
\]
which is a meromorphic function of $\lambda$. 

The space 
$\boldsymbol{\cE}^0$ decomposes into the direct sum of finite-dimensional
subspaces as follows. Let $P=M_PA_PN_P$ be a $\Gamma$-cuspidal proper parabolic
subgroup.
For $\sigma_P\in\hat M_P$ and 
$\nu\in\hat K$ let $\cE(\sigma_P,\nu)$ be the space of all continuous functions
$\Phi\colon (\Gamma\cap P)A_PN_P\bs G\to \C$ such that for all $x\in G$ the
function $m\in M_P\mapsto \Phi(mx)$ belongs to the $\sigma_P$-isotypical 
subspace of the right regular representation of $M$ and for all $x\in G$ the 
function $k\in K\mapsto \Phi(xk)$ belongs to the $\nu$-isotypical subspace of 
the right regular representation of $K$. For $\sigma\in\hat{M}$ set
\[
\boldsymbol{\cE}(\sigma,\nu):=\bigoplus_{P\in\fP}\cE(\sigma_P,\nu),
\]
where $\sigma_P\in\hat M_P$ is obtained from $\sigma$ by conjugation. 
Each $\boldsymbol{\cE}(\sigma,\nu)$ is finite-dimensional. Furthermore, let
\[
\boldsymbol{\cE}(\sigma):=\bigoplus_{\nu\in\hat K}\boldsymbol{\cE}(\sigma,\nu),
\]
where the direct sum is understood in the algebraic sense. Now
consider an orbit $\vartheta\in W(A)\bs \hat M$. Let $\vartheta=\{\sigma,
w\sigma\}$. 
Put
\[
\boldsymbol{\cE}(\vartheta,\nu):=\begin{cases}\boldsymbol{\cE}(\sigma,\nu),&
w\sigma=\sigma,\\
\boldsymbol{\cE}(\sigma,\nu)\oplus\boldsymbol{\cE}(w\sigma,\nu),&
w\sigma\neq \sigma
\boldsymbol{\cE}(\vartheta,\nu).
\end{cases}
\]
Then it follows that
\begin{equation}\label{decomp3}
\boldsymbol{\cE}^0=\bigoplus_{\vartheta,\nu}\boldsymbol{\cE}
(\vartheta,\nu),
\end{equation}
where $\vartheta$ runs over $W(A)\bs\hat M$ and $\nu$ over $\hat K$.
The operator $\mathbf{C}(\lambda)$ preserves this decomposition. For 
$\vartheta\in W(A)\bs \hat M$, $\nu\in\hat K$ and $\lambda\in\C$ let
\begin{equation}\label{intertwop}
\mathbf{C}(\vartheta,\nu,\lambda)\colon \boldsymbol{\cE}(\vartheta,\nu)\to
\boldsymbol{\cE}(\vartheta,\nu)
\end{equation}
be the restriction of $\mathbf{C}(\lambda)$. We note that for $\vartheta=
\{\sigma,w\sigma\}$, $\mathbf{C}(\vartheta,\nu,\lambda)$ maps 
$\boldsymbol{\cE}(\sigma,\nu)$ into $\boldsymbol{\cE}(w\sigma,\nu)$.
We denote the corresponding operator by
\begin{equation}\label{intertwop1}
\mathbf{C}(\sigma,\nu,\lambda)\colon \boldsymbol{\cE}(\sigma,\nu)\to
\boldsymbol{\cE}(w\sigma,\nu).
\end{equation}
Taking the direct sum with respect to $\nu\in\hat K$, we get operators
\begin{equation}\label{intertwop2}
\mathbf{C}(\sigma,\lambda)\colon \boldsymbol{\cE}(\sigma)\to
\boldsymbol{\cE}(w\sigma).
\end{equation}
Next we recall the functional equations satisfied by $E$ and $C$. 
For $\mathbf{\Phi}\in\boldsymbol{\cE}^0$ and
$\lambda\in\C$ we have
\begin{equation}\label{FE1}
E(\mathbf{\Phi},\lambda)=E(\mathbf{C}(\lambda)\mathbf{\Phi},-\lambda),
\end{equation}
and
\begin{equation}\label{FE2}
\mathbf{C}(\lambda)\mathbf{C}(-\lambda)=\Id.
\end{equation}
Furthermore, let $f\in C^\infty_c(G)$ be right $K$-finite. Then $\pi_\lambda(f)$
acts on $\boldsymbol{\cE}^0$ and we have
\begin{equation}
\mathbf{C}(\lambda)\pi_\lambda(f)=\pi_{-\lambda}(f)
\mathbf{C}(\lambda),\quad \lambda\in\C.
\end{equation}
Thus $\mathbf{C}(\lambda)$ is an intertwining operator for the induced 
representation $\pi_\lambda$. 

Now we come to the relation with the spectral resolution of $R_\Gamma^c$. 
For $P=M_PA_PN_P\in\fP$ let $R_{M_P}$ denote the right regular representation of
$M_P$ on
$L^2(M_P)$. Since $M_P$ is compact, it decomposes discretely as
\begin{equation}\label{regrep}
R_{M_P}=\bigoplus_{\sigma_P\in\hat M_P}d(\sigma_P)\sigma_P,
\end{equation}
where $d(\sigma_P)=\dim(\sigma_P)$. For $\lambda\in\C$ let
$\xi_\lambda:A_P\to\C$ be the quasi-character given by
$\xi_\lambda(a_P(t)):=e^{t\lambda}$.
Let $\Ind_P^G(R_{M_P},\lambda)$, 
be the representation of $G$, induced from
$R_{M_P}\otimes \xi_{\lambda+(d-1)/2}$. Then
we have 
\[
\pi_{P,\lambda}\cong \Ind_P^G(R_{M_P},\lambda).
\]
The theory of Eisenstein series implies that
\[
R_\Gamma^c\cong\bigoplus_{P\in\fP}\int_\R\pi_{P,\lambda}\,d\lambda=
\int_\R\pi_\lambda\,d\lambda.
\]
Using the decomposition \eqref{regrep}, the induced representation decomposes
correspondingly into the direct sum of principal series representations
$\pi_{\sigma,\lambda}$. This gives the spectral resolution of $R_\Gamma^c$
(see \cite[Section 3]{Warner}). 

The spectral resolution of $R_\Gamma^c$ can be described explicitly in terms of
Eisenstein
series as follows. Let $\{e_n\colon n\in I\}$ be an orthonormal basis of
$\boldsymbol{\cE}$ which is adapted to the decomposition \eqref{decomp3}, i.e.,
each $e_n$ belongs to some subspace $\boldsymbol{\cE}(\vartheta,\nu)$.  Then 
$R_\Gamma(\alpha)$ is an integral operator with kernel $K_\alpha(x,y)$. The
following proposition is main result about the spectral resolution of the
kernel.
\begin{prop}\label{specres5}
Let $\alpha$ be a $K$-finite function in $\mathcal{C}^1(G)$. Then 
$R_\Gamma^c(\alpha)$ is an integral operator with kernel $K^c_\alpha(x,y)$
given by
\begin{equation}\label{contkernel}
K^c_\alpha(x,y)=\frac{1}{4\pi}\sum_{m,n\in I}\int_\R 
\langle\pi_\lambda(\alpha)e_m,e_n\rangle E(e_n,i\lambda,x)
\overline{E(e_m,i\lambda,y)}\;d\lambda.
\end{equation}
Furthermore, the kernel
$
K^d_\alpha=K_\alpha-K^c_\alpha
$
is integrable over the diagonal and
\[
\Tr(R_\Gamma^d(\alpha))=\int_{\Gamma\bs G} K_\alpha^d(x,x)\;dx.
\]
\end{prop}
\begin{proof}
See \cite[Theorem 4.7]{Warner}. 
\end{proof}

The Eisenstein series are not square integrable. However, the truncated
Eisenstein series, which are obtained by subtracting the constant terms in
each cups, are square integrable. Their inner product  gives rise to the 
Maass-Selberg relations which we recall next. 

Let $Y_0>0$ be such that \eqref{Zerlegung des FB} holds. Let $Y\geq Y_0$. For 
$P\in\mathfrak{P}$ let $\chi_{P,Y}$ be the characteristic function of
$N_PA_{P}^0\left[Y\right]K\subset G$. Let $\Phi\in\boldsymbol{\mathcal{E}}^0$. 
For $Y\geq Y_0$ put
\begin{align*}
E^Y(\Phi,\lambda,x):=E(\Phi,\lambda,x)
-\sum_{P\in\mathfrak{P}}\frac{1}{\vol\left(\Gamma\cap
N_P\backslash N_P\right)}\sum_{\gamma\in\Gamma\cap
N_P\backslash\Gamma}\chi_{P,Y}
(\gamma g)E_{P}(\Phi,\lambda,\gamma g),
\end{align*}
where $E_{P}(\Phi,\lambda,x)$ is as in \eqref{c-term1}. 
By \eqref{Eigenschaft des FB} at most one summand in this sum is not zero.
By \cite{Harish-Chandra} the function $E^Y(\Phi,\lambda)$ belongs to
$L^2(\Gamma\backslash G)$. 
Now we have the following proposition.
\begin{prop}\label{KoMs}
Let  $\Phi,\Psi\in\boldsymbol{\cE}^0$ and $\lambda\in\aL^*$. Then one has
\begin{align*}
&\int_{\Gamma\bs G}E^Y(\Phi,i\lambda,x)\overline{E^Y}(\Psi,i\lambda,x)\;dx
=-\left<\mathbf{C}(-i\lambda)\frac{d}{dz}\mathbf{C}(i\lambda)\Phi,\Psi\right>
\nonumber\\
&+2\left<\Phi,\Psi\right>\log{Y}+\frac{Y^{2i\lambda}}{2i\lambda}
\left<\Phi,\mathbf{C}(i\lambda)\Psi\right>-\frac{Y^{-2i\lambda}}{2i\lambda}
\left<\mathbf{C}(i\lambda)\Phi,\Psi\right>.
\end{align*}
\end{prop}
At the end of this section, we remark that the space $L^{2}_{d}(\Gamma\backslash
G)$ admits a further decomposition
\begin{align}\label{Weitere Zerlegung des diskreten Teiles}
L^{2}_{d}(\Gamma\backslash
G)=L^{2}_{{\rm{cusp}}}(\Gamma\backslash G)\oplus
L^{2}_{{\rm{res}}}(\Gamma\backslash G).
\end{align}
Here $L^{2}_{{\rm{cusp}}}(\Gamma\backslash G)$ is the space spanned by the cusp
forms, i.e. the square integrable functions $f$, which for all
$P\in\mathfrak{P}$ satisfy
\begin{align*}
f_{P}^{0}(x):=\int_{\Gamma\cap N_{P}\backslash N_{P}}{f(nx)dn}=0\quad \text{for
almost all $x\in G$}.
\end{align*}
One does not know much about $L^{2}_{{\rm{cusp}}}(\Gamma\backslash G)$ and its
size in general. On the other hand, let
$\Phi\in\boldsymbol{\mathcal{E}}(\boldsymbol{\sigma},\nu)$. Let $s_{0}\in (0,n]$
be a pole of $E(\Phi,s)$. Then the function $x\mapsto
{\rm{Res}}|_{s=s_{0}}E(\Phi,s)$ is square integrable on $\Gamma\backslash G$
and $L^{2}_{{\rm{res}}}(\Gamma\backslash G)$ is spanned by all these residues of
Eisenstein series.

\section{Bochner Laplace operators}\label{secBLO}
\setcounter{equation}{0}
Regard $G$ as a principal $K$-fibre bundle over $\tilde{X}$. By the invariance
 of $\mathfrak{p}$ under $\Ad(K)$, the assignment
\begin{align*}
T_{g}^{{\rm{hor}}}:=\{\frac{d}{dt}\bigr|_{t=0}g\exp{tX}\colon
X\in\mathfrak{p}\} 
\end{align*}
defines a horizontal distribution on $G$. This connection is called the 
canonical connection. 
Let $\nu$ be a finite-dimensional unitary representation of $K$ on 
$(V_{\nu},\left<\cdot,\cdot\right>_{\nu})$. Let
\begin{align*}
\tilde{E}_{\nu}:=G\times_{\nu}V_{\nu}
\end{align*}
be the associated homogeneous vector bundle over $\tilde{X}$. Then 
$\left<\cdot,\cdot\right>_{\nu}$ induces a $G$-invariant metric 
$\tilde{B}_{\nu}$ on $\tilde{E}_{\nu}$. Let $\widetilde{\nabla}^{\nu}$ be the 
connection on $\tilde{E}_{\nu}$ induced by the canonical connection. Then 
$\widetilde{\nabla}^{\nu}$ is $G$-invariant.
Let  
\begin{align*}
E_{\nu}:=\Gamma\backslash(G\times_{\nu}V_{\nu})
\end{align*}
be the associated locally homogeneous bundle over $X$. Since 
$\tilde{B}_{\nu}$ and $\widetilde{\nabla}^{\nu}$ are $G$-invariant, they push 
down to a metric $B_{\nu}$ and a connection $\nabla^{\nu}$ on $E_{\nu}$. Let
\begin{align}\label{globsect}
C^{\infty}(G,\nu):=\{f:G\rightarrow V_{\nu}\colon f\in C^\infty,\;
f(gk)=\nu(k^{-1})f(g),\,\,\forall g\in G, \,\forall k\in K\}.
\end{align}
Let
\begin{align}\label{globsect1}
C^{\infty}(\Gamma\backslash G,\nu):=\left\{f\in C^{\infty}(G,\nu)\colon 
f(\gamma g)=f(g)\:\forall g\in G, \forall \gamma\in\Gamma\right\}.
\end{align}
Let $C^{\infty}(X,E_{\nu})$ denote the space of smooth sections of $E_{\nu}$. 
Then there is a canonical isomorphism
\begin{align*}
A:C^{\infty}(X,E_{\nu})\cong C^{\infty}(\Gamma\backslash G,\nu)
\end{align*}
(see \cite[p. 4]{Mi1}).
There is also a corresponding isometry for the space $L^{2}(X,E_{\nu})$ of 
$L^{2}$-sections of $E_{\nu}$. For every $X\in\mathfrak{g}$, $g\in G$ and 
every $f\in C^{\infty}(X,E_{\nu})$ one has
\begin{align*}
A(\nabla^{\nu}_{L(g)_{*}X}f)(g)=\frac{d}{dt}|_{t=0}Af(g\exp{tX}).
\end{align*}
Let
$\widetilde{\Delta}_{\nu}={\widetilde{\nabla^\nu}}^{*}{\widetilde{\nabla}}^{\nu}
$
be 
the Bochner-Laplace operator of $\widetilde{E}_{\nu}$. Since $\widetilde{X}$ is
complete, 
$\widetilde{\Delta}_{\nu}$ with domain the smooth compactly supported sections
is 
essentially self-adjoint \cite{Chernoff}. Its self-adjoint extension
will be denoted by $\widetilde{\Delta}_\nu$ too. By \cite[Proposition
1.1]{Mi1} it follows that on $C^{\infty}(G,\nu)$
one has
\begin{align}\label{BLO}
\widetilde{\Delta}_{\nu}=-R_\Gamma(\Omega)+\nu(\Omega_K),
\end{align} 
where $\Omega_K$ is the Casimir operator of $\mathfrak{k}$ with respect to the
restriction of the normalized Killing form of $\mathfrak{g}$ to
$\mathfrak{k}$. 
Let $\tilde{A}_{\nu}$ be the differential operator on $E_{\nu}$ which acts as
$-R_\Gamma(\Omega)$ on $C^{\infty}(G,\nu)$.  Then it follows from
\eqref{BLO} that $\tilde{A}_{\nu}$ is bounded from below and essentially
self-adjoint.
Its self-adjoint extension will be denoted by $\tilde{A}_\nu$ too.
Let $e^{-t\tilde{A}_{\nu}}$ be the corresponding heat semigroup on
$L^2(G,\nu)$, where $L^{2}(G,\nu)$ is defined analogously to 
\eqref{globsect}. 
Then the same arguments as in \cite[section1]{Cheeger} imply that there 
exists a function
\begin{align}\label{DefK}
K_t^\nu\in C^{\infty}(G\times G,\End(V_\nu)),
\end{align}
with the following properties: $K_t^\nu(g,g')$ is symmetric in in the 
$G$-variables, for each $g\in G$,  the function $g'\mapsto K_t^\nu(g,g')$
belongs to $L^2(G,\End(V_\nu))$, it satisfies 
\begin{align*}
K_t^\nu(gk,g'k')=\nu(k^{-1})K_t^\nu(g,g')\nu(k'),\: \forall g,g'\in
G,\: \forall k,k'\in K
\end{align*}
and it is the kernel of the heat operator, i.e., 
\begin{align*}
(e^{-t\tilde{A}_\nu}\phi)(g)=\int_{G}K_t^\nu(g,g')\phi(g')dg',\;\forall\phi\in
L^2(G,\nu).
\end{align*}
Since $\Omega$ is $G$-invariant, $K_t^\nu$ is invariant under the diagonal
action
of $G$.
Hence there exists a function 
\begin{align*}
{H}^{\nu}_{t}:G\longrightarrow {\rm{End}}(V_{\nu})
\end{align*}
which satisfies 
\begin{align}\label{propH}
{H}^{\nu}_{t}(k^{-1}gk')=\nu(k)^{-1}\circ {H}^{\nu}_{t}(g)\circ\nu(k'),
\:\forall k,k'\in K, \forall g\in G,
\end{align}
such that
\begin{align}\label{KH}
K_t^\nu(g,g')=H_t^\nu(g^{-1}g'),\quad\forall g,g'\in G.
\end{align}
Thus one has
\begin{align}\label{kernelcov}
(e^{-t\tilde{A}_{\nu}}\phi)(g)=\int_{G}{{H}^{\nu}_{t}(g^{-1}g')\phi(g')dg'}
,
\quad\phi\in  L^{2}(G,\nu),\quad g\in G.
\end{align}
By the  arguments of \cite[Proposition 2.4]{BM},  $H^\nu_t$ belongs to
all 
Harish-Chandra Schwartz spaces 
$(\mathcal{C}^{q}(G)\otimes {\rm{End}}(V_{\nu}))$, $q>0$.\\
Now we pass to the quotient $X=\Gamma\backslash\widetilde{X}$.
Let $\Delta_\nu={\nabla^\nu}^*\nabla^\nu$ the closure of the Bochner-Laplace
operator
with domain the smooth compactly supported sections of $E_\nu$. Then
$\Delta_\nu$ is self-adjoint 
and by \eqref{BLO} it induces the operator
$-R_\Gamma(\Omega)+\nu(\Omega_K)$ on $C^\infty(\Gamma\backslash G,\nu)$. Thus if
we let $A_\nu$ be the operator 
$-R_\Gamma(\Omega)$ on $C_c^\infty(\Gamma\backslash G,\nu)$, then $A_\nu$ is
bounded from
below and essentially self-adjoint. The closure of $A_\nu$ will be
denoted by $A_\nu$ too. 
Let $e^{-tA_{\nu}}$ be the heat-semigroup of $A_\nu$ on $L^{2}(\Gamma\backslash
G,\nu)$. Let
\begin{align}\label{kernelx}
H^{\nu}(t;x,x'):=\sum_{\gamma\in\Gamma}{H}^{\nu}_{t}(g^{-1}\gamma g'),
\end{align}
where $x,x'\in \Gamma\backslash G$, $x=\Gamma g $, $x'=\Gamma g' $. 
By \cite[Chapter 4]{Warner} this series converges absolutely and locally
uniformly. It follows from \eqref{kernelcov} that
\begin{align*}
(e^{-tA_{\nu}}\phi)(x)=\int_{\Gamma\backslash
G}{{H}^{\nu}(t;x,x')\phi(x')dx'}
,
\quad\phi\in  L^{2}(\Gamma\backslash G,\nu),\quad x\in \Gamma\backslash G.
\end{align*}
Put
\begin{align}\label{Deffh}
{h}^{\nu}_{t}(g):=\tr{H}^{\nu}_{t}(g),
\end{align}
where $\tr$ denotes the trace in $\End{V_\nu}$. Define the operator
$R_{\Gamma}(h^{\nu}_{t})$ on $L^{2}(\Gamma\backslash G)$ by
\begin{align*}
R_{\Gamma}(h^{\nu}_{t})f(x):=\int_{G}{h^{\nu}_{t}(g)f(xg)dg}.
\end{align*}
Then $R_{\Gamma}(h^{\nu}_{t})$ is an integral-operator on
$L^{2}(\Gamma\backslash G)$, whose kernel is given by
\begin{align}\label{Defhnut}
h^{\nu}(t;x,x'):=\tr H^{\nu}(t;x,x').
\end{align}
We shall now compute the Fourier transform of $h^{\nu}_{t}$. 
Let $\pi$ be a unitary admissible representation of $G$ on a Hilbert space 
$\mathcal{H}_\pi$. Let $\check{\nu}$ be the contragredient representation of
$\nu$ and let
$P_{\check{\nu}}(\pi)$ be the projection of $\mathcal{H}_\pi$ onto
$\mathcal{H}_\pi^{\check{\nu}}$
, the $\check{\nu}$-isotypical component 
of $\mathcal{H}_\pi$. By assumption $\mathcal{H}_\pi^{\check{\nu}}$ is finite
dimensional.
Furthermore, by \eqref{propH} on has
\begin{align}\label{Proj}
\pi(h^{\nu}_{t})=P_{\check{\nu}}(\pi)\pi(h^{\nu}_t)P_{\check{\nu}}(\pi).
\end{align}
By \cite{Onishchik}, $\S$ 4,
Proposition 4 and $\S$ 7, Proposition 3 we have $\check{\nu}\cong\nu$. The
restriction of 
$\pi(h^{\nu}_{t})$ to $\mathcal{H}_\pi^{\nu}$ will be denoted by 
$\pi(h^{\nu}_{t})$ too. Define a bounded operator on 
$\mathcal{H}_{\pi}\otimes V_{\nu}$ by
\begin{align}\label{Definiton von pi(tilde(H))}
\tilde{\pi}(H^{\nu}_{t}(g)):=\int_{G}{\pi(g)\otimes H^{\nu}_{t}(g)dg}.
\end{align}
Then relative to the splitting 
\begin{align*}
\mathcal{H}_{\pi}\otimes V_{\nu}=\left(\mathcal{H}_{\pi}\otimes 
V_{\nu}\right)^{K}\oplus\left(\left(\mathcal{H}_{\pi}\otimes 
V_{\nu}\right)^{K}\right)^{\bot},
\end{align*}
$\tilde{\pi}(H^{\nu}_{t})$ has the form
\begin{align*}
\begin{pmatrix}
\pi({H}^{\nu}_{t})&0\\0&0
\end{pmatrix},
\end{align*}
where $\pi(H^{\nu}_{t})$ acts on $\left(\mathcal{H}_{\pi}\otimes 
V_{\nu}\right)^{K}$.  It follows as in \cite[Corollary 2.2]{BM} that
\begin{equation}\label{integop}
\pi(H_t^\nu)=e^{t\pi(\Omega)}\Id,
\end{equation}
where $\Id$ is the identity on $\left(\mathcal{H}_{\pi}\otimes 
V_{\nu}\right)^{K}$.
Now let $A:\mathcal{H}_\pi\rightarrow \mathcal{H}_\pi$ be a bounded operator
which is an
intertwining operator for $\pi|_{K}$. Then $A\circ\pi(h_t^\nu)$ is again a
finite rank
operator. Define an operator $\tilde{A}$ on $\mathcal{H}_\pi\otimes V_\nu$ by
$\tilde{A}:=A\otimes\Id$.
Then by the same argument as in \cite[Lemma 5.1]{BM} one has
\begin{equation}\label{equtrace}
\Tr \left(\tilde{A}\circ
\tilde{\pi}(H_t^\nu)\right)=\Tr\left(A\circ\pi(h_t^\nu)\right).
\end{equation}
Together with \eqref{integop} we obtain
\begin{equation}\label{TrFT}
\Tr\left(A\circ\pi(h_t^\nu)\right)=e^{t\pi(\Omega)}\cdot\Tr\tilde{A}|
_{(\mathcal{H}_{\pi}\otimes V_{\nu})^{K}}.
\end{equation}
Let $\pi\in\hat{G}$ and let $\Theta_\pi$ be its global character.
Taking $A=\Id$ in \eqref{TrFT}, one obtains
\begin{equation*}
\Theta_{\pi}(h^{\nu}_{t})=e^{t\pi(\Omega)}\cdot
\dim(\mathcal{H}_{\pi}\otimes V_{\nu})^{K}=e^{t\pi(\Omega)}
\cdot\left[\pi:\check{\nu}\right]=e^{t\pi(\Omega)}\cdot
\left[\pi:\nu\right].
\end{equation*}
By \cite[Theorem 9.16]{Knapp2} we have $\left[\nu:\sigma\right]\leq 1$
for all $\nu\in\hat{K}$ and all $\sigma\in\hat{M}$. Now consider
the principal series representation $\pi_{\sigma,\lambda}$, where
$\sigma\in\hat{M}$ and $\lambda\in\mathbb{R}$. Let 
$\Theta_{\sigma,\lambda}$ be the global character of $\pi_{\sigma,\lambda}$.
By Frobenius reciprocity \cite[p.208]{Knapp} it follows that for all
$\nu\in\hat K$
\[
\left[\pi_{\sigma,\lambda}:\nu\right]=\left[\nu:\sigma\right].
\]
Hence it follows that for $\left[\nu:\sigma\right]\neq 0$
\begin{align*}
\Theta_{\sigma,\lambda}(h^{\nu}_{t})&=e^{t\pi_{\sigma,\lambda}(\Omega)}
\end{align*}
and $
\Theta_{\sigma,\lambda}(h^{\nu}_{t})
=0
$
for $\left[\nu:\sigma\right]=0$.
The Casimir eigenvalue can be computed as follows. For $\sigma\in\hat M$ with
highest weight given by \eqref{Darstellungen von M} resp. \eqref{Darstellungen
von M'}, let
\begin{align}\label{csigma}
c(\sigma):=\sum_{j=2}^{n+1}(k_{j}(\sigma)+\rho_{j})^{2}-\sum_{j=1}^{n+1}
\rho_{j}^{2}.
\end{align}
Then one has
\begin{align}\label{infchps}
\pi_{\sigma,\lambda}(\Omega)=-\lambda^{2}+c(\sigma).
\end{align}
For $G=\Spin(2n+1,1)$ this was proved in \cite[Corollary 2.4]{MP}. For
$G=\Spin(2n+2,1)$,
one can proceed in the same way. 
Thus we obtain the following proposition.
\begin{prop}\label{fouriertrf2}
For $\sigma\in\hat{M}$ and $\lambda\in\mathbb{R}$ let 
$\Theta_{\sigma,\lambda}$ be the global character of $\pi_{\sigma,\lambda}$.
Let $c(\sigma)$ be defined by \eqref{csigma}.
Then one has
\begin{align*}
\Theta_{\sigma,\lambda}(h^{\nu}_{t})&=e^{t(c(\sigma)-\lambda^{2})}
\end{align*} for 
$\left[\nu:\sigma\right]\neq 0$ and $\Theta_{\sigma,\lambda}(h^{\nu}_{t})
=0$ otherwise. 
\end{prop}
Finally, by the definition of $\boldsymbol{\pi}_{\sigma,\lambda}$,
\eqref{infchps} also gives
\begin{align}\label{Infchp}
\boldsymbol{\pi}_{\sigma,\lambda}(\Omega)=\lambda^2+c(\sigma).
\end{align}

\section{The regularized trace}\label{secrel}
\setcounter{equation}{0}
In this section we define the regularized trace of the heat operator. 
The decomposition \eqref{Zerlegung von L2} induces a decomposition of
$L^2(\Gamma\backslash G,\nu)\cong \left(L^2(\Gamma\backslash G,\nu)\otimes
V_\nu\right)^K$ as
\begin{align*}
L^2(\Gamma\backslash
G,\nu)=L^2_d(\Gamma\backslash
G,\nu)\oplus L^2_c(\Gamma\backslash
G,\nu). 
\end{align*}
This decomposition is invariant under $A_\nu$ in the sense of unbounded
operators.
Let $A_\nu^d$ denote the restriction of $A_\nu$ to $L^2_d(\Gamma\backslash
G,\nu)$. Then the spectrum of $A_\nu^d$ is discrete. Let 
$\lambda_{1}\leq\lambda_{2}\leq...$ be the sequence of eigenvalues of 
$A_{\nu}^{d}$, counted with multiplicity. 
This sequence may be finite or infinite. For $\lambda\in\left[0,\infty\right)$ 
let
\begin{align*}
N(\lambda):=\#\{j\colon\lambda_{j}\leq\lambda\}.
\end{align*}
be the counting function of eigenvalues. 
By \cite[Theorem 0.1]{Mu4} there exists  $C>0$ such that
\begin{align}\label{Weylsches Gesetz von Donelly}
N(\lambda)\leq C(1+\lambda^{2d})
\end{align} 
for all $\lambda\ge 0$. In fact, in the present case, the exponent is $d/2$.
This follows from estimation of the counting function of the cuspidal 
eigenvalues \cite[Theorem 9.1]{Do2} together with the fact that the residual 
spectrum is finite in the present case. Hence the sum
$\sum_{j}e^{-t\lambda_{j}}$
converges for all $t>0$, the operator $e^{-tA_{\nu}^{d}}$ is of trace class and 
one has
\begin{align}
\Tr\left(e^{-tA_{\nu}^{d}}\right)=\sum_{j}{e^{-t\lambda_{j}}}.
\end{align}
Let $H^\nu_t$ be the kernel of $e^{-t\tilde A_\nu}$ and let $h_t^\nu=\tr
H^\nu_t$. 
Then $h_t^\nu$ belongs to $\mathcal{C}^1(G)$. Let $h^\nu(t;x,y)$ be the kernel
of $R_\Gamma(h^\nu_t)$. By Proposition \ref{specres5},
the kernel $h^\nu_c(t;x,y)$ of $R_\Gamma^c(h_t^\nu)$ is given by
\begin{equation}\label{contp}
h^{\nu}_{c}(t;x,y)=\frac{1}{4\pi}\sum_{k,l}
\int_{\mathbb{R}}{\left<\pi_{i\lambda}(h_t^\nu)e_l,e_k\right>E
(e_k,i\lambda,x)\overline{E}(e_l,i\lambda,y)\;d\lambda},
\end{equation}
where $\{e_k\colon k\in I\}$ is an orthonormal basis of $\boldsymbol{\cE}$
adapted to the decomposition \eqref{decomp3}. Let
\begin{align}\label{ZerKern}
h^\nu_d(t;x,y)=h^\nu(t;x,y)-h^\nu_c(t;x,y).
\end{align}
By the second part of Proposition \ref{specres5}, $h_d^\nu$ is the kernel of 
$R_\Gamma^d(h_t^\nu)$ and we have
\begin{align}\label{Spurdisk}
\Tr(e^{-tA_{\nu}^{d}})=\Tr(R_{\Gamma}^d(h^{\nu}_t))=\int_{\Gamma\backslash G}
h^{\nu}_d(t;x,x)dx.
\end{align}

Now the argument on page 82 in \cite{Warner} can be extended to 
$h_t^\nu\in\mathcal{C}(G)$ and one has 
\begin{align*}
\int_{\mathbb{R}}\int_{\Gamma\backslash
G}{{\left|\sum_{k,l}\left<\boldsymbol{\pi}_{i\lambda}(h_t^\nu)e_l,e_k\right>
E^Y(e_k,i\lambda,x)\overline{E}^Y(e_l,i\lambda,x)\right|dx}d\lambda}
<\infty.
\end{align*}
Thus one can apply Proposition \ref{KoMs} and interchange the order of
integration. Let $\mathbf{C}(\sigma,\nu,\lambda)$ be the operator 
\eqref{intertwop1}. Arguing now as in \cite[page 82-84]{Warner}
and using Proposition \ref{fouriertrf2} one obtains
\begin{align*}
\int_{X(Y)}h^\nu_c(t;x,x)\;dx&=
\sum_{\substack{\sigma\in\hat{M};\sigma=w_0\sigma\\\left[\nu:\sigma\right]\neq0}
}
\frac{\Tr\left(\boldsymbol{\pi}_{\sigma,0}(h_t^\nu)\mathbf{C}
(\sigma,\nu,0)\right)}{4}+\sum_{\substack{\sigma\in\hat{M}\\
\left[\nu:\sigma\right]\neq
0}}\biggl(\frac{\kappa e^{tc(\sigma)}\log{Y}\dim(\sigma)}{\sqrt{4\pi t}}\\
&-\frac{1}{4\pi}\int_{\R}\Tr\left(\boldsymbol{\pi}_{\sigma,
i\lambda}(h_t^\nu)\mathbf{C}(\sigma,\nu,-i\lambda)\frac{d}{dz}
\mathbf{C}(\sigma,\nu,i\lambda)\right)\,d\lambda
\biggr)+o(1),
\end{align*}
as $Y\to\infty$. Now recall that the restriction of the representation
$\boldsymbol{\pi}_{\sigma,i\lambda}$
to $K$ is independent of the parameter $\lambda$. Let 
\[
\widetilde{\boldsymbol{C}}(\sigma,\nu,\lambda)\colon  
(\boldsymbol{\cE}(\sigma)\otimes V_\nu)^K\to 
(\boldsymbol{\cE}(\sigma)\otimes V_\nu)^K
\]
be the restriction of
$\boldsymbol{C}(\sigma,\lambda)\otimes \Id_{V_\nu}$ to 
$(\boldsymbol{\cE}(\sigma)\otimes V_\nu)^K$, where 
$\boldsymbol{C}(\sigma,\lambda)$ be the operator \eqref{intertwop2}.
Using the intertwining property of
$\boldsymbol{C}(\sigma,\lambda)$, equation \eqref{TrFT} and equation
\ref{Infchp}
one obtains
\begin{align*}
\int_{X(Y)}h^\nu_c(t;x,x)\,dx&=\sum_{\substack{\sigma\in\hat{M}
;\sigma=w_0\sigma\\
\left[\nu:\sigma\right]\neq
0}}e^{tc(\sigma)}\frac{\Tr(\widetilde{\boldsymbol{C}}(\sigma,\nu,0))}{4}
+\sum_{\substack{\sigma\in\hat{M}\\
\left[\nu:\sigma\right]\neq
0}}\biggl(\frac{\kappa e^{tc(\sigma)}\log{Y}\dim(\sigma)}{\sqrt{4\pi t}}\\
&-\frac{1}{4\pi}\int_{\R}e^{-t\left(\lambda^2-c(\sigma)\right)}
\Tr\left(\widetilde{\boldsymbol{C}}(\sigma,\nu,-i\lambda)\frac{
d}{dz}\widetilde{\boldsymbol{C}}(\sigma,\nu,i\lambda)\right)d\lambda
\biggr)+o(1),
\end{align*}
as $Y\to\infty$. Thus together with \eqref{ZerKern}, \eqref{Spurdisk} we obtain
\begin{equation}
\begin{split}
\int_{X(Y)}h^\nu(t;x,x)\,dx&=
\sum_{\substack{\sigma\in\hat{M}\\
\left[\nu:\sigma\right]\neq0}}
\frac{\kappa e^{tc(\sigma)}\dim(\sigma)\log{Y}}{\sqrt{4\pi t}}
+\sum_j e^{-t\lambda_j}\\
&+\sum_{\substack{\sigma\in\hat{M};\sigma=w_0\sigma\\
\left[\nu:\sigma\right]\neq
0}}e^{tc(\sigma)}\frac{\Tr(\widetilde{\boldsymbol{C}}(\sigma,\nu,0))}{4}\\
&-\frac{1}{4\pi}\sum_{\substack{\sigma\in\hat{M}\\
\left[\nu:\sigma\right]\neq
0}}\int_{\R}e^{-t\left(\lambda^2-c(\sigma)\right)}
\Tr\left(\widetilde{\boldsymbol{C}}
(\sigma,\nu,-i\lambda)\frac{
d}{dz}\widetilde{\boldsymbol{C}}(\sigma,\nu,i\lambda)\right)\,d\lambda\\
&+o(1)
\end{split}
\end{equation}
as $Y\to\infty$. It follows that $\int_{X(Y)} \tr h^\nu(t;x,x) dx$ has an 
asymptotic expansion as $Y\to\infty$ and following \cite{Me}, 
we take the constant coefficient as the definition of the regularized trace. 
\begin{defn} The regularized trace of $e^{-tA_\nu}$ is defined as
\begin{equation}\label{regtrace2}
\begin{split}
\Tr_{\rel}\left(e^{-tA_\nu}\right)&=\Tr\left(e^{-tA^d_\nu}\right)+
\sum_{\substack{\sigma\in\hat{M};\sigma=w_0\sigma\\
\left[\nu:\sigma\right]\neq
0}}e^{tc(\sigma)}\frac{\Tr(\widetilde{\boldsymbol{C}}(\sigma,\nu,0))}{4}\\
&-\frac{1}{4\pi}\sum_{\substack{\sigma\in\hat{M}\\
\left[\nu:\sigma\right]\neq
0}}\int_{\R}e^{-t\left(\lambda^2-c(\sigma)\right)}
\Tr\left(\widetilde{\boldsymbol{C}}(\sigma,\nu,-i\lambda)
\frac{d}{dz}\widetilde{\boldsymbol{C}}(\sigma,\nu,i\lambda)\right)\,d\lambda.
\end{split}
\end{equation}
\end{defn}
\begin{bmrk}
The right hand side of \eqref{regtrace2} equals the spectral side
of the Selberg trace formula applied to $\exp(-tA_\nu)$. This follows from
\cite[Theorem 8.4]{Warner}.
\end{bmrk}
\begin{bmrk}
There are slightly  different methods to regularize the trace. One is to  
truncate the zero Fourier coefficients of $h^\nu(t;x,y)$ at level $Y\ge Y_0$.
The resulting kernel $h^\nu_Y(t;x,y)$ is integrable over the diagonal. The
integral $\int_X h^\nu_Y(t;x,x)\;dx$ depends on $Y$ in a simple way. If one
subtracts off the term which contains $Y$, one gets another definition of the
regularized trace which is closely related to \eqref{regtrace2}.
\end{bmrk}

\section{The trace formula}\label{sectr}
\setcounter{equation}{0}

In this section we apply the Selberg trace formula to study the regularized
trace of the heat operator $e^{-tA_\nu}$. To begin with we
briefly recall the Selberg trace formula. First we introduce the
distributions involved. Let $\alpha$ be a K-finite Schwartz function. Let
\begin{align*}
I(\alpha):=\vol(\Gamma\backslash G)\alpha(1).
\end{align*}
By \cite[Theorem 3]{HC}, the Plancherel theorem can be applied to $\alpha$. For
groups of real rank one which do not possess a compact Cartan subgroup it is
stated in \cite[Theorem 13.2]{Knapp}.
Thus if $P_\sigma(z)$ is as in section \ref{secPL}, then for an odd-dimensional
 $X$ one has
\begin{align}\label{Idcontr}
I(\alpha)=\vol(X)\sum_{\substack{\sigma\in\hat{M}\\\left[\nu:\sigma\right]\neq
0}}\int_{\mathbb{R}}{P_{\sigma}(i\lambda)\Theta_{\sigma,\lambda}(\alpha)}
d\lambda,
\end{align}
where the sum is finite since $\alpha$ is $K$-finite. In even dimensions an
additional contribution of the discrete series appears.
Next let $\CC(\Gamma)_{\s}$ be the set of semi-simple conjugacy classes
$\left[\gamma\right]$. Put
\begin{align*}
H(\alpha):=\int_{\Gamma\backslash
G}\sum_{\left[\gamma\right]\in\CC(\Gamma)_{\s}-\left[1\right]}
\alpha(x^{-1}\gamma x)dx.
\end{align*}
By \cite[Lemma 8.1]{Warner} the integral converges absolutely. Its Fourier
transform can be computed as
follows.
Since $\Gamma$ is assumed to be torsion free, every
nontrivial semi-simple element $\gamma$ is conjugate to an element
$m(\gamma)\exp{\ell(\gamma)H_1}$, $m(\gamma)\in M$. By \cite[Lemma
6.6]{Wallach},
$l(\gamma)>0$
is unique and $m(\gamma)$ is determined up to conjugacy in $M$.
Moreover, $\ell(\gamma)$ is the length of the unique closed geodesic
associated to $\left[\gamma\right]$. It follows that $\Gamma_{\gamma}$, the
centralizer of $\gamma$ in $\Gamma$, is infinite cyclic. Let $\gamma_{0}$ denote
its generator which is semi-simple too. For
$\gamma\in\left[\Gamma\right]_{S}-\{\left[1\right]\}$ let
$a_\gamma:=\exp{\ell(\gamma)H_1}$ and let
\begin{align}\label{hyperbcontr}
L(\gamma,\sigma):=\frac{\overline{\Tr(\sigma)(m_{\gamma})}}
{\det\left(\Id-\Ad(m_\gamma a_\gamma)|_{\bar\nf}\right)}e^{-n\ell(\gamma)}.
\end{align}
Assume that $\dim(X)$ is odd. 
Proceeding as in \cite{Wallach} and using \cite[equation 4.6]{Gang}, one obtains
\begin{align}\label{Hyperb}
H(\alpha)=&\sum_{\sigma\in\hat{M}}\sum_{\left[\gamma\right]\in \CC(\Gamma)_{\s}
-\left[1\right]}\frac{l(\gamma_{0})}{2\pi}
L(\gamma,\sigma)\int_{-\infty}^{\infty}{\Theta_{\sigma,\lambda}(\alpha)e^{
-il(\gamma)\lambda}d\lambda},
\end{align}
where the sum is finite since $\alpha$ is $K$-finite.  

Now let $P\in\mathfrak{P}$. For every
$\eta\in\Gamma\cap N_{P}-\{1\}$ let $X_{\eta}:=\log{\eta}$. Write
$\left\|\cdot\right\|$ for the norm induced on $\mathfrak{n}_{P}$ by the
restriction of $\frac{1}{4n}B(\cdot,\theta\cdot)$. Then for $\Real(s)>0$
the Epstein-type zeta function $\zeta_{P}$, defined by
\begin{align}\label{Definition der Ep Zetafunktion}
\zeta_{P}(s):=\sum_{\eta\in\Gamma\cap
N_{P}-\{1\}}{\left\|X_{\eta}\right\|}^{-2n(1+s)},
\end{align}
converges and $\zeta_P$ has a meromorphic continuation to $\C$ with a simple
pole at $0$. 
Let $C_{P}(\Gamma)$ be the constant term of $\zeta_{P}$ at $s=0$. Then put
\begin{align*}
&T_{P}(\alpha):=\int_{K}\int_{N_{P}}{\alpha(kn_{P}k^{-1})dkdn_{P}}=\int_{K}\int_
{N}{\alpha(kn_{0}k^{-1})dn_{0}}\\ 
&T(\alpha):=\sum_{P\in\mathfrak{P}}C_{P}(\Gamma)\frac{\vol(\Gamma\cap
N_{P}\backslash N_{P})}{\vol(S^{2n-1})}T_{P}(\alpha)\\
&T_{P}'(\alpha):=\int_{K}{\int_{N_{P}}{\alpha(kn_{P}k^{-1}){\log\left\|\log{n_{P
}}\right\|}dn_{P}}dk}.
\end{align*} 
Then $T$ and $T_{P'}$ are tempered distributions. 
The distributions $T$ is invariant. Let 
\begin{align*}
C(\Gamma):=\sum_{P\in\mathfrak{P}}C_{P}(\Gamma)\frac{\vol(\Gamma\cap
N_{P}\backslash N_{P})}{\vol(S^{2n-1})}.
\end{align*}
Applying the Fourier inversion formula and
the Peter-Weyl-Theorem to
equation 10.21 in \cite{Knapp}, one obtains 
the Fourier transform of T as:
\begin{align}\label{Fouriertrafo T}
T(\alpha)=\sum_{\sigma\in\hat{M}}\dim(\sigma)\frac{1}{2\pi}C(\Gamma)\int_{
\mathbb{R}}\Theta_{\sigma,\lambda
}(\alpha)d\lambda.
\end{align}
The distributions $T_{P}'$ are not invariant. However, they can be made
invariant using the standard Knapp-Stein intertwining operators. These operators
are defined as follows. Let $\bar{P}_{0}:=\bar{N_0}A_0M_0$ be the
parabolic subgroup opposite to $P_{0}$. Let $\sigma\in\hat{M}$ and let
$(\cH^\sigma)^\infty$ be the subspace of $C^\infty$-vectors in 
$\cH^\sigma$. For $\Phi\in({\cH}^{\sigma})^\infty$ and 
$\lambda\in\C$ define $\Phi_\lambda:G\rightarrow V_\sigma$ by
\begin{align*}
\Phi_{\lambda}(nak):=\Phi(k)e^{(i\lambda e_1+\rho)\log{a}}.
\end{align*}
Then for $\Iim(\lambda)<0$ the integral
\begin{align}\label{Verkettungsop als Integral}
J_{\bar{P}_{0}|P_{0}}(\sigma,\lambda)(\Phi)(k):=
\int_{\bar{N}}{\Phi_{\lambda}(\bar{n}k)d\bar{n}},
\end{align}
is convergent and
$J_{\bar{P}_{0}|P_{0}}(\sigma,\lambda):(\mathcal{H}^{\sigma})^\infty
\longrightarrow (\mathcal{H}^{\sigma})^\infty$ defines an intertwining operator
between $\pi_{\sigma,\lambda}$ and $\pi_{\sigma,\lambda,\bar{P}_{0}}$,
where $\pi_{\sigma,\lambda,\bar{P}_{0}}$ denotes the principal series
representation associated to $\sigma$, $\lambda$ and $\bar{P}_{0}$.
As an operator-valued function,
$J_{\bar{P}_{0}|P_{0}}(\sigma,\lambda)$ has a meromorphic continuation to
$\C$ (see \cite{Knapp Stein}). Let $\nu\in\hat{K}$ be a $K$-type of
$\pi_{\sigma,\lambda}$. Since
$[\nu:\sigma]\leq 1$ for every $\nu\in\hat{K}$, it follows from Frobenius
reciprocity and Schur's lemma
that 
\begin{align}\label{Definition c-Funktion}
J_{\bar{P}_{0}|P_{0}}(\sigma,\lambda)|_{(H^{\sigma})^{\nu}}=c_{\nu}
(\sigma:\lambda)\cdot \Id,
\end{align}
where $c_{\nu}(\sigma:\lambda)\in\C$. 
The function $z\mapsto
c_{\nu}(\sigma:z)$ can be computed explicitly. Assume that $d=2n+1$. Let 
$k_{2}(\sigma)e_{2}+\dots+k_{n+1}(\sigma)e_{n+1}$ be the highest weight
of $\sigma$ as in \eqref{Darstellungen von M} and let
$k_{2}(\nu)e_{2}+\dots+k_{n+1}(\nu)e_{n+1}$ be the highest
weight of $\nu$ as in \eqref{Darstellungen von K}. Then taking the different
parametrization into account, it follows from Theorem 8.2 in \cite{Eguchi} that
there
exists a constant $\alpha(n)$ depending on $n$ such that
\begin{align}\label{cfunktion}
c_{\nu}(\sigma:z)=\alpha(n)\frac{\prod_{j=2}^{n+1}\Gamma(iz-k_{j}(\sigma)-
\rho_{j})\prod_{j=2}^{n+1}\Gamma(iz+k_{j}(\sigma)+\rho_{j})}{
\prod_{j=2}^{n+1}\Gamma(iz-k_{j}(\nu)-\rho_{j})\prod_{j=2}^{n+1}
\Gamma(iz+k_{j}(\nu)+\rho_{j}+1)}.
\end{align}
This formula implies that
\begin{align}\label{Glcfnktn}
c_{\nu}(\sigma:z)^{-1}\frac{d}{dz}c_{\nu}(\sigma
:z)
=\sum_{j=2}^{n+1}\sum_{\substack{\left|k_{j}(\sigma)\right|<l\\\leq
k_{j}(\nu)}}\frac{i}{
iz-l-\rho_{j}}-\sum_{j=2}^{n+1}\sum_{l=\left|k_{j}(\sigma)\right|}^{k_{j
}(\nu)}\frac{i}{iz+l+\rho_{j}}.
\end{align}
Next let $d=2n+2$. Let 
$k_{2}(\sigma)e_{2}+\dots+k_{n+1}(\sigma)e_{n+1}$ be the highest weight
of $\sigma$ as in \eqref{Darstellungen von M'} and let
$k_{1}(\nu)e_{1}+\dots+k_{n+1}(\nu)e_{n+1}$ be the highest
weight of $\nu$ as in \eqref{Darstellungen von K'}. Then by
\cite[Theorem 8.2]{Eguchi}, there exists a constant $\alpha(n)$ depending only on
$n$ such that
\begin{align}\label{cfunktion'}
c_{\nu}(\sigma:z)=\alpha(n)\frac{\Gamma(2iz)\prod_{j=2}^{n+1}
\Gamma(iz-k_j(\sigma)-\rho_j)\prod_{j=2}^{n+1}\Gamma(iz+k_j(\sigma)+\rho_j)}
{2^{2iz} \prod_{j=1}^{n+1}
\Gamma(iz-k_j(\nu)-\rho_j+1)\prod_{j=1}^{n+1}\Gamma(iz+k_j(\nu)+\rho_j)}.
\end{align}
Equation \eqref{cfunktion} and \eqref{cfunktion'} imply that
$J_{\bar{P}_{0}|P_{0}}(\sigma,\lambda)$ has no poles on $\mathbb{R}-\{0\}$ and
is
invertible there and that $J_{\bar{P}_{0}|P_{0}}(\sigma,z)^{-1}$ is defined as a
meromorphic function of $z$. It follows that the weighted character
\begin{equation}\label{weightch}
\Tr\left(J_{\bar{P}_{0}|P_{0}}(\sigma,z)^{-1}\frac{d}{dz}
J_{\bar{P}_{0}|P_{0}}(\sigma,z)\pi_{\sigma,z}(\alpha)\right)
\end{equation}
is regular for $z\in\R-\{0\}$. Let $\epsilon>0$ be
sufficiently small. Let $H_{\epsilon}$ be the
half-circle from $-\epsilon$ to $\epsilon$ in the lower half-plane, oriented
counter-clockwise. Let $D_{\epsilon}$ be the path which is the union of
$\left(-\infty,-\epsilon\right]$, $H_{\epsilon}$ and
$\left[\epsilon,\infty\right)$. 
Using \eqref{cfunktion}, \eqref{cfunktion'} and the fact that the matrix
coefficients of 
$\pi_{\sigma,z}(\alpha)$ are rapidly decreasing, it follows that
\eqref{weightch}
is integrable over $D_{\epsilon}$. Let
\begin{equation}\label{j1}
J_{\sigma}(\alpha):=\frac{\kappa\dim\sigma}{4\pi
i}\int_{D_{\epsilon}}{\Tr\left(J_{\bar{P}_{0}|P_{0}}(\sigma,z)^{-1}\frac{d}{
dz}J_{\bar{P}_{0}|P_{0}}(\sigma,z)\pi_{\sigma,z}(\alpha)\right)dz}.
\end{equation}
The change of contour is only necessary if
$J_{\bar{P}_{0}|P_{0}}(\sigma,s)$ has a pole at $0$. Let
\begin{align}\label{Definition von J}
J(\alpha):=-\sum_{\sigma\in\hat{M}}J_{\sigma}(\alpha).
\end{align}
By \eqref{propH} and Proposition \ref{fouriertrf2} one has
\begin{align}\label{reprJ}
J(h^\nu_t)=-\frac{\kappa}{4\pi
i}\sum_{\sigma\in\hat{M}}\left[\nu:\sigma\right]\dim
(\sigma)\int_{D_{\epsilon}}e^{-t(z^2-c(\sigma))}c_{\nu}(\sigma:z)^{-1}
\frac{d}{dz}c_{\nu}(\sigma:z)dz.
\end{align}
For notational convenience, if $\nu\in\hat{K}$ and $\sigma\in\hat{M}$
with $\left[\nu:\sigma\right]=0$ we let $c_{\nu}(\sigma:z):= 0$.
Now we define a distribution $\mathcal{I}$ by
\begin{align}\label{Definition invariant parabolisch}
&\mathcal{I}(\alpha):=\sum_{P\in\mathfrak{P}}T_{P}'(\alpha)-J(\alpha).
\end{align}
We claim that $\mathcal{I}$ is an invariant distribution. This can be seen as 
follows. Using the formula for $J_M(m,\alpha)$ on p. 92 of \cite{Hoffmann}, 
we get $J_{M_P}(1,\alpha)=T^\prime_P(\alpha)$. Next using the
formula for the invariant distribution $I_P(m,\alpha)$ on p. 93 of 
\cite{Hoffmann} and formula (8) of \cite{Hoffmann}, it follows that 
\[
I_P(1,\alpha)=T^\prime_P(\alpha)+\sum_{\sigma\in \hat
M_0}\frac{\dim(\sigma)}{4\pi i}
\int_{D_\epsilon}\Tr\left(J_{\bar{P}_{0}|P_{0}}(\sigma,z)^{-1}\frac{d}{
dz}J_{\bar{P}_{0}|P_{0}}(\sigma,z)\pi_{\sigma,z}(\alpha)\right)dz.
\]
Summing over $P\in\mathfrak{P}$, we get
\[
\sum_{P\in\mathfrak{P}}I_P(1,\alpha)=\mathcal{I}(\alpha)-J(\alpha),
\]
which proves our claim.
 
\begin{thrm}\label{Spurf}
With the above notations, one has
\begin{align*}
\Tr_{\rel}(e^{-tA_{\nu}})=I(h^{\nu}_{t})+H(h^{\nu}_{t})+T(\tilde{
h}^{\nu}_{t})+\mathcal{I}(h^{\nu}_{t})+J(h^{\nu}_{t}).
\end{align*}
\end{thrm}
\begin{proof}
By \eqref{regtrace2}, $\Tr_{\rel}\left(e^{-tA_\nu}\right)$ is the difference of 
$\Tr\left(e^{-tA_\nu^d}\right)$ and the terms in the trace formula which are
associated to the continuous spectrum. These are the last two terms in the
trace formula \cite[Theorem 8.4]{Warner}. Using \cite[Theorem 8.4]{Warner}, 
the Theorem on page 299 in \cite{Osborne}, and taking our normalization of 
measures into account, we obtain the claimed equality.
\end{proof}
The Fourier transform of the distribution $\mathcal{I}$ was computed in
\cite{Hoffmann}.
We shall now state his result.
For $\sigma\in\hat{M}$ with highest weight
$k_{2}(\sigma)e_{2}+\dots+k_{n+1}(\sigma)e_{n+1}$ and
$\lambda\in\mathbb{R}$ define
$\lambda_{\sigma}\in(\mathfrak{h})_{\C}^{*}$ by 
\begin{align*}
\lambda_{\sigma}:=i\lambda
e_{1}+\sum_{j=2}^{n+1}(k_{j}(\sigma)+\rho_{j})e_{j}. 
\end{align*}
Let $S(\mathfrak{b}_{\C})$ be the symmetric algebra of $\mathfrak{b}_{\C}$. 
Define $\Pi\in S(\mathfrak{b}_{\C})$ by
\begin{align}\label{Def Pi}
\Pi:=\prod_{\alpha\in\Delta^{+}(\mathfrak{m}_{\C},\mathfrak{b}_{\C})}
H_{\alpha}.
\end{align}
The restriction of the Killing form to
$\mathfrak{h}_{\C}$ defines a non-degenerate symmetric bilinear
form. We will identify $\mathfrak{h}_{\C}^{*}$ with
$\mathfrak{h}_{\C}$ via this form and denote the induced symmetric
bilinear form on $\mathfrak{h}_{\C}^{*}$ by
$\left<\cdot,\cdot\right>$. Then for
$\alpha\in\Delta^{+}(\mathfrak{g}_{\C},\mathfrak{h}_{\C})$
we denote by
$s_{\alpha}:\mathfrak{h}_{\C}^{*}\rightarrow \mathfrak{h}_{
\C}^{*}$ the reflection
$s_{\alpha}(x)=x-2\frac{\left<x,\alpha\right>}{\left<\alpha,\alpha\right>}
\alpha$. 
Now the Fourier transform of $\mathcal{I}$ is computed as follows. 
\begin{thrm}\label{FTorbint}
For every $K$-finite $\alpha\in C^{2}(G)$ one has
\begin{align*}
\mathcal{I}(\alpha)=\frac{\kappa}{4\pi}\sum_{\sigma\in\hat{M}}\int_{\mathbb{R
}}{\Omega(\check{\sigma},-\lambda)\Theta_{\sigma,\lambda}(\alpha)d\lambda},
\end{align*}
where
\begin{align*}
\Omega(\sigma,\lambda):=-2\dim(\sigma)\gamma-\frac{1}{2}
\sum_{\alpha\in\Delta^{+}(\gL_\C,\aL_\C)}\frac{\Pi(s_{\alpha}\lambda_{\sigma})
}{\Pi(\rho_{M})}\left(\psi(1+\lambda_{\sigma}(H_{\alpha}))+\psi(1-\lambda_{
\sigma}(H_{\alpha}))\right).
\end{align*}
Here $\psi$ denotes the digamma function and $\gamma$ denotes the
Euler-Mascheroni constant. Moreover $\check{\sigma}$ denotes the contragredient
representation of $\sigma$ and $\Pi$ is as in \eqref{Def Pi}. 
\end{thrm}
\begin{proof}
This follows from \cite[Theorem 5]{Hoffmann}, \cite[Theorem 6]{Hoffmann},
\cite[Corollary on page 96]{Hoffmann}. Here we use that for $d$ even and
$\pi\in\hat G_d$, the 
discrete series of $G$, the
term ${\left|D_G(a)\right|}^{1/2}\Theta_{\check{\pi}}(a)$ occurring in
\cite[Theorem 5]{Hoffmann} vanishes for $a=1$. This can be seen as follows. 
By the formula for the character of the discrete series 
\cite[Theorem 12.7]{Knapp}, \cite[Theorem 10.1.1.1]{Wa2}, one needs to show
that $\sum_{w\in W_K}\det(w)=0$. This has been established in the proof
of Lemma 5 in \cite{DeGeorge}.
\end{proof}
For the applications we have in mind, we shall now
transform the functions $\Omega(\lambda,\sigma)$ a bit. In the rest
of this section we assume that $d=\dim(X)$ is odd, $d=2n+1$. We start
with the following elementary lemma.
\begin{lem}\label{Lemomega}
One has
\begin{align*}
\sum_{\alpha\in\Delta^{+}(\gL_\C,\aL_\C)}\frac{\Pi(s_{\alpha}\lambda_{\sigma})}{
\Pi(\rho_{M})}=2\dim{\sigma}.
\end{align*}
\end{lem}
\begin{proof}
This is proved in \cite[page 95]{Hoffmann} but can also be seen as follows.
Let $\xi\in\mathfrak{b}_{\C}^{*}$, $\xi=\xi_{2}e_{2}+\dots+\xi_{n+1}e_{n+1}$. 
Then it follows from \eqref{Explizite Formel fur H(alpha)} that
\begin{align}\label{Explizite Form fur pi}
\Pi(\xi)=\prod_{2\leq i<j\leq n+1}(\xi_{i}-\xi_{j})(\xi_{i}+\xi_{j}).
\end{align}
If $\tau$ is a permutation of $\{2,\dots,n+1\}$ and 
\begin{align*}
\xi_{\tau}:=\xi_{2}e_{\tau(2)}+\dots+\xi_{n+1}e_{\tau(n+1)}
\end{align*}
it follows from \eqref{Explizite Form fur pi} that
\begin{align}\label{Permutationsinvarianz}
\Pi(\xi_{\tau})=\pm\Pi(\xi).
\end{align}
Write
$\Lambda(\sigma)+\rho_{M}=\xi_{2}e_{2}+\dots+\xi_{n+1}e_{n+1}$. Then if
$\alpha=e_{1}\pm e_{j}$, one has
\begin{align}\label{Spiegelung}
s_{\alpha}(\lambda_{\sigma})=\mp
\xi_{j}e_{1}+\xi_{2}e_{2}+\dots+\xi_{j-1}e_{j-1}\mp i\lambda
e_{j}+\xi_{j+1}e_{j+1}+\dots+\xi_{n+1}e_{n+1}.
\end{align}
Using \eqref{wsigma} and \eqref{Explizite Form fur pi} it follows that
\begin{align}\label{GlPi}
\Pi(s_{e_{1}+e_{j}}(\lambda_{\sigma}))=\Pi(s_{e_{1}-e_{j}}(\lambda_{\sigma}
));\quad
\Pi(s_{e_{1}+e_{j}}(\lambda_{\sigma}))=\Pi(s_{e_{1}+e_{j}}(\lambda_{w_{0}\sigma}
)).
\end{align}
Thus by \eqref{Dimension sigma} and \eqref{Explizite Form fur pi} for
$\alpha=e_1\pm e_j$ one gets
\begin{align}\label{Glpj}
\frac{\Pi(s_{\alpha}(\lambda_{\sigma}))}{\Pi(\rho_{M})}=&\frac{(-1)^{j}}{
\Pi(\rho_M)}\prod_{\substack{2\leq k<l\leq n+1\\ k,l\neq
j}}(\xi_{k}^2-\xi_{l}^2)\prod_{\substack{p=2\\ p\neq
j}}^{n+1}\left(-\lambda^2-\xi_{p}^2\right)
\nonumber\\=&\frac{1}{\Pi(\rho_M)}\prod_{2\leq
k<l\leq
n+1}(\xi_{k}^2-\xi_{l}^2)\prod_{\substack{p=2\\ p\neq
j}}^{n+1}\frac{-\lambda^2-\xi_{p}^2}{\xi_{j}^2-\xi_{p}^2}\nonumber\\
=&\dim(\sigma)\prod_{\substack{p=2\\ p\neq
j}}^{n+1}\frac{-\lambda^2-\xi_{p}^2}{\xi_{j}^2-\xi_{p}^2}
\end{align}
Now as in \cite[Lemma 5.6]{MP} one has
\begin{align*}
\sum_{j=2}^{n+1}\prod_{\substack{p=2\\ p\neq
j}}^{n+1}\frac{-\lambda^2-\xi_{p}^2}{\xi_{j}^2-\xi_{p}^2}=1
\end{align*}
for every $\lambda$. This proves the lemma.
\end{proof}
For
$j=2,\dots,n+1$ and $\lambda\in\C$ let
\begin{align}\label{Defpoly}
P_{j}(\sigma,\lambda):=\frac{\Pi(s_{e_{1}+e_{j}}\lambda_{\sigma})}{\Pi(\rho_{M}
)}.
\end{align}
Then if $\sigma$ is of highest weight
$k_2(\sigma)e_2+\dots+k_{n+1}(\sigma)e_{n+1}$ as
in \eqref{Darstellungen von M} it follows from \eqref{Glpj}
that
\begin{align}\label{formpj}
P_j(\sigma,\lambda)=\dim(\sigma)\prod_{\substack{p=2\\ p\neq
j}}^{n+1}\frac{-\lambda^2-(k_p(\sigma)+\rho_p)^2}{
(k_j(\sigma)+\rho_j)^2-(k_p(\sigma)+\rho_p)^2}.
\end{align}
In particular
$P_{j}(\sigma,\lambda)$ is an even polynomial in $\lambda$ of degree $2n-2$.
\begin{prop}\label{propomega}
Let $\sigma\in\hat{M}$ be of highest weight
$k_{2}(\sigma)e_2+\dots+k_{n+1}(\sigma)e_{n+1}$. Assume that all $k_j(\sigma)$
are
integral an that $k_{n+1}(\sigma)>0$.
Then one has
\begin{align*}
\Omega(\sigma,\lambda)=\Omega(w_0\sigma,\lambda);\quad
\Omega(\sigma,\lambda)=\Omega(\check{\sigma},-\lambda).
\end{align*}
Moreover one can write
\begin{align*}
\Omega(\sigma,\lambda)=\Omega_1(\sigma,\lambda)+\Omega_2(\sigma,\lambda),
\end{align*}
where $\Omega_1(\sigma,\lambda)$ and $\Omega_2(\sigma,\lambda)$ are defined as
follows. Let $m_0:=\left|k_{n+1}(\sigma)\right|-1$.
Then one puts
\begin{align*}
\Omega_{1}(\sigma,\lambda):=-\dim{\left(\sigma\right)}
\left(2\gamma+\psi(1+i\lambda)+\psi(1-i\lambda)+\sum_{1\leq l\leq
m_0}\frac{2l}{
l^2+\lambda^2}\right).
\end{align*}
Furthermore for every $j$ let $P_j(\sigma,\lambda)$ be as in \eqref{Defpoly}.
For $m_0\leq l\leq k_j(\sigma)+\rho_j$ define an even polynomial
$Q_{j,l}(\sigma,\lambda)$ by
\begin{align}\label{DefQ}
Q_{j,l}(\sigma,\lambda):=\frac{P_j(\sigma,\lambda)-P_j(\sigma,il)}{
l+i\lambda}
+\frac{P_j(\sigma,\lambda)-P_j(\sigma,il)}{l-i\lambda}.
\end{align}
Then 
\begin{align*}
\Omega_2(\sigma,\lambda):=&-\sum_{j=2}^{n+1}\sum_{\substack{m_0<
l<\\k_{j}(\sigma)+\rho_{j}}}P_j(\sigma,il)\frac{2l}{\lambda^2+l^2}
-\sum_{j=2}^{n+1}\dim(\sigma)\frac{k_j(\sigma)+\rho_j}{
(k_j(\sigma)+\rho_j)^2+\lambda^2}\\
&-\sum_{j=2}^{n+1}\sum_{\substack{m_0<
l<\\k_{j}(\sigma)+\rho_{j}}}Q_{j,l}(\sigma,\lambda)-\frac{1}{2}\sum_{\substack{
l=k_j(\sigma)+\rho_j\\ 2\leq j\leq n+1}}Q_{j,l}(\sigma,\lambda).
\end{align*}
Finally, if $k_{n+1}(\sigma)<0$, one puts
$\Omega_1(\sigma,\lambda)=\Omega_1(w_0\sigma,\lambda)$, 
$\Omega_2(\sigma,\lambda)=\Omega_2(w_0\sigma,\lambda)$.
\end{prop}

\begin{proof}
Let $j\in\{2,\dots,n+1\}$. We have
\begin{align}\label{GlPi2}
\lambda_{\sigma}(H_{e_{1}\pm e_{j}})=i\lambda \pm
\left(k_{j}(\sigma)+\rho_{j}\right).
\end{align}
Now recall that $\rho_{n+1}=0$ and that the highest weight of $w_{0}\sigma$ is
given by $k_{2}(\sigma)e_{2}+\dots+k_{n}(\sigma)e_{n}-k_{n+1}(\sigma)e_{n+1}$.
Moreover recall that for $M=\Spin(n)$ one has $\check{\sigma}\cong\sigma$ if n
is odd and $\check{\sigma}\cong w_0\sigma$ if $n$ is even.
Thus \eqref{GlPi} and \eqref{GlPi2} imply that
$\Omega(\lambda,\sigma)=\Omega(\lambda,w_{0}\sigma)$ and
$\Omega(\lambda,\sigma)=\Omega(-\lambda,\check{\sigma})$.
Using these equations, we can assume that $k_{n+1}(\sigma)>0$.
Moreover, using $\psi(z+1)=\frac{1}{z}+\psi(z)$ , \eqref{GlPi} and \eqref{GlPi2}
we obtain
\begin{align*}
&\frac{\Pi(s_{e_{1}+e_{j}}\lambda_{\sigma})}{\Pi(\rho_{M})}
\left(\psi(1+\lambda_{\sigma}(H_{e_{1}+e_{j}}))+\psi(1-\lambda_{\sigma}
(H_{e_{1}+e_{j}}))\right)\nonumber\\
+&\frac{\Pi(s_{e_{1}-e_{j}}\lambda_{\sigma})}{\Pi(\rho_{M})}
\left(\psi(1+\lambda_{\sigma}(H_{e_{1}-e_{j}}))+\psi(1-\lambda_{\sigma}
(H_{e_{1}-e_{j}}))\right)\nonumber\\
=&2\frac{\Pi(s_{e_{1}+e_{j}}\lambda_{\sigma})}{\Pi(\rho_{M})}
\biggl(\psi(1+i\lambda)+\psi(1-i\lambda)
+\sum_{1\leq l\leq
m_0}\frac{2l}{l^2+\lambda^2}\\&+\sum_{m_0<l<k_{j}(\sigma)+\rho_
{j}
}\frac{2l}{l^2+\lambda^2}+\frac{\left(k_{j}(\sigma)+\rho_{j}\right)}{\left(k_{
j}(\sigma)+\rho_{j}\right)^2+\lambda^{2}}\biggr).
\end{align*}
Using Lemma \ref{Lemomega} and \eqref{GlPi} we obtain
\begin{align*}
\Omega(\sigma,\lambda)=\Omega_1(\sigma,\lambda)-\sum_{j=2}^{n+1}P_{j}(\sigma,
\lambda)\left(\sum_{
\substack{
m_0< l\\ < k_{j}(\sigma)+\rho_{j}}}\frac{2l}{l^2+\lambda^2}+\frac{\left(k_{j}
(\sigma)+\rho_{j}\right)}{\left(k_{j}(\sigma)+\rho_{j}\right)^2+\lambda^{2}}
\right).
\end{align*}
Since $P_j(\sigma,\lambda)$ is an even polynomial in $\lambda$, for every 
$j=2,\dots,n+1$ and every $l$ with $m_0\leq l\leq
\left|k_{j}(\sigma)\right|+\rho_{j}$ we can write
\begin{align*}
\frac{P_j(\sigma,\lambda)l}{l^2+\lambda^2}=\frac{1}{2}
Q_{j,l}(\sigma,\lambda)+P_j(\sigma,il)\frac{l}{l^2+\lambda^2}.
\end{align*}
Using \eqref{formpj} it follows that
\begin{align*}
P_j(\sigma,i\left(k_j(\sigma)+\rho_j\right))=\dim(\sigma).
\end{align*}
This implies the proposition. 
\end{proof}
\begin{bmrk}\label{halfint}
There is a similar formula for $\sigma\in\hat M$ with half-integer weight.
\end{bmrk}
In order to define the analytic torsion, we need to know that the 
regularized trace of $e^{-t\Delta_p(\tau)}$ admits an asymptotic expansion
as $t\to+0$. We establish this in general for the operators $e^{-tA_\nu}$.
To begin with, we prove some auxiliary lemmas. 
\begin{lem}\label{Lemas1}
Let
$\phi_1(t):=\int_{\mathbb{R}}e^{-t\lambda^2}\frac{1}{\lambda^2+c^2}d\lambda$.
Then there exist $a_j\in\C$ such that
\begin{align*}
\phi_1(t)\sim \sum_{j=0}^\infty a_jt^{\frac{j}{2}}.
\end{align*}
as $t\rightarrow 0$.
\end{lem}
\begin{proof}
We have
\begin{align*}
\phi_1(t)=&e^{tc^2}\int_{\mathbb{R}}\frac{e^{-t\left(\lambda^2+c^2\right)}}{
\lambda^2+c^2}d\lambda.
\end{align*}
One has
\begin{align*}
\frac{d}{dt}\int_\mathbb{R}\frac{e^{-t\left(\lambda^2+c^2\right)}}{\lambda^2+c^2
}d\lambda=-\frac{\sqrt{\pi}}{\sqrt{t}}.
\end{align*}
Thus one has
\begin{align*}
\int_\mathbb{R}\frac{e^{-t\left(\lambda^2+c^2\right)}}{\lambda^2+c^2}
d\lambda=C+\sqrt{\pi t}.
\end{align*}
Writing $e^{tc^2}$ as a power series, the proposition follows. 
\end{proof}

\begin{lem}\label{Lemas2}
Let $\phi_2(t):=\int_{\mathbb{R}}{e^{-t\lambda^2}\psi(1+i\lambda)d\lambda}$.
Then
there exist complex coefficients $a_j$, $b_j$, $c_j$ such that as $t\to 0$,
there is an asymptotic expansion 
\begin{align*}
\phi_2(t)\sim \sum_{j=0}^\infty a_jt^{j-1/2}+\sum_{j=0}^\infty b_j t^{j-1/2}\log{t}
+\sum_{j=0}^\infty c_jt^{j}.
\end{align*}
\end{lem}
\begin{proof}
The asymptotic behavior of the Laplace transform at 0 of functions which
admit suitable asymptotic expansions at infinity has been treated in
\cite{HL}.\\
Recall that
\begin{align}\label{aspsi}
\psi(z+1)=\log{z}+\frac{1}{2z}-\sum_{k=1}^{N}\frac{B_{2k}}{2k}\cdot\frac{1}{
z^{2k}}+R_{N}(z),\:N\in\mathbb{N},
\end{align}
where $B_{i}$ are the Bernoulli-numbers and
\begin{align*}
R_{N}(z)=O(z^{-2N-2}),\:z\rightarrow\infty.
\end{align*}
uniformly on sectors $-\pi+\delta<\Arg(z)<\pi-\delta$.
Consider
\begin{align*}
\phi_2^+(t):=\int_{0}^\infty{e^{-t\lambda^2}\psi(1+i\lambda)d\lambda}.
\end{align*}
Let $\chi$ be the characteristic function of $\left[1,\infty\right)$. 
Define a function 
\begin{align*}
g(\lambda):=\psi(1+i\lambda)-\log({i\lambda})-\frac{\chi(\lambda)}{2i\lambda}
\end{align*}
and define a function
\begin{align*}
h(\lambda):=\frac{g(\sqrt{\lambda})}{2\sqrt{\lambda}}.
\end{align*}
Then by \eqref{aspsi} there is an asymptotic expansion
\begin{align}\label{ash}
h(\lambda)\sim\sum_{k=1}^{\infty}a_k\lambda^{-k-1/2},\quad \lambda\to\infty.
\end{align}
First define
\begin{align*}
\psi_2^+(t):=\int_0^\infty e^{-t\lambda^2}g(\lambda)d\lambda=\int_0^\infty
e^{-t\lambda}h(\lambda)d\lambda.
\end{align*}
Then by \eqref{ash} and \cite[Corollary 5.2]{HL} one obtains
\begin{align*}
\psi_2^+(t)\sim \sum_{k=0}^\infty a_k' t^{k+1/2}+\sum_{k=0}^\infty c_k' t^k
\end{align*}
for complex $a_k'$, $c_k'$.
Next we have
\begin{align*}
\int_{0}^\infty e^{-t\lambda^2}\log{\lambda}\:d\lambda=t^{-1/2}\int_0^\infty
e^{-\lambda^2}\log{\lambda}d\lambda-\frac{\sqrt{\pi}}{4
}t^{-1/2}\log{t}.
\end{align*}
Finally we have
\begin{align*}
\int_1^\infty
e^{-t\lambda^2}\lambda^{-1}d\lambda=&\int_{\sqrt{t}}^1e^{-\lambda^2}\lambda^{-1}
d\lambda+ \int_1^\infty e^{-\lambda^2}\lambda^{-1}d\lambda\\
=&\int_{\sqrt{t}}^1 \sum_{k=0}^\infty
(-1)^k\frac{\lambda^{2k-1}}{k!}d\lambda+C\\
=&-\log{\sqrt{t}}+\sum_{k=1}^\infty (-1)^{k+1}\frac{t^{k}}{k!2k}+C'.
\end{align*}
Putting everything together, we obtain the desired asymptotic expansion for
$\phi_2^+$. For the integral over $(-\infty,0]$  we proceed similarly. \\
Alternatively, one can also proceed as in \cite[page 156-157, page
165-166]{Koyama}. 
The methods of \cite{HL} and \cite{Koyama} are closely related.
\end{proof}

\begin{lem}\label{Lemas3}
Let $P(z):=\sum_{j=0}^N a_j z^{2j}$ be an even polynomial. Then there exist
$a_j'\in\C$ such that
\begin{align*}
\int_{\mathbb{R}}e^{-t\lambda^2}P(\lambda)d\lambda=\sum_{j=0}^N a_j'
t^{-j-\frac{1}{2}}.
\end{align*}
\end{lem}
\begin{proof}
This follows by a change of variables.
\end{proof}
\begin{prop}\label{AsEnt}
Assume that $\dim(X)$ is odd.
There exist coefficients $a_{j}$, $b_{j}$, $c_j$, $j\in\mathbb{N}$, such
that one has
\begin{equation}\label{asympexp11}
\Tr_{\rel}(e^{-tA_{\nu}})\sim
\sum_{j=0}^\infty a_{j}t^{j-\frac{d}{2}}+\sum_{j=0}^\infty b_{j}t^{
j-\frac{1}{2}}\log{t}+\sum_{j=0}^\infty c_j t^j
\end{equation}
as $t\to+0$. 
\end{prop}
\begin{proof}
We use Theorem \ref{Spurf} and derive an asymptotic expansion of each term
on the right hand side.  We can
always ignore additional factors of the form $e^{-tc}$, $c>0$ by expanding this
term in a power series.
The term $I(h^\nu_t)$ has the desired asymptotic expansion by Proposition
\ref{fouriertrf2}, equation
\eqref{Idcontr} and Lemma \ref{Lemas3} . 
Second using \cite[ Proposition 5.4]{GaWa} one obtains
$H(h^\nu_t)=O(e^{-\frac{c}{t}})$ for a constant $c>0$.
By Proposition \ref{fouriertrf2} and equation \eqref{Fouriertrafo T}, the term
$T(h^\nu_t)$ has an asymptotic expansion
starting with $t^{-\frac{1}{2}}$. 
For every $\sigma\in\hat{M}$ with $\left[\nu:\sigma\right]\neq 0$ we write
$\Omega(\lambda,\sigma)$ as in Proposition \ref{propomega}. 
Then by Proposition \ref{fouriertrf2}, Proposition \ref{propomega} together
with remark \ref{halfint}, Lemma
\ref{Lemas1}, Lemma \ref{Lemas2} and Lemma \ref{Lemas3} it follows
that the term $\mathcal{I}(h^\nu_t)$ has the claimed asymptotic expansion in
$t$. The term $J(h^\nu_t)$ has the claimed asymptotic expansion by
equation \eqref{reprJ} and Lemma \ref{Lemas1}.
\end{proof}
\begin{bmrk}
The proposition remains true in even dimensions. The proof, however,
 would require more work due to the discrete series. This is not needed for our
purpose. 
\end{bmrk}
\begin{bmrk}
The asymptotic expansion \eqref{asympexp11} has been established under the
assumption that $\Gamma$ satisfies \eqref{a1}. We want to point out that it is 
expected to hold in much greater generality. For example, Albin and Rochon 
\cite[Theorem A.1]{AR} have shown that for a Dirac type operator 
$\slashed \partial$ on a manifold with fibred cusps, the regularized trace
of the heat operator $e^{-t\slashed{\partial}^2}$ admits an asymptotic expansion as 
$t\to 0$. The method can be modified so that it works for the Laplacian on 
$p$-forms twisted by a flat bundle. 
\end{bmrk}

\section{The analytic torsion}\label{sectors}
\setcounter{equation}{0}
Let $\tau$ be an irreducible finite dimensional representation of $G$ on
$V_{\tau}$. Let $E'_{\tau}$ be the flat
vector bundle associated to the
restriction of $\tau$ to $\Gamma$. Then $E'_{\tau}$ is canonically isomorphic to
the locally homogeneous vector bundle $E_{\tau}$ associated to $\tau|_{K}$. By
\cite{MM}, there exists an inner product $\left<\cdot,\cdot\right>$ on
$V_{\tau}$ such that
\begin{enumerate}
\item $\left<\tau(Y)u,v\right>=-\left<u,\tau(Y)v\right>$ for all
$Y\in\mathfrak{k}$, $u,v\in V_{\tau}$
\item $\left<\tau(Y)u,v\right>=\left<u,\tau(Y)v\right>$ for all
$Y\in\mathfrak{p}$, $u,v\in V_{\tau}$.
\end{enumerate}
Such an inner product is called admissible. It is unique up to scaling. Fix an
admissible inner product. Since $\tau|_{K}$ is unitary with respect to this
inner product, it induces a metric on $E_{\tau}$ which will be called admissible
too. Let $\Lambda^{p}(E_{\tau})$ be the bundle of $E_{\tau}$ valued p-forms on
$X$. Let 
\begin{equation}\label{nutau}
\nu_{p}(\tau):=\Lambda^{p}\Ad^{*}\otimes\tau:\:K\rightarrow\Gl(\Lambda^{p}
\mathfrak{p}^{*}\otimes V_{\tau}).
\end{equation}
There is a canonical isomorphism
\begin{align}\label{p Formen als homogenes Buendel}
\Lambda^{p}(E_{\tau})\cong\Gamma\backslash(G\times_{\nu_{p}(\tau)}(\Lambda^{p}
\mathfrak{p}^{*}\otimes V_{\tau})).
\end{align}
If $\Lambda^{p}(X,E_{\tau})$ are the smooth $E_{\tau}$-valued $p$-forms on $X$,
the isomorphism \eqref{p Formen als homogenes Buendel} induces an isomorphism
\begin{align}\label{Beschreibung der Schnitte}
\Lambda^{p}(X,E_{\tau})\cong C^{\infty}(\Gamma\backslash G,\nu_{p}(\tau)),
\end{align}
A corresponding isomorphism also holds for the $L^{2}$-spaces. Let
$\Delta_{p}(\tau)$ be the Hodge-Laplacian on $\Lambda^{p}(X,E_{\tau})$ with
respect to the admissible inner product. By (6.9) in \cite{MM}, on
$C^{\infty}(\Gamma\backslash G,\nu_{p}(\tau))$ one has
\begin{align}\label{kuga}
\Delta_{p}(\tau)=-\Omega+\tau(\Omega)\Id.
\end{align}
If $\Lambda(\tau)=k_1(\tau)e_1+\dots k_{n+1}(\tau)e_{n+1}$ is the highest weight
of $\tau$, we have
\begin{equation}\label{tauomega}
\tau(\Omega)=\sum_{j=1}^{n+1}\left(k_j(\tau)+\rho_j\right)^2-\sum_{j=1}^{n+1}
\rho_j^2.
\end{equation} 
For $G=\Spin(2n+1,1)$ this was proved in \cite[sect. 2]{MP}. For
$G=\Spin(2n+2,1)$,
one can proceed in the same way.
Let $0\le\lambda_1\le\lambda_2\le
\cdots$ be the eigenvalues of $\Delta_p(\tau)$. By \eqref{kuga} and
\eqref{regtrace2} we have
\begin{equation}\label{regtrace3}
\begin{split}
\Tr_{\rel}\left(e^{-t\Delta_p(\tau)}\right)&=\sum_{j}e^{-t\lambda_j}+
\sum_{\substack{\sigma\in\hat{M};\sigma=w_0\sigma\\
\left[\nu_p(\tau):\sigma\right]\neq 0}}
e^{-t(\tau(\Omega)-c(\sigma))}\frac{\Tr(\widetilde{\boldsymbol{C}}(\sigma,
\nu_p(\tau),0))}{4
}\\
&-\frac{1}{4\pi}\sum_{\substack{\sigma\in\hat{M}\\
\left[\nu_p(\tau):\sigma\right]\neq
0}}e^{-t(\tau(\Omega)-c(\sigma))}\\
&\hskip20pt\cdot\int_{\R}e^{-t\lambda^2}
\Tr\left(\widetilde{\boldsymbol{C}}(\sigma,\nu_p(\tau),-i\lambda)\frac{d}{dz}
\widetilde{\boldsymbol{C}}(\sigma,\nu_p(\tau),i\lambda)\right)\,d\lambda.
\end{split}
\end{equation}
Let
\begin{equation}\label{ktau5}
K(t,\tau):=\sum_{p=0}^{d}(-1)^{p}p\Tr_{\rel}(e^{-t\Delta_{p}(\tau)}).
\end{equation}
Then the analytic torsion is defined in terms of the Mellin transform of 
$K(t,\tau)$. 
For every $p=0,\dots,d$, let $\nu_p(\tau)$ be the representation
\eqref{nutau} and let $h_t^{\nu_p(\tau)}$
be defined by \eqref{Deffh}. Put
\begin{equation}\label{kernel5}
k_t^{\tau}:=e^{-t\tau(\Omega)}\sum_{p=0}^{d}(-1)^{p}ph_t^{\nu_p(\tau)}.
\end{equation}
By Theorem \ref{Spurf} we have
\begin{align}\label{FTK}
K(t,\tau)=I(k_t^\tau)+H(k_t^\tau)+T(k_t^\tau)+\mathcal{I}
(k_t^\tau)+J(k_t^\tau).
\end{align}
This equality will be used in section \ref{Beweis} to study the Mellin 
transform of $K(t,\tau)$. 

To define the analytic torsion, we need to determine 
the asymptotic behavior of the regularized trace of $e^{-t\Delta_p(\tau)}$ as 
$t\to\infty$. 
To begin with we estimate the exponential factors occurring on the right hand
side of \eqref{regtrace3}.
\begin{lem}\label{lowerbd1} 
\begin{enumerate}
\item \label{dfkd} Let $G=\Spin(2n+2,1)$. Let $\tau$ be an irreducible
representation
of $G$. Then
\[
\tau(\Omega)-c(\sigma)\ge \frac{1}{4}
\]
for all $\sigma\in\hat M$ with $[\nu_p(\tau)\colon\sigma]\neq0$.
\item\label{erer}
Let $G=\Spin(2n+1,1)$.
Let $\tau$ be an irreducible representation of
$G$ with highest weight $\tau_1e_1+\cdots+\tau_{n+1}e_{n+1}$ as in
\eqref{Darstellungen von G}. Then
\[
\tau(\Omega)-c(\sigma)\ge \tau_{n+1}^2
\]
for all $\sigma\in\hat M$ with $[\nu_p(\tau)\colon\sigma]\neq0$. Moreover
assume that $\sigma\in\hat{M}$ is such that $[\nu_p(\tau)\colon\sigma]\neq0$
and such that $\sigma=w_0\sigma$. Then one has
\[
 \tau(\Omega)-c(\sigma)\ge \left(\tau_n+1\right)^2+\tau_{n+1}^2\geq
1+\tau_{n+1}^2.
\]
\end{enumerate}
\end{lem}
\begin{proof}
For $p=0,...,d$ let 
\[
\nu_p:=\Lambda^p\Ad^*_\pL\colon K\to \Gl(\Lambda^p\pL^*).
\]
Recall that $\nu_p(\tau)=\tau|_{K}\otimes\nu_p$. 
Let $\nu\in\hat{K}$ with
$\left[\nu_p(\tau):\nu\right]\neq 0$. Then by \cite[Proposition
9.72]{Knapp},
there exists a $\nu^\prime\in\hat{K}$ with
$\left[\tau:\nu^\prime\right]\neq 0$ of 
highest weight $\Lambda(\nu^\prime)\in\mathfrak{b}_{\C}^*$ and a
$\mu\in\mathfrak{b}_{\C}^*$
which is a weight of $\nu_p$ such that the highest weight $\Lambda(\nu)$ of
$\nu$ is given
by $\mu+\Lambda(\nu^\prime)$. 
Now let $\nu^\prime\in\hat{K}$ be such that
$\left[\tau:\nu^\prime\right]\neq 0$. Let
$\Lambda(\nu^\prime)$ be the highest weight of $\nu^\prime$ as in
\eqref{Darstellungen von K} resp. \eqref{Darstellungen von K'}. Then by
\cite[Theorem 8.1.3, Theorem 8.1.4]{GW} we
have
\[
\tau_{j-1}\geq k_j(\nu^\prime)\geq 0,\quad j=2,\dots,n+1,
\]
if $d=2n+1$ and
\[
\tau_j\geq \left|k_j(\nu^\prime)\right|, \quad j=1,\dots,n+1,
\]
if $d=2n+2$.
Moreover, every weight 
$\mu\in\mathfrak{b}_{\C}^*$ of $\nu_p$ is given as
\begin{align*}
\mu=\pm e_{j_{1}}\pm \dots\pm e_{j_{p}},\quad j_{1}<j_2<\dots<j_{p}\leq n+1.
\end{align*}
Thus, if $\nu\in\hat{K}$ is such that $\left[\nu_p(\tau):\nu\right]\neq 0$,
the highest weight $\Lambda(\nu)$ of $\nu$, given as in 
\eqref{Darstellungen von K} resp. \eqref{Darstellungen von K'}, satisfies
\begin{align*}
\tau_{j-1}+1\geq k_j(\nu)\geq 0,\quad j\in\{2,\dots,n+1\},
\end{align*}
if $d=2n+1$ and 
\begin{align*}
\tau_{j}+1\geq \left|k_j(\nu)\right|\geq 0,\quad j\in\{1,\dots,n+1\},
\end{align*}
if $d=2n+2$. Let $\sigma\in\hat{M}$ be  such that
$\left[\nu_p(\tau):\sigma\right]\neq 0$. 
Then using \cite[Theorem 8.1.3, Theorem 8.1.4]{GW} it follows that
\begin{align*}
\tau_{j-1}+1\geq \left|k_j(\sigma)\right|
\end{align*}
for every $j\in\{2,\dots,n+1\}$, where the $k_j(\sigma)$ are as in
\eqref{Darstellungen von M} resp. \eqref{Darstellungen von M'}. Furthermore note
that by \eqref{rho} we have
$\rho_{j-1}=\rho_j+1$. Using \eqref{tauomega} and \eqref{csigma} we get
\[
c(\sigma)=\sum_{j=2}^{n+1}(k_j(\sigma)+\rho_j)^2-\sum_{j=1}^{n+1}\rho_j^2\le 
\sum_{j=2}^{n+1}(\tau_{j-1}+\rho_{j-1})^2-\sum_{j=1}^{n+1}\rho_j^2=\tau(\Omega)
-\left(\tau_{n+1}+\rho_{n+1}\right)^2.
\]
If $G=\Spin(2n+2,2)$, we have $\rho_{n+1}=1/2$ and thus item \eqref{dfkd} and
the first statement
of item \eqref{erer} are proved.\\
Now assume that $G=\Spin(2n+1,1)$. Assume that $\sigma$ additionally satisfies
$\sigma=w_0\sigma$. This is
equivalent to 
$k_{n+1}(\sigma)=0$ by \eqref{wsigma}. Thus since $\rho_{n+1}=0$, $\rho_n=1$ we
get
\[
c(\sigma)=\sum_{j=2}^{n}(k_j(\sigma)+\rho_j)^2-\sum_{j=1}^{n+1}\rho_j^2\le 
\sum_{j=2}^{n}(\tau_{j-1}+\rho_{j-1})^2-\sum_{j=1}^{n+1}\rho_j^2=\tau(\Omega)
-\left(\tau_n+1\right)^2-\tau_{n+1}^2.
\]
Finally by \eqref{Darstellungen von G} we have $\tau_n\geq 0$. This proves
the lemma.
\end{proof}

The next two lemmas are also needed to determine the  behavior of 
the regularized trace as $t\to\infty$.
\begin{lem}\label{asympexp3}
There is an asymptotic expansion
\[
\int_{\R}e^{-t\lambda^2}
\Tr\left(\widetilde{\boldsymbol{C}}(\sigma,\nu_p(\tau),-i\lambda)\frac{
d}{dz}\widetilde{\boldsymbol{C}}(\sigma,\nu_p(\tau),i\lambda)\right)d\lambda
\sim \sum_{j=1}^\infty b_jt^{-j/2}
\]
as $t\to\infty$.
\end{lem}
\begin{proof}
Since $\widetilde{\boldsymbol{C}}(\sigma:\nu_p(\tau):i\lambda)$ is analytic 
near $\lambda=0$, we have a power series expansion
\[
\Tr\left(\widetilde{\boldsymbol{C}}(\sigma,\nu_p(\tau),-i\lambda)\frac{d}{dz}
\widetilde{\boldsymbol{C}}(\sigma,\nu_p(\tau),i\lambda)\right)=
\sum_{j=0}^\infty a_j\lambda^{j}
\]
which converges for $|\lambda|\le 2\varepsilon$. Hence we get an asymptotic
expansion
\[
\int_{-\varepsilon}^{\varepsilon}e^{-t\lambda^2}
\Tr\left(\widetilde{\boldsymbol{C}}(\sigma,\nu_p(\tau),-i\lambda)\frac{d}{dz}
\widetilde{\boldsymbol{C}}(\sigma,\nu_p(\tau),i\lambda)\right)\,d\lambda
\sim\sum_{j=1}^\infty b_jt^{-j/2}.
\]
The integral over $(-\infty,-\varepsilon/2]\cup[\varepsilon/2,\infty)$ is
exponentially decreasing. This proves the lemma.
\end{proof}

\begin{lem}\label{lowerbd2}
Let $G=\Spin(2n+1,1)$. 
Let $\tau\in\hat{G}$ and assume that $\tau\neq\tau_{\theta}$. For
$p\in\{0,\dots,d\}$ let $\lambda_0\in\R^+$ be an eigenvalue of $\Delta_p(\tau)$.
Then one has $\lambda_0>1/4$.
\end{lem}
\begin{proof}
If $\tau\neq\tau_{\theta}$ one has $\left|\tau_{n+1}\right|\geq 1/2$. 
Let $\hat{G}$ be the unitary dual of $G$. 
Recall that if $\lambda_0$ is an eigenvalue of $\Delta_p(\tau)$, there exists
a $\pi\in\hat{G}$ with
$[\pi:\check{\nu}_p(\tau)]=[\pi:\nu_p(\tau)]\neq 0$
such that 
\begin{align*}
\lambda_0=-\pi(\Omega)+\tau(\Omega).
\end{align*}
Since $\rk(G)>\rk(K)$, it follows from \cite[Theorem 8.54]{Knapp} and
\cite[Corollary 6.2]{Trombi} that $\hat{G}$ consist of
the unitary principle series representations $\pi_{\sigma,\lambda}$,
$\sigma\in\hat{M}$,
$\lambda\in\R$ and the complementary series representations
$\pi^c_{\sigma,\lambda}$,
$\sigma\in\hat{M}$, $\lambda\in\R$.
First consider a unitary principle series representation $\pi_{\sigma,\lambda}$.
Then by Frobenius reciprocity \cite[page 208]{Knapp},
$[\pi_{\sigma,\lambda}:\nu_p(\tau)]$ is non zero
iff $\left[\nu_p(\tau):\sigma\right]$ is non zero. Thus together with
\eqref{infchps} and Lemma \ref{lowerbd1}, for every $\lambda\in\R$ one has 
\begin{align*}
-\pi_{\sigma,\lambda}(\Omega)+\tau(\Omega)=-c(\sigma)+\lambda^2+\tau(\Omega)\geq
1/4.
\end{align*}
Next consider a complementary series representation $\pi^c_{\sigma,\lambda}$.
Again it follows from Frobenius reciprocity that
$[\pi_{\sigma,\lambda}:\nu_p(\tau)]$ is non
zero iff $\left[\nu_p(\tau):\sigma\right]$ is non zero.
Moreover by \cite[Proposition 49, Proposition 53]{Knapp Stein}, if
$\pi^c_{\sigma,\lambda}$
belongs to the complementary series one has $\sigma=w_0\sigma$ and
$0<\lambda<1$. Recall that by \eqref{infchps} one has
\begin{align*}
\pi^c_{\sigma,\lambda}(\Omega)=c(\sigma)+\lambda^2. 
\end{align*}
Thus together with Lemma \ref{lowerbd1} one gets
\begin{align*}
-\pi^c_{\sigma,\lambda}
(\Omega)+\tau(\Omega)=-c(\sigma)-\lambda^2+\tau(\Omega)\geq \tau_{n+1}^2\geq
1/4. 
\end{align*}
\end{proof}

We are now ready to introduce the analytic torsion. We distinguish between 
the odd- and even-dimensional case. The reason is that the even-dimensional
case can be treated more elementary. 

First assume that $d=2n+1$. 
Let $h_p(\tau):=\dim(\ker\Delta_p(\tau)\cap L^2)$. 
Using \eqref{regtrace3}, Lemma \ref{lowerbd1} and Lemma \ref{asympexp3},
it follows that there is an asymptotic expansion
\begin{equation}\label{asympexp4}
\Tr_{\rel}\left(e^{-t\Delta_p(\tau)}\right)\sim h_p(\tau)+\sum_{j=1}^\infty c_j 
t^{-j/2},\quad t\to\infty.
\end{equation}
On the other hand, by Proposition \ref{AsEnt}, 
$\Tr_{\rel}\left(e^{-t\Delta_p(\tau)}\right)$ has also an asymptotic expansion
as
$t\to0$. Thus we can define the spectral zeta function by
\begin{equation}\label{speczeta1}
\begin{split}
\zeta_p(s;\tau)&:=\frac{1}{\Gamma(s)}\int_0^1 t^{s-1}\left(\Tr_{\rel}\left(
e^{-t\Delta_p(\tau)}\right)-h_p(\tau)\right)\,dt\\
&+\frac{1}{\Gamma(s)}\int_1^\infty t^{s-1}\left(\Tr_{\rel}\left(
e^{-t\Delta_p(\tau)}\right)-h_p(\tau)\right)\,dt.
\end{split}
\end{equation}
By Proposition \ref{AsEnt}, the first integral on the right converges in
the half-plane $\Re(s)>d/2$ and admits a meromorphic extension to $\C$ which
is holomorphic at $s=0$. By \eqref{asympexp4}, the second integral converges
in the half-plane $\Re(s)<1/2$ and also admits a meromorphic extension to $\C$
which is holomorphic at $s=0$.

Now assume that $\tau\neq\tau_{\theta}$. This is equivalent to
$\tau_{n+1}\neq 0$. 
Then by \eqref{Darstellungen von G} and Lemma \ref{lowerbd1} we have
$\tau(\Omega)
-c(\sigma)>1/4$ for all $\sigma\in\hat M$ with $[\nu_p(\tau)\colon\sigma]\neq0$
and $p=0,\dots,d$.  Furthermore by Lemma \ref{lowerbd2} we have
$\ker(\Delta_p(\tau)\cap L^2)=0$, $p=0,\dots,d$.
By \eqref{regtrace3} it follows that there exist $C,c>0$ such
that for all $p=0,\dots,d$:
\begin{equation}
\Tr_{\rel}\left(e^{-t\Delta_p(\tau)}\right)\le C e^{-ct},\quad t\ge 1.
\end{equation}
Using Proposition \ref{AsEnt}, it follows that 
$\zeta_p(s;\tau)$ can be defined as in the compact case by
\begin{equation}\label{speczeta2}
\zeta_p(s;\tau):=\frac{1}{\Gamma(s)}\int_0^\infty t^{s-1}
\Tr_{\rel}\left(e^{-t\Delta_p(\tau)}\right)\;dt.
\end{equation}
The integral converges absolutely and uniformly on compact subsets of 
$\Re(s)>d/2$ and admits a meromorphic extension to $\C$ which is holomorphic 
at $s=0$. We define the regularized determinant of $\Delta_p(\tau)$ as in the 
compact case by
\begin{equation}\label{regdet}
\det \Delta_p(\tau):=\exp\left(-\frac{d}{ds}\zeta_p(s;\tau)\big|_{s=0}
\right).
\end{equation}
In analogy to the compact case we now define the analytic torsion 
$T_X(\tau)\in\R^+$ associated to the the flat bundle
$E_{\tau}$, equipped with the admissible metric, by
\begin{equation}\label{anator3}
T_{X}(\tau):=\prod_{p=0}^{d}\det\Delta_{p}(\tau)^{(-1)^{p+1}p/2}.
\end{equation}
Let $K(t,\tau)$ be defined by \eqref{ktau5}. 
If $\tau\not\cong\tau_\theta$, then $K(t,\tau)=O(e^{-ct})$ as $t\to\infty$, and
the analytic torsion is given by
\begin{align}\label{Def Tor}
\log{T_{X}(\tau)}=\frac{1}{2}\frac{d}{ds}\biggr|_{s=0}\left(\frac{1}{\Gamma(s)}
\int_{0}^{\infty}t^{s-1}K(t,\tau)\,dt\right),
\end{align}
where  the right hand side is defined near $s=0$ by analytic continuation. 

Now assume that  $d=2n+2$. We use \eqref{Def Tor} as the definition 
of $T_X(\tau)$. Let $h_p(\tau):=\dim(\ker\Delta_p(\tau)\cap L^2)$ and let
\begin{align*}
h(\tau):=\sum_{p=0}^d(-1)^pp h_p(\tau).
\end{align*}
Then it follows from \eqref{regtrace3} and Lemma \ref{lowerbd1} that there
exists a constant $c>0$ such that
\begin{align}\label{decK}
K(t,\tau)-h(\tau)=O(e^{-ct}),\quad t\to\infty. 
\end{align}
Next we use \eqref{FTK} to determine the short-time asymptotics of $K(t,\tau)$
and to prove Proposition \ref{toreven}.
To compute the terms on the right hand side of \eqref{FTK}, we note that by
\cite[Lemma 4.1]{MP} 
we have
\begin{equation}\label{vanish}
\Theta_{\sigma,\lambda}(k_t^\tau)=0,\quad \forall \sigma\in\hat
M,\,\lambda\in\R.
\end{equation}
This result immediately implies $H(k_t^\tau)=0$ by \eqref{Hyperb}, 
$T(k_t^\tau)=0$ by \eqref{Fouriertrafo T},  and $\mathcal{I}(k_t^\tau)=0$ by
Theorem \ref{FTorbint}. The identity contribution is given by
\[
I(k_t^\tau)=\vol(X)k_t^\tau(1).
\]
Since $k_t^\tau$ is a $K$-finite function in $\mathcal{C}(G)$, the Plancherel
Theorem can be
applied to $k_t^\tau$ by \cite[Theorem 3]{HC}. Thus by \cite[Theorem
13.5]{Knapp} and \eqref{vanish} we have
\[
k_t^\tau(1)=\sum_{\pi\in\hat G_d}a(\pi)\Theta_\pi(k_t^\tau),
\]
where $\hat G_d$ denotes the discrete series and $a(\pi)\in\C$. Since
$k_t^\tau$ is $K$-finite, the sum is finite. In \cite[Section 5]{MP} it was
shown that for each $\pi\in\hat G_d$, $\Theta_\pi(k_t^\tau)$ is independent
of $t>0$. This implies that $I(k_t^\tau)$ is independent of $t$. Summarizing,
it follows from \eqref{FTK} that there exists $c(\tau)\in\C$ such that
\begin{align}\label{KJ}
K(t,\tau)=c(\tau)+J(k_t^\tau).
\end{align}
Next we investigate $J(k_t^\tau)$. Using \eqref{kernel5} and \eqref{reprJ},
we have
\begin{equation}\label{jtau}
\begin{split}
J(k_t^\tau)=-\frac{\kappa(X)}{4\pi i}\sum_{p=1}^d(-1)^p p 
\sum_{\substack{\nu\in\hat{K}\\\left[\nu_p(\tau):\nu\right]\neq0}}&\sum_{
\sigma\in\hat M}
[\nu:\sigma]\dim(\sigma)e^{-t(\tau(\Omega)-c(\sigma))}\\
&\cdot\int_{D_\epsilon}e^{-tz^2}c_\nu(\sigma:z)^{-1}\frac{d}{dz}
c_\nu(\sigma:z)\;dz.
\end{split}
\end{equation}
Thus by Lemma \ref{lowerbd1} one has
\begin{align*}
J(k_t^\tau)=O(e^{-ct}),\:t\to\infty
\end{align*}
for some constant $c>0$.
Using  \eqref{decK} and \eqref{KJ} it follows that $c(\tau)=h(\tau)$ 
and we get
\begin{align}\label{KeqJ}
K(t,\tau)-h(\tau)=J(k_t^\tau).
\end{align}
For the short-time asymptotic of $K(t,\tau)$, we use
equation \eqref{cfunktion'}, Lemma \ref{Lemas1}, Lemma \ref{Lemas2} and
\eqref{KeqJ}.
This implies that there exist $a_j,b_j\in\C$ such that
\begin{align*}
K(t,\tau)\sim \sum_{j=0}^\infty a_jt^{j-1/2}+\sum_{j=0}^\infty b_j
t^{j-1/2}\log{t}
+\sum_{j=0}^\infty c_jt^{j}
\end{align*}
as $t\to 0$. Together with \eqref{decK} it follows that the integral
\begin{align*}
\int_0^\infty t^{s-1}\left(K(t,\tau)-h(\tau)\right)dt
\end{align*}
converges for $\Real(s)\gg 0$ and admits a meromorphic continuation to $s\in\C$
with at most a simple pole at $s=0$. Then in analogy with \eqref{Def Tor}, we 
define the analytic torsion $T_{X}(\tau)\in\R^+$ of $E_\tau$ with respect to 
the admissible metric by 
\begin{align*}
T_{X}(\tau)=\exp{\left(\frac{1}{2}\frac{d}{ds}\left(\frac{1}{
\Gamma(s)}
\int_{0}^{\infty}t^{s-1}K(t,\tau)\,dt\right)\bigg|_{s=0}\right)}.
\end{align*}
Let $\tau=\tau_\lambda$ be an irreducible finite-dimensional representation of
$G$
with highest weight $\lambda\in\Lambda(G)$. Using \eqref{jtau}
it follows that there exist a function 
$\psi\colon \R^+\times\Lambda(G)\to \R$
such that 
\[
J(k_t^{\tau_\lambda})=\kappa(X)\psi(t,\lambda)
\]
for all even-dimensional $X$ and $\lambda\in\Lambda(G)$. For $\lambda\in
\Lambda(G)$ let
\[
\Phi(\lambda):=\frac{1}{2}\frac{d}{ds}\left(\frac{1}{\Gamma(s)}
\int_0^\infty \psi(t,\lambda)t^{s-1}\;dt\right)\bigg|_{s=0},
\]
where the value at $s=0$ is defined by analytic continuation. Then by
the definition of $T_X(\tau)$ we have
\[
\log T_X(\tau_\lambda)=\kappa(X)\Phi(\lambda)
\]
for all even-dimensional $X$ and $\lambda\in\Lambda(G)$. This proves Proposition
\ref{toreven}.

\section{Virtual heat kernels}\label{secvir}
\setcounter{equation}{0}

In order to deal with the Mellin transform of the terms on the right hand side
of \eqref{Spurf} we express $k_t^\tau$ in terms of certain auxiliary 
heat kernels which are easier to handle.
These functions first occurred in \cite{Bunke} in a different context. To begin
with, we need some preparation. In this section we assume that  $d=2n+1$.

Let $\tau\in\hat{G}$ and let  $\Lambda(\tau)=\tau_{1}e_{1}+\dots
+\tau_{n+1}e_{n+1}$ be its highest weight. For $w\in W$ let $l(w)$ denote its 
length with respect to 
the simple roots which define the positive roots above. Let 
\begin{align*}
W^{1}:=\{w\in W_{G}\colon w^{-1}\alpha>0\:\forall \alpha\in
\Delta(\mathfrak{m}_{\C},\mathfrak{b}_{\C})\}
\end{align*}
Let $V_{\tau}$ be the representation space of $\tau$. For $k=0,\dots,2n$ let 
$H^{k}(\overline{\mathfrak{n}},V_{\tau})$ be the cohomology of 
$\overline{\mathfrak{n}}$ with coefficients in $V_{\tau}$. Then 
$H^{k}(\overline{\mathfrak{n}},V_{\tau})$ is an $MA$ module. 
For $w\in W^1$ let $V_{\tau(w)}$ be the $MA$ module of highest weight 
$w(\Lambda(\tau)+\rho_{G})-\rho_{G}$. By a theorem of Kostant 
(see \cite[Theorem III.3]{Borel}), it follows that as $MA$-modules one has
\begin{align*}
H^{k}(\overline{\mathfrak{n}};V_{\tau})\cong\sum_{\substack{w\in W^{1}\\ 
l(w)=k}}V_{\tau(w)},
\end{align*}
Note that $\bar\nf\cong \nf^*$ as $MA$-modules. Using the Poincare 
principle \cite[(7.2.3)]{Kostant}, we get
\begin{equation}\label{kostant1}
\sum_{k=0}^{2n}(-1)^{k}\Lambda^{k}\nf^*
\otimes V_\tau=\sum_{w\in W^{1}}(-1)^{l(w)}V_{\tau(w)}.
\end{equation}
as $MA$-modules.

For $w\in W^{1}$ let $\sigma_{\tau,w}$ be the representation of $M$ with 
highest weight 
\begin{equation}\label{sigmatauw}
\Lambda(\sigma_{\tau,w}):=w(\Lambda(\tau)
+\rho_{G})|_{\mathfrak{b}_{\mathbb{C}}}-\rho_{M}
\end{equation}
 and let  $\lambda_{\tau,w}\in\mathbb{C}$ such that 
\begin{equation}\label{lambdatauw}
w(\Lambda(\tau)+\rho_{G})|_{\mathfrak{a}_{\mathbb{C}}}=\lambda_{\tau,w}e_{1}.
\end{equation}
For $k=0,\dots n$ let
\begin{align}\label{lambdatau}
\lambda_{\tau,k}=\tau_{k+1}+n-k
\end{align}
and $\sigma_{\tau,k}$ be the representation of $G$ with highest weight
\begin{align}\label{sigmatau}
\Lambda_{\sigma_{\tau}^{k}}:=(\tau_{1}+1)e_{2}+\dots+(\tau_{k}+1)e_{k+1}
+\tau_{k+2}e_{k+2}+\dots+\tau_{n+1}e_{n+1}.
\end{align}
Then by the computations in \cite[Chapter VI.3]{Borel} one has
\begin{equation}\label{lambdadecom}
\begin{split}
\{(\lambda_{\tau,w},\sigma_{\tau,w},l(w))\colon w\in W^{1}\}
&=\{(\lambda_{\tau,k},\sigma_{\tau,k},k)\colon k=0,\dots,n\}\\
&\sqcup\{(-\lambda_{\tau,k},w_{0}\sigma_{\tau,k},2n-k)\colon k=0,\dots,n\}.
\end{split}
\end{equation}

We will also need the following lemma.
\begin{lem}\label{Aussage uber Casimirew}
For every $w\in W^{1}$ one has
\begin{align*}
\tau(\Omega)=\lambda_{\tau,w}^{2}+c(\sigma_{\tau,w}).
\end{align*}
\end{lem}
\begin{proof}
Let $I(\hL_\C)$ be the Weyl group invariant elements of the symmetric algebra
$S(\hL_\C)$ of $\hL_\C$. Let
\[
\gamma\colon Z(\gL_\C)\to I(\hL_\C)
\]
be the Harish-Chandra isomorphism \cite[Section VIII,5]{Knapp}. A standard 
computation gives
\begin{equation}\label{gammacas}
\gamma(\Omega)=\sum_{j=1}^{n+1} H_j^2-\sum_{j=1}^{n+1} \rho_j^2.
\end{equation}
Each $\Lambda\in
\hL_\C^*$ defines a homomorphism $\chi_\Lambda\colon Z(\gL_\C)\to \C$ by
\[
\chi_\Lambda(Z):=\Lambda(\gamma(Z)).
\]
Then we have
\[
\tau(\Omega)=\chi_{\Lambda(\tau)+\rho_G}(\Omega)
\]
(see \cite[Section VIII,6]{Knapp}). Using that 
$\chi_\Lambda=\chi_{w\Lambda}$, $w\in W$
and \eqref{gammacas}, we get
\begin{align*}
\tau(\Omega)=&\chi_{\Lambda(\tau)+\rho_{G}}(\Omega) 
=\chi_{w(\Lambda(\tau)+\rho_{G})}(\Omega)
=\chi_{\Lambda(\sigma_{\tau,w})+\rho_{M}+\lambda_{\tau}(w)e_{1}}(\Omega)
=\lambda_{\tau,w}^{2}+c(\sigma_{\tau,w}).
\end{align*}
\end{proof}

Fix $\sigma\in\hat{M}$ and assume that $\sigma\neq w_{0}\sigma$. For $\nu\in
\hat K$ let $m_\nu(\sigma)\in\{-1,0,1\}$ be defined by \eqref{multipl1}.
Let $H^\nu_t$ be the kernel of $e^{-t\tilde A_\nu}$ as in \eqref{kernelcov} and 
let $h_t^\nu:=\tr H_t^\nu$. Put
\begin{align}\label{Defh}
h^{\sigma}_{t}(g):=e^{-tc(\sigma)}\sum_{\substack{\nu\\
m_{\nu}(\sigma)\neq 0}}m_\nu(\sigma)h_t^\nu(g).
\end{align}

\begin{prop}\label{auxk}
For $k=0,\dots,n$ let $\sigma_{\tau,k}$ and $\lambda_{\tau,k}$ be as in
\eqref{lambdadecom}. Then one has
\begin{align*}
k_t^\tau=\sum_{k=0}^{n}(-1)^{k+1}e^{-t\lambda_{\tau,k}^2}h_t^{\sigma_{\tau,k}}
.
\end{align*}
\end{prop}
\begin{proof}
It is easy to see that as $M$-modules $\mathfrak{p}$ and
$\mathfrak{a}\oplus\mathfrak{n}$ are equivalent. 
Thus in the sense of $M$-modules one has
\begin{equation}\label{alter2}
\sum_{p=0}^{d}(-1)^p p\,\Lambda^p\mathfrak  p^*
=\sum_{p=0}^{d}(-1)^{p}p\left(\Lambda^p\mathfrak n^*
+\Lambda^{p-1}\mathfrak n^*\right)
=\sum_{p=0}^{d-1}(-1)^{p+1}\Lambda^p\mathfrak n^*.
\end{equation}
Let $i^*\colon R(K)\to R(M)^{W(A)}$ be the restriction map. 
Then it follows from \eqref{alter2}, \eqref{kostant1} and
\eqref{lambdadecom} that we have
\begin{equation}\label{Altsum}
\sum_{p=0}^{d}(-1)^{p}p\hskip2pt i^*(\nu_{p}(\tau))=
\sum_{k=0}^n(-1)^{k+1}(\sigma_{\tau,k}+w_0\sigma_{\tau,k}).
\end{equation}
Since $\tau\neq \tau_{\theta}$ we have $\sigma_{\tau,k}\neq w_0\sigma_{\tau,k}$
for
all $k$ by \eqref{Tau theta}, \eqref{wsigma} and \eqref{sigmatau}.
Thus as in \eqref{multipl1} we can write
\begin{equation*}
\sigma_{\tau,k}+w_0\sigma_{\tau,k}=\sum_{\nu\in\hat
K}m_\nu(\sigma_{\tau,k})i^*(\nu).
\end{equation*}
Moreover, the restriction map $i^*$ is injective. 
Therefore  the following equality holds in $R(K)$ :
\begin{align*}
\sum_{p=0}^{d}(-1)^{p}p\nu_{p}(\tau)=
\sum_{k=0}^n(-1)^{k+1}\sum_{\nu\in\hat{K}}m_{\nu}(\sigma_{\tau,k})\nu.
\end{align*}
Since $R(K)$ is a free abelian group generated by the representations
$\nu\in\hat{K}$,
it follows that for every $\nu\in\hat{K}$ one has
\begin{align}\label{kequ}
\sum_{p=0}^{d}(-1)^p
p\left[\nu_p(\tau):\nu\right]=\sum_{k=0}^n(-1)^{k+1}m_\nu(\sigma_{\tau,k}).
\end{align}
Moreover let us remark that if $\nu,\nu_1,\nu_2$ are finite dimensional unitary
representations
of $K$ with $\nu=\nu_1\oplus\nu_2$ one has 
\begin{align}\label{Keradd}
h_t^\nu=h_t^{\nu_1}+h_t^{\nu_2}.
\end{align}
Thus we obtain
\begin{align*}
k_t^\tau=\sum_{p=0}^{d}(-1)^ppe^{-t\tau(\Omega)}h_t^{\nu_p(\tau)}=&\sum_
{
p=0
} ^{d}(-1)^pp\sum_{\nu\in\hat{K}
}\left[\nu_p(\tau):\nu\right]e^{-t\tau(\Omega)}h_t^\nu\\
=&\sum_{\nu\in\hat{K}}\sum_{p=0}^{d}(-1)^p
p\left[\nu_p(\tau):\nu\right]e^{-t\tau(\Omega)}h_t^\nu\\
=&\sum_{\nu\in\hat{K}}\sum_{k=0}^n(-1)^{k+1}m_\nu(\sigma_{\tau,k})e^{
-t\left(\tau(\Omega)\right)}h_t^\nu\tag{+}\\
=&\sum_{k=0}^n(-1)^{k+1}\sum_{\nu\in\hat{K}}m_\nu(\sigma_{\tau,k})e^{
-t\left(\tau(\Omega)\right)}h_t^\nu\\
=&\sum_{k=0}^n(-1)^{k+1}\sum_{\nu\in\hat{K}}m_\nu(\sigma_{\tau,k})e^{
-t\left(\lambda_{\tau,k}^2+c(\sigma_{\tau,k})\right)}h_t^\nu\tag{++}\\
=&\sum_{k=0}^n(-1)^{k+1}e^{-t\lambda_{\tau,k}^2}h_t^{\sigma_{\tau,k}}\tag{+++}.
\end{align*}
Here the second equality in the first line follows from \eqref{Keradd}, $(+)$ is
\eqref{kequ}, $(++)$ follows from Lemma \ref{Aussage uber Casimirew} and
$(+++)$ follows 
from \eqref{Defh}.
\end{proof}

Finally we compute the Fourier transform of $h_t^{\sigma}$,  
$\sigma\in\hat{M}$. Using \eqref{multipl1} 
and Proposition \ref{fouriertrf2}, it follows that for a principal series 
representation $\pi_{\sigma',\lambda}$, $\lambda\in\mathbb{R}$ we have
\begin{align}\label{FTSup}
\Theta_{\sigma',\lambda}(h^{\sigma}_{t})=e^{-t\lambda^{2}} \quad
\text{for $\sigma'\in\{\sigma, w_{0}\sigma\}$};\qquad
\Theta_{\sigma',\lambda}(h^{\sigma}_{t})=0, \quad\text{otherwise}. 
\end{align}

\section{$L^2$-torsion}\label{secl2}
\setcounter{equation}{0}

In this section we briefly discuss the $L^2$-torsion $T^{(2)}_X(\tau)$. 
We assume that  $d=2n+1$. For
the trivial representation, the $L^2$-torsion of complete hyperbolic manifolds
of finite volume has been studied in \cite{LS}. Although $X$ is
not compact, the $L^2$-torsion can be defined as in the compact case \cite{Lo}.
This follows from the fact that $\widetilde X$ is homogeneous. We assume
that the highest weight of $\tau$ satisfies $\tau_{n+1}\neq0$. Let 
$\widetilde\Delta_p(\tau)$ be the Laplace operator on $\widetilde E_\tau$-valued
$p$-forms on $\widetilde X$. By \eqref{kuga} the kernel of 
$e^{-t\widetilde\Delta_p(\tau)}$ is given by 
$e^{-t\tau(\Omega)}H^{\nu_p(\tau)}_t$ where $H^{\nu_p(\tau)}_t$ is the kernel of
the
operator induced by $-\Omega$ in the homogeneous bundle attached to 
$\nu_p(\tau)$ (see \eqref{kernelx}). Then the $\Gamma$-trace 
of $e^{-t\widetilde\Delta_p(\tau)}$ (see \cite{Lo} for its definition) is given
by
\begin{equation}\label{gammatr0}
\Tr_\Gamma\left(e^{-t\widetilde\Delta_p(\tau)}\right)=\vol(X)
e^{-t\tau(\Omega)}h^{\nu_p(\tau)}_t(1).
\end{equation}
Applying the Plancherel theorem to $h^{\nu_p(\tau)}_t(1)$ and using Proposition 
\ref{fouriertrf2}, we get
\begin{equation}\label{gammatr1}
\Tr_\Gamma\left(e^{-t\widetilde\Delta_p(\tau)}\right)
=\vol(X)\sum_{\substack{\sigma\in\hat{M}\\
\left[\nu_p(\tau):\sigma\right]\neq0}}
e^{-t(\tau(\Omega)-c(\sigma))}\int_{\mathbb{R}}e^{-t\lambda^2}P_{\sigma}
(i\lambda)\;d\lambda.
\end{equation}
Since $P_\sigma(z)$ is an even polynomial of degree $d-1$, we get an asymptotic
expansion 
\begin{equation}\label{gammatr2}
\Tr_\Gamma\left(e^{-t\widetilde\Delta_p(\tau)}\right)\sim\sum_{k=0}^\infty a_j
t^{j-d/2},
\quad t\to 0.
\end{equation}
Since we are assuming that the highest weight of $\tau$ satisfies
$\tau_{n+1}\neq0$, it follows from Lemma \eqref{lowerbd1} and \eqref{gammatr1} 
there exists $c>0$ such that
\begin{equation}\label{gammatr3}
\Tr_\Gamma\left(e^{-t\widetilde\Delta_p(\tau)}\right)=O\left(e^{-ct}\right)
\end{equation}
as $t\to\infty$. Therefore the Mellin transform
\[
\int_0^\infty\Tr_\Gamma\left(e^{-t\widetilde\Delta_p(\tau)}\right)t^{s-1}\;dt
\]
converges absolutely and uniformly on compact subsets of $\Re(s)>d/2$ and
admits a meromorphic extension to $\C$. Moreover, since the asymptotic expansion
\eqref{gammatr2} has no constant 
term, $\mathcal{M}I(s,\tau)$ is regular at $s=0$.
So we can define the $L^2$-torsion $T^{(2)}_X(\tau)\in\R^+$ by
\begin{equation}\label{def-l2tor}
\log
T^{(2)}_X(\tau)=\frac{1}{2}\frac{d}{ds}\left(\frac{1}{\Gamma(s)}\sum_{p=1}^d
(-1)^p p
\int_\R\Tr_\Gamma\left(e^{-t\widetilde\Delta_p(\tau)}\right)\;t^{s-1}\,dt
\right)\bigg|_{s=0}.
\end{equation}
Now recall that the contribution of the identity $I(k_t^\tau)$ to the right 
hand side of \eqref{FTK} is given by 
\[
I(t,\tau):=\vol(X)k_t^\tau(1).
\]
Let 
\[
\mathcal{M}I(s,\tau):=\int_0^\infty I(t,\tau)t^{s-1}\;dt
\]
be the Mellin transform. 
Using \eqref{kernel5} and the considerations above, it follows that the 
integral converges for $\Re(s)>d/2$ and has a meromorphic extension to $\C$ 
which is regular at $s=0$. Let $\mathcal{M}I(\tau)$
be its value at $s=0$. Then by \eqref{kernel5}, \eqref{gammatr0}, and
\eqref{def-l2tor} we have
\begin{equation}\label{mellin}
\log T^{(2)}_X(\tau)=\frac{1}{2}\mathcal{M}I(\tau).
\end{equation}
Our next goal is to compute $\mathcal{M}I(\tau)$. 
Let $\sigma_{\tau,k}$ and $\lambda_{\tau,k}$, $k=0,\dots,n$, be defined by
\eqref{lambdatau} and \eqref{sigmatau}, respectively. Then for every $k$ we 
have $\sigma_{\tau,k}\neq w_0\sigma_{\tau,k}$. Let $P_{\sigma_{\tau},k}$ be the 
Plancherel polynomial. Using 
Proposition \ref{auxk}, the Plancherel theorem, \eqref{FTSup} and 
\eqref{P-Polynom is W-inv}, we obtain
\begin{equation}\label{identcontr2}
I(t,\tau)=2\vol(X)\sum_{k=0}^n(-1)^{k+1}e^{-t\lambda_{\tau,k}^2}
\int_\R e^{-t\lambda^2} P_{\sigma_{\tau},k}(i\lambda)\,d\lambda.
\end{equation}
To evaluate the Mellin transform of $I(t,\tau)$ at $s=0$, we use the following
elementary lemma.
\begin{lem}\label{l2tor2}
Let $P$ be an even polynomial.
Let $c>0$ and $\sigma\in\hat{M}$. 
For $\Real(s)>\frac{d}{2}$ let 
\begin{align*}
E(s):=\int_{0}^{\infty}{t^{s-1}e^{-tc^{2}}
\int_{\mathbb{R}}{e^{-t\lambda^{2}}P(i\lambda)d\lambda} dt}.
\end{align*}
Then $E(s)$ has a meromorphic continuation to $\C$. Moreover $E(s)$
is regular at zero and
\begin{align*}
E(0)=-2\pi\int_{0}^{c}P(\lambda)d\lambda.
\end{align*}
\end{lem}
\begin{proof}
This follows from Lemma 2 and Lemma 3 in \cite{Fried}. 
\end{proof}
We have $\lambda_{\tau,k}>0$ for every $k$. Applying Lemma \eqref{l2tor2} to the
right hand side of \eqref{identcontr2}
we obtain
\begin{align*}
\mathcal{M}I(\tau)=4\pi\vol(X)\sum_{k=0}^{n}(-1)^k\int_{0}^{\lambda_{\tau,
k}}P_{\sigma_{\tau,k}}(\lambda)d\lambda.
\end{align*}
Together with \eqref{mellin} we get the following proposition.
\begin{prop}\label{l2tor4}
Let $\tau$ be such that $\tau_{n+1}\neq0$. Then we have
\[
\log T^{(2)}_X(\tau)=2\pi\vol(X)\sum_{k=0}^n(-1)^k\int_0^{\lambda_{\tau,k}}
P_{\sigma_{\tau,k}}(\lambda)\;d\lambda.
\]
\end{prop}

\section{Proof of the main results}\label{Beweis}
\setcounter{equation}{0}
In this section we assume that $d=\dim(X)$ is odd. Let $d=2n+1$. 
We fix natural numbers $\tau_{1},\dots,\tau_{n+1}$ with $\tau_{1}\geq\tau_{2}
\geq\dots\geq\tau_{n+1}$. For $m\in\mathbb{N}$ we let $\tau(m)$ be the
representation of $G$ with highest weight
$(m+\tau_{1})e_{1}+\dots+(m+\tau_{n+1})e_{n+1}$. Then $\tau(m)$ satisfies
$\tau(m)\circ\theta\not\sim
\tau$. Hence the analytic torsion $T_{X}(\tau(m))$ is defined by 
\eqref{Def Tor}.

Our goal is to study the asymptotic behavior of $\log T_{X}(\tau(m))$ as
$m\rightarrow \infty$.
To begin with, for $k\in\{0,\dots,n\}$ let $\lambda_{\tau(m),k}\in\R$ and
$\sigma_{\tau(m),k}\in\hat{M}$ with highest 
weight $\Lambda(\sigma_{\tau(m),k})$ be defined as in \eqref{lambdatau} resp.
\eqref{sigmatau}. One has
\begin{equation}\label{sigmaneu}
\begin{split}
\Lambda(\sigma_{\tau(m),k})=&(m+\tau_{1}+1)e_{2}+\dots+(m+\tau_{k}+1)e_{k+1}
\\
&+(m+\tau_{k+2})e_{k+2}+\dots+(m+\tau_{n+1})e_{n+1}
\end{split}
\end{equation}
and
\begin{align}\label{lambdaneu}
\lambda_{\tau(m),k}=m+\tau_{k+1}+n-k.
\end{align}
We use the decomposition \eqref{FTK} of  $K(t,\tau(m))$  and study the
Mellin transform of each term on the right hand side separately. 
First we consider the identity contribution which is given by
\[
I(t,\tau(m)):=\vol(X)k_t^{\tau(m)}(1).
\]
Its Mellin transform $\mathcal{M}I(\tau(m))$ has been computed in the
previous section and the contribution to $\log T_X(\tau(m))$ equals
\[
\frac{1}{2}\mathcal{M}I(\tau(m))=\log T^{(2)}_X(\tau(m)).
\]  
In order to study the asymptotic behavior of $\log T^{(2)}_X(\tau(m))$ as
$m\to\infty$, we use Proposition \ref{l2tor4}.  
Let
\[
P_\tau(m):=2\pi\sum_{k=0}^{n}(-1)^k\int_{0}^{\lambda_{\tau(m),
k}}P_{\sigma_{\tau(m),k}}(\lambda)d\lambda.
\]
Using \eqref{lambdaneu} and the explicit form of the Plancherel polynomial
$P_{\sigma_{\tau(m),k}}(\lambda)$, it follows that $P_\tau(m)$ is a polynomial
in $m$
of degree $n(n+1)/2+1$. The coefficient of the leading power has been determined
at the end of section 5 of \cite{MP}. Let $C(n)$ be constant given by 
\eqref{cn}.
Combining the results above with the computations of the leading 
coefficient of $P_\tau(m)$ in \cite{MP}, we get
\begin{prop}\label{Idkontr}
We have
\begin{align*}
\log T^{(2)}_X(\tau(m))=-C(n)\vol(X)m\dim{\tau(m)}+O(m^{\frac{n(n+1)}{2}}),
\end{align*}
as $m\to\infty$. 
\end{prop}

Thus to prove our main results we have to show that the Mellin transform of the
terms in \eqref{FTK} which are different from the identity contribution are of
lower order as $m\to\infty$.
We begin with the contribution of the hyperbolic term to the analytic torsion.
For $[\gamma]\in C(\Gamma)_s-[1]$ and $\sigma\in\hat M$ let 
$L(\gamma,\sigma)$ be defined by \eqref{hyperbcontr}. Put
\begin{equation}\label{syml}
L_{\sym}(\gamma;\sigma):=L(\gamma;\sigma)+L(\gamma;w_0\sigma).
\end{equation}
Using \eqref{Hyperb}, Proposition \ref{auxk} and \eqref{FTSup}, it follows that
the hyperbolic contribution is given by
\begin{equation}\label{hyperbcont4}
H(t,\tau(m)):=\sum_{k=0}^n (-1)^{k+1}
e^{-t\lambda_{\tau(m),k}^2}\sum_{[\gamma]\in
C(\Gamma)_s-[1]}\frac{\ell(\gamma)}{n_\Gamma(\gamma)} 
L_{\sym}(\gamma;\sigma_{\tau(m),k})\frac{e^{-\ell(\gamma)^2/4t}}{(4\pi
t)^{\frac{1}{2}}}.
\end{equation}
In order to study the Mellin transform of $H(t,\tau(m))$, we use the following
proposition.
\begin{prop}\label{hyperbcontr3}
Let $\lambda>\sqrt{2}n$ and $\sigma\in\hat{M}$. 
For every 
$s\in\mathbb{C}$ the integral
\begin{align}\label{Integral}
G(s,\lambda;\sigma):=\int_{0}^{\infty}t^{s-1}e^{-t\lambda^2}\sum_{
\left[
\gamma\right]\in\CC(\Gamma)_{\s}-\left[
1\right]}
\frac{\ell(\gamma)}{n_{\Gamma}(\gamma)}L(\gamma;\sigma)
\frac{e^{-\ell(\gamma)^{2}/4t}}{(4\pi t)^{\frac{1}{2}}}dt
\end{align}
converges absolutely and is an entire function of $s$. There exists a constant
$C_0$ which is independent of $\sigma$ and $\lambda$ such that
\begin{align}\label{estimat5}
\left|G(0,\lambda;\sigma)\right|\leq C_0 \dim(\sigma).
\end{align}
\end{prop}
\begin{proof}
Let
\begin{align*}
f(t):=\sum_{\left[\gamma\right]\in\CC(\Gamma)_{\s}-\left[
1\right]}\frac{\ell(\gamma)}{n_{\Gamma}(\gamma)}
L(\gamma;\sigma)\frac{e^{-\ell(\gamma)^{2}/4t}}{(4\pi
t)^{\frac{1}{2}}}.
\end{align*}
We have
\begin{align*}
|f(t)|\leq \dim(\sigma)\sum_{\left[\gamma\right]\in\CC(\Gamma)_{\s}-\left[
1\right]}\frac{\ell(\gamma)}
{n_{\Gamma}(\gamma)}L(\gamma;1)
\frac{e^{-\ell(\gamma)^{2}/4t}}{(4\pi t)^{\frac{1}{2}}},
\end{align*}
where $1$ stands for the trivial representation of $M$. Now let $\Delta_{0}$ 
be the Laplace operator acting on $C^{\infty}(X)$ and let $\Delta_{0}^d$ be 
its restriction to the point spectrum. Then the right hand side is exactly
the hyperbolic contribution to the Selberg trace formula for 
$\Tr(e^{-t\Delta_0^d})$. So we can apply the trace formula to
estimate the right hand side. Denote the trivial representation of $K$ by $1$
too.
Then if we apply the trace formula \cite[Theorem 8.4, Theorem 9.3]{Warner}
and use
equation \eqref{csigma}, Proposition \ref{fouriertrf2}, 
equation \eqref{Idcontr}
and equation \eqref{Hyperb}, it follows that there
exist constants $c_1(\Gamma), c_2(\Gamma)$ such that
\begin{align*}
&e^{-tn^2}\sum_{\left[\gamma\right]\in\CC(\Gamma)_{\s}-\left[
1\right]}\frac{\ell(\gamma)}{n_{\Gamma}(\gamma)}L(\gamma;1)
\frac{e^{-\ell(\gamma)^{2}/4t}}{(4\pi
t)^{\frac{1}{2}}}=\Tr\left(e^{-t\Delta_{0}^d}
\right)+\int_{\R}e^{-t(\lambda^{2}+n^2)}\vol(X)P_1(i\lambda)d\lambda\\
-&\int_{\R}e^{-t(\lambda^{2}+n^2)}\left(\psi(1+i\lambda)+c_2(\Gamma)+
\Tr\left(\widetilde{\boldsymbol{C}}
(1,1,-i\lambda)\frac{
d}{dz}\widetilde{\boldsymbol{C}}(1,1,i\lambda)\right)
\right)d\lambda+c_1(\Gamma)e^{-tn^2}.
\end{align*}
The right hand side of this equation
is bounded for $t\ge1$. Thus there exists a constant $C_{1}$ which is 
independent of $\sigma$ such that 
\begin{equation}\label{estim10}
|f(t)|\leq C_{1}\dim(\sigma)\,e^{tn^2},\quad t\ge 1.
\end{equation}
For $\lambda>n$ and $s\in\C$ put 
\[
G_0(s,\lambda;\sigma):=\int_1^\infty t^{s-1}e^{-t\lambda^2}f(t)\;dt.
\]
Then it follows from \eqref{estim10} that  
$G_0(s,\lambda;\sigma)$ is an entire function of $s$ and that for
$\lambda>\sqrt{2}n$
we can estimate 
\begin{equation}\label{estim1}
|G_0(0,\lambda;\sigma)|\le\int_{1}^{\infty}t^{-1}e^{-t\lambda^{2}}|f(t)|\;dt
\le C_1\dim(\sigma)\, 
e^{-\frac{\lambda^2}{4}},\quad \lambda>\sqrt{2}n.
\end{equation}
Next we consider the integral from 0 to 1. To begin with, we need to estimate
$L(\gamma,\sigma)$. By \cite[ Proposition 5.4]{GaWa} there exist a constant
$C_2>0$ such that for $R>0$ one
has
\begin{align}\label{esthypc}
\#\{\left[\gamma\right]\in\CC(\Gamma)_{\s}\colon \ell(\gamma)\leq
R\}\leq C_2e^{2nR}.
\end{align}
Thus if we let
\begin{align}\label{infspec}
c:=\min\{\ell(\gamma)\colon\left[\gamma\right]\in\CC(\Gamma)_{\s}-\left[1\right]
\}
\end{align}
we have $c>0$. Moreover one has
\begin{align*}
\det{\left(\Id-\Ad(m_\gamma
a_\gamma)|_{\bar{\mathfrak{n}}}\right)}\geq\left(1-e^{-\ell(\gamma)}
\right)^{n}.
\end{align*}
Hence there exists a constant $C_3$ such that for all 
$\left[\gamma\right]\in\CC(\Gamma)_{\s}-\left[1\right]$ one has
\begin{align*}
\frac{1}{\det{\left(\Id-\Ad(m_\gamma
a_\gamma)|_{\bar{\mathfrak{n}}}\right)}}\leq C_3.
\end{align*}
It follows that
there exists a constant
$C_4$ which is independent of $\sigma$ such that for every
$\left[\gamma\right]\in\CC(\Gamma)_{\s}-\left[
1\right]$ one has
\begin{equation}\label{estim2}
\frac{\ell(\gamma)}{n_{\Gamma}(\gamma)}|L(\gamma;\sigma)|
\le \frac{\dim(\sigma)\ell(\gamma)e^{-n\ell(\gamma)}}{\det(\Id
-\Ad(m_{\gamma}a_{\gamma})|_{{\bar{\nf}}})}
\leq C_4\dim(\sigma).
\end{equation}
Let $c$ be as in \eqref{infspec}.
Then $c>0$ by \eqref{esthypc}. Using \eqref{esthypc} and \eqref{estim2}, it
follows that there exists 
$C_5>0$ which is independent of $\sigma$ such that 
\begin{equation}\label{estim3}
|f(t)|\leq C_5\dim(\sigma) e^{-\sqrt{c}/t},\quad 0<t\le 1.
\end{equation}
For $\lambda\geq 0$ and $s\in\C$ put
\[
G_1(s,\lambda;\sigma)=\int_0^1 t^{s-1}e^{-t\lambda^2}f(t)\,dt.
\]
By \eqref{estim3}, $G_1(s,\lambda;\sigma)$ is an entire function of $s$ 
and its value at zero can be estimated as follows
\begin{equation*}
|G_1(0,\lambda;\sigma)|\le 
\int_{0}^{1}t^{-1}e^{-t\lambda^2}|f(t)|\,dt\leq C_{6}\dim(\sigma)
\int_{0}^{1}e^{-t\lambda^{2}}e^{-\frac{\sqrt{c}}{2t}}\,dt\leq C_{6}\dim(\sigma).
\end{equation*}
Together with \eqref{estim1} the proposition follows. 
\end{proof}
Now let $m>\sqrt{2}n$. Then by \eqref{lambdaneu} one has
$\lambda_{\tau(m),k}>\sqrt{2}n$ for
every $k$. Thus by \eqref{hyperbcont4} and 
Proposition \ref{hyperbcontr3} the integral
\[
\mathcal{M}H(s,\tau(m)):=\int_0^\infty t^{s-1} H(t,\tau(m))\;dt
\]
absolutely and uniformly on compact subsets of $\C$ and defines an entire
function of $s$. Denote by $\mathcal{M}H(\tau(m))$ its value at zero. It
can be estimated as follows.
\begin{prop}\label{PropH}
There exists a constant $C$ such that for every $m>\sqrt{2}n$ one has
\begin{align*}
|\mathcal{M}H(\tau(m))|\leq Cm^{\frac{n(n-1)}{2}}.
\end{align*}
\end{prop}
\begin{proof}
By \eqref{Dimension sigma} and \eqref{sigmaneu} there exists a constant $C$ such
that for
every $m\in\N$ one has
\begin{align}\label{dimstau}
\dim(\sigma_{\tau(m),k})\leq C m^{\frac{n(n-1)}{2}}.
\end{align}
The proposition follows from Proposition \ref{hyperbcontr3}.
\end{proof}
The contribution of the distribution $T$ can be treated without difficulty. 
\begin{prop}\label{propT}
For $\Real(s)>>0$ let
\begin{align*}
\mathcal{M}T(s,\tau(m)):=\int_{0}^{\infty}t^{s-1}T(k_t^{\tau(m)})dt.
\end{align*}
Then $\mathcal{M}T(s,\tau(m))$ has a meromorphic continuation to $\C$
and is regular at $s=0$. 
Let $\mathcal{M}T(\tau(m)$ denote its value at $s=0$. Then there exists a
constant $C$ which is independent of $m$ such that
\begin{align*}
|\mathcal{M}T(\tau(m))|\leq Cm^{\frac{n(n+1)}{2}}.
\end{align*}
\end{prop}
\begin{proof}
By Proposition \ref{auxk}, equation \eqref{Fouriertrafo T} and equation
\eqref{FTSup} we have
\begin{align*}
\mathcal{M}T(s,\tau(m))=\frac{C(\Gamma)}{2\sqrt{\pi}}\sum_{k=0}^{n}
(-1)^{k+1}\dim(\sigma_{\tau(m),k})\left(\lambda_{\tau(m),k}\right)^{-2s+1}
\Gamma\left(s-\frac{1}{2}\right).
\end{align*}
The proposition follows from \eqref{lambdaneu} and \eqref{dimstau}.
\end{proof}
To treat the remaining terms, we need the following two auxiliary lemmas.
\begin{lem}\label{Ratcontr}
For $c\in(0,\infty)$, $s\in\C$, $\Real(s)>0$,
$j\in\left[0,\infty\right)$
let
\begin{align*}
\zeta_{c}(s):=\frac{1}{\pi}\int_{0}^{\infty}{t^{s-1}e^{-tc^2}\int_{D_{\epsilon}}
{\frac{e^{
-tz^{2}}}{iz+j}dz}dt},
\end{align*}
where $D_\epsilon$ is the same contour as in \eqref{j1}. Then
$\zeta_{c}(s)$ has a meromorphic continuation to
$\C$ with a simple pole at 0. Moreover, one has
\begin{align*}
\frac{d}{ds}\biggr|_{s=0}\frac{\zeta_c(s)}{\Gamma(s)}=-2\log{\left(c+j\right)}.
\end{align*}
\end{lem}
\begin{proof}
The statement about the convergence of the integral and the meromorphic
continuation
follows from Lemma \ref{Lemas1} and standard methods. Note that 
\[
\int_{D_{\epsilon}}\frac{e^{-tz^{2}}}{iz}\;dz
=\frac{1}{2}\int_{|z|=\epsilon}\frac{e^{-tz^{2}}}{iz}\;dz=\pi .
\]
Hence, for $j=0$ we have
\begin{align*}
\zeta_{c}(s)=c^{-2s}\Gamma(s) 
\end{align*}
and the claim follows in this case.
Assume that $j>0$. Then one has
\begin{equation}\label{zetaneu}
\zeta_{c}(s)=\frac{j}{\pi}\int_{0}^{\infty}t^{s-1}e^{-tc^2}\int_{\mathbb{R}}
\frac{e^{-t\lambda^2}}{\lambda^2+j^2}d\lambda dt.
\end{equation}
For $\Real(z^2)>0$, $\Real(z)>0$ define a function $\zeta(z,s)$ by
\begin{align*}
\zeta(z,s):=\frac{j}{\pi}\int_{0}^{\infty}{t^{s-1}e^{-tz^2}\int_{\mathbb{R}}{
\frac{e^{
-t\lambda^{2}}}{\lambda^{2}+j^{2}}d\lambda}dt}.
\end{align*}
Then it is easy to see that $\zeta(z,s)$ is holomorphic in $z$.
Let
\begin{align*}
\phi(z,s):=\frac{j}{\pi}\int_{0}^{\infty}{t^{s-1}\int_{\mathbb{R}}{\frac{e^{
-t\left(\lambda^{2}+z^{2}\right)}}{\lambda^{2}+j^{2}}d\lambda}dt}-\frac{j}{\pi}
\int_{0}^{\infty}{t^{s-1}\int_{\mathbb{R}}{\frac{e^{-t\left(\lambda^{2}+j^{2}
\right)}}{\lambda^{2}+j^{2}}d\lambda}dt}.
\end{align*}
Then, since $e^{-tz^{2}}-e^{-tj^{2}}=O(t),\:t\rightarrow 0$,
the integral converges for $\Real(s)>-1$. One has
\begin{align*}
\frac{d}{dz}\phi(z,0)=-\frac{2jz}{\pi}\int_{0}^{\infty}{\int_{\mathbb{R}}{\frac{
e^{-t\left(\lambda^{2}+z^{2}\right)}}{\lambda^{2}+j^{2}}d\lambda}dt}=\frac{-2}{
z+j}.
\end{align*}
Since $\phi(j,0)=0$, one has
\begin{align}\label{PF von phi}
\phi(z,0)=-2\log{\left(z+j\right)}+2\log{2j}.
\end{align}
On the other hand, one has
\begin{align*}
\zeta(j,s)=\frac{j}{\pi
s}\int_{0}^{\infty}{\left(\frac{d}{dt}t^{s}\right)\int_{\mathbb{R}}{\frac{e^{
-t\left(\lambda^{2}+j^{2}\right)}}{\lambda^{2}+j^{2}}d\lambda}dt}=\frac{j^{-2s}}
{\sqrt{\pi}\:s}\Gamma(s+\frac{1}{2}).
\end{align*}
Hence for $s\rightarrow 0$ one has
\begin{align*}
\zeta(j,s)=\frac{1}{s}-2\log{j}+\frac{\Gamma'(\frac{1}{2})}{\sqrt{\pi}}
+O(s)=\frac{1}{s}-2\log{j}+\psi\left(\frac{1}{2}\right)+O(s).
\end{align*}
Together with \eqref{PF von phi} this gives for $s\rightarrow 0$: 
\begin{align*}
\zeta(z,s)=&\frac{1}{s}-2\log{j}+\psi\left(\frac{1}{2}\right)-2\log{
\left(z+j\right)}+2\log{2j}+O(s)\\=
&\frac{1}{s}-2\log{\left(z+j\right)}-\gamma + O(s),
\end{align*}
where we used $\psi(\frac{1}{2})=-2\log{2}-\gamma$. Since for $s\to 0$ one has
\begin{align}\label{Gammaent}
\frac{1}{\Gamma(s)}=s+\gamma s^2+O(s^3),
\end{align}
the proposition follows.
\end{proof}

\begin{lem}\label{Gammacontr}
Let $c\in\R^+$, $s\in\C$, $\Real(s)>1/2$.
Define
\begin{align*}
\tilde{\zeta}_c(s):=\frac{1}{\pi}\int_{0}^{\infty}{t^{s-1}e^{-tc^2}\int_{
\mathbb{R}}{e^{
-t\lambda^{2
}}\psi\left(1+i\lambda\right)d\lambda}dt}.
\end{align*}
Then 
$\tilde{\zeta}_c(s)$ has a meromorphic continuation to
$s\in\C$ with at most a simple pole at $s=0$. Moreover there exist a constant
$C(\psi)$ which is independent of $c$ such that
\begin{align*}
\frac{d}{ds}\biggr|_{s=0}\frac{\tilde{\zeta}_c(s)}{\Gamma(s)}=-2\log{
\Gamma\left(1+c\right)}+C(\psi).
\end{align*}
\end{lem}
\begin{proof}
The convergence of the integral and the statement about the meromorphic
continuation follow from Lemma \ref{Lemas2}
and standard methods. Fix $c_{0}\in\R^+$. Since
$e^{-tz^{2}}-e^{-tc_{0}^{2}}=O(t)$ as $t\to 0$,
it follows from Lemma \ref{Lemas2} that the integral
\begin{align*}
&\tilde{\phi}_c(s,z):=\int_{0}^{\infty}{t^{s-1}\int_{\mathbb{R}}{\left(e^{
-t\left(\lambda^{2}+z^{2}\right)}-e^{-t\left(\lambda^{2}+z_{0}^{2}\right)}
\right)\psi\left(1+i\lambda\right)d\lambda}dt}
\end{align*}
converges for $\Real(s)>-\frac{1}{2}$ and is holomorphic in $z\in\C$,
$\Real(z)>0$, $\Real(z^2)>0$. One has
\begin{align*}
\frac{\partial}{\partial z}\tilde{\phi}_c(0,z)&=-2z\int_{\mathbb{R}}{\frac{
\psi\left(1+i\lambda\right)}{\lambda^{2}+z^{2}}d\lambda}=-2\pi\psi(1+z).
\end{align*}
This proves the lemma.
\end{proof}
Next we treat the contribution of the distribution $\mathcal{I}$ to the 
analytic torsion. By Theorem \ref{FTorbint}, Proposition \ref{auxk} and 
\eqref{FTSup} we have
\begin{equation}
\mathcal{I}(k_t^{\tau(m)})=\frac{\kappa}{\pi}
\sum_{k=0}^n(-1)^{k+1}e^{-t\lambda_{\tau(m),k}^2}
\int_\R\Omega(\sigma_{\tau(m),k},\lambda)e^{-t\lambda^2}\;d\lambda.
\end{equation}
By Proposition \ref{propomega} we have the decomposition
\[
\Omega(\sigma_{\tau(m),k},\lambda)=\Omega_1(\sigma_{\tau(m),k},\lambda)+
\Omega_2(\sigma_{\tau(m),k},\lambda).
\]
Using the description of $\Omega_1$ and $\Omega_2$ together with
Lemma \ref{Lemas1}, Lemma \ref{Lemas2} and Lemma \ref{Lemas3}, it follows that
$\mathcal{I}(k_t^{\tau(m)})$ admits an asymptotic expansion
\begin{align*}
\mathcal{I}(k_t^{\tau(m)})\sim \sum_{k=0}^\infty a_kt^{k-(d-2)/2}
+\sum_{k=0}^{\infty}b_kt^{k-1/2}\log{t}+c_0
\end{align*}
as $t\to 0$. Moreover, since $\lambda_{\tau(m),k}>m$ for every $k$, it follows
that $\mathcal{I}(k_t^{\tau(m)})=O(e^{-tm^2})$ as $m\to\infty$.
Thus for $s\in\C$ with $\Real(s)>(d-2)/2$ the
integral
\begin{align*}
\mathcal{M}\mathcal{I}(s;\tau(m)):=\int_{0}^{\infty}t^{s-1}\mathcal{I}(k_t^{
\tau(m)})dt
\end{align*}
converges and 
has a meromorphic continuation to $\C$ with  at most a simple 
pole at $s=0$. Let 
\begin{align*}
\mathcal{M}\mathcal{I}(\tau(m)):=\frac{d}{ds}\biggr|_{s=0}\frac{\mathcal{M}
\mathcal{I}(s;\tau(m))}{\Gamma(s)}.
\end{align*}
Next we will estimate $\mathcal{M}\mathcal{I}(\tau(m))$ as $m\rightarrow
\infty$. To this end we establish some auxiliary lemmas.
\begin{lem}\label{Lemii} 
There exists a constant $C$ such that for every $m$ one has
\begin{align}\label{Gammaterm}
\sum_{k=0}^{n}(-1)^k\dim(\sigma_{\tau(m),k}
)\left(\log\Gamma(m+\lambda_{
\tau(m),k})+\gamma\lambda_{\tau(m),k}+C(\psi)\right)\leq
Cm^{\frac{n(n+1)}{2}},
\end{align}
where $C(\psi)$ is as in Lemma \ref{Gammacontr}. 
\end{lem}
\begin{proof}
By \eqref{lambdadecom} and \eqref{kostant1} one has
\begin{align}\label{altdim}
2\sum_{k=0}^n(-1)^k\dim(\sigma_{\tau(m),k})=\dim(\tau)\sum_{p=0}^{2n}
(-1)^p\dim\Lambda^p\mathfrak{n}^*=0.
\end{align} 
Thus the sum on the left hand side of \eqref{Gammaterm} equals
\begin{align*}
\sum_{k=0}^{n}(-1)^k\dim(\sigma_{\tau(m),k}
)\left(\log\frac{\Gamma(m+\lambda_{
\tau(m),k})}{\Gamma(2m)}+\gamma\lambda_{\tau(m),k}\right)
\end{align*}
It follows from \eqref{lambdaneu} that there exists a constant
$C$ which is independent of $m$ such that
\begin{align*}
\log{\frac{
\Gamma(m+\lambda_{\tau(m),k})}{\Gamma(2m)}}\leq C \log{m}.
\end{align*}
Using \eqref{lambdaneu} and \eqref{dimstau} the proposition is proved.
\end{proof}
The next two lemmas are concerned with the polynomials $P_j(\sigma,\lambda)$,
$j=2,\dots,n+1$, which are defined by \eqref{Defpoly}.
\begin{lem}\label{Llemii}
Let $k\in\{0,\dots,n\}$ and let
$j\in\{2,\dots,n+1\}$.
Then there exists a constant $C$ such that for every $m$ one has
\begin{align}\label{Polyeins}
\left|P_{j}(\sigma_{\tau(m),k},\lambda)\right|\leq C m^{\frac{(n-1)(n-2)}{2}}
\sum_{i=0}^{2(n-1)}(1+\left|\lambda\right|)^im^{2(n-1)-i}
\end{align}
and such that
\begin{align}\label{Polyzwei}
\left|\frac{d}{d\lambda}P_{j}(\sigma_{\tau(m),k},\lambda)\right|\leq C
m^{\frac{(n-1)(n-2)}{2}}
\sum_{i=0}^{2(n-1)-1}(1+\left|\lambda\right|)^im^{2(n-1)-i}
\end{align}
for all $\lambda\in\C$. 
\end{lem}
\begin{proof}
If we use the explicit formula \eqref{formpj} for the polynomials 
$P_j(\sigma,\lambda)$, combined with \eqref{sigmaneu} and \eqref{dimstau}, 
the lemma follows.
\end{proof}
\begin{lem}\label{lemii}
Let $k\in\{0,\dots,n\}$ and let
$j\in\{2,\dots,n+1\}$.
For $l\in\N$ with $m\leq l\leq
k_j(\sigma_{\tau(m),k})+\rho_j$
let the even polynomial $Q_{j,l}(\sigma_{\tau(m),k},\lambda)$ be defined by
\eqref{DefQ}.
Then there exists a constant $C$ such that for every $m$ one has
\begin{align*}
\bigg|\int_{0}^{\lambda_{\tau(m),k}}Q_{j,l}(\sigma_{\tau(m),k}
,i\lambda)\;d\lambda\bigg|\leq Cm^{\frac{n(n+1)}{2}}.
\end{align*}
\end{lem}
\begin{proof}
By \eqref{DefQ} we have
\[
Q_{j,l}(\sigma_{\tau(m),k},i\lambda)=\frac{P_j(\sigma_{\tau(m),k},i\lambda)-
P_j(\sigma_{\tau(m),k},il)}{l-\lambda}
+\frac{P_j(\sigma_{\tau(m),k},i\lambda)-P_j(\sigma_{\tau(m),k},il)}{l+\lambda}.
\]
Using the fact that $P_j(\sigma,z)$ is an even polynomial, together with
equations \eqref{lambdaneu} and \eqref{Polyzwei}, we obtain
\begin{align*}
\int_{0}^{\lambda_{\tau(m),k}}Q_{j,l}(\sigma_{\tau(m),k},i\lambda)d\lambda\leq 
2\lambda_{\tau(m),k}\max_{\left|\xi\right|\leq l+\lambda_{\tau(m),k}}
\left|\frac{d}{d\lambda}\biggr|_{\lambda=\xi}
P_{j}(\sigma,i\lambda)\right|\leq Cm^{\frac{n(n+1)}{2}}.
\end{align*}
\end{proof}
Now we can estimate $\mathcal{M}\mathcal{I}(\tau(m))$ as $m\to\infty$. 
\begin{prop}\label{propI}
There exists a constant $C$ such that for every $m$ one has
\begin{align*}
\left|\mathcal{M}\mathcal{I}(\tau(m))\right|\leq Cm^{\frac{n(n+1)}{2}}.
\end{align*}
\end{prop}
\begin{proof}
Let 
\[
\mathcal{M}\mathcal{I}(s;\sigma_{\tau(m),k})=\int_0^\infty t^{s-1} 
e^{-t\lambda_{\tau(m),k}^2}\mathcal{I}(h_t^{\sigma_{\tau(m),k}})\;dt.
\]
As in the case of $\mathcal{M}\mathcal{I}(s;\tau(m))$ it follows 
that the integral converges for $\Re(s)>(d-2)/2$ and
admits a meromorphic continuation to $\C$ with at most a simple
pole  at $s=0$. By Proposition \ref{auxk} we have
\begin{align*}
\mathcal{M}\mathcal{I}(\tau(m))=\sum_{k=0}^n(-1)^{k+1}\frac{d}{ds}\biggr|_{s=0}
\frac{\mathcal{M}\mathcal{I}
(s;\sigma_{\tau(m),k})}{\Gamma(s)}.
\end{align*}
Let $k\in\{0,\dots,n\}$. Then by \eqref{sigmaneu} one has
$k_{n+1}(\sigma)\geq m$. Thus we can apply Proposition \ref{propomega} with
$m_0=m-1$.
Using Lemma \ref{Ratcontr} together with \eqref{zetaneu}, Lemma 
\ref{Gammacontr} and Lemma \ref{l2tor2} we obtain
\begin{align*}
&\frac{d}{ds}\biggr|_{s=0}\frac{\mathcal{M}\mathcal{I}(s;\sigma_{\tau(m),k})}{
\Gamma(s)}
=2\kappa\dim(\sigma_{\tau(m),k})\left(\log\Gamma(m+\lambda_{\tau(m),k}
)+\gamma\lambda_{\tau(m),k}+C(\psi)\right)
\\ &+\kappa\sum_{j=2}^{n+1}
\sum_{m\leq l<
k_j(\sigma_{\tau(m),k})+\rho_j}\left(2P_{j}(\sigma_{\tau(m),k},il)
\log{\left(l+\lambda_{\tau(m),k}\right)}
+\int_{0}^{\lambda_{\tau(m),k}}
Q_{j,l}(\sigma_{\tau(m),k},i\lambda)d\lambda\right)\\
&+\kappa\:\sum_{\substack{j=2\\ l=k_{j}(\sigma_{\tau(m),k})+\rho_j
}}^{n+1}\left(\dim(\sigma_{\tau(m),k})\log{\left(l+\lambda_{\tau(m),k}\right)}
+\frac{1}{2}\int_{0}^{\lambda_{\tau(m),k}}Q_{j,l}
(\sigma_{\tau(m),k},i\lambda)d\lambda\right).
\end{align*}
By \eqref{sigmaneu} we have
$k_j(\sigma_{\tau(m),k})+\rho_j\leq m+\tau_1+n$ for every $j=2,\dots,n+1$, 
and by \eqref{lambdaneu} we have
$\lambda_{\tau(m),k}\leq m+\tau_1+n$. Thus if we apply Lemma \ref{Lemii}, Lemma
\ref{Llemii}, Lemma \ref{lemii} and \eqref{dimstau},
the proposition follows.  
\end{proof}

Finally we consider the asymptotic behavior of the contribution of the 
non-invariant distribution $J$ to $\log T_X(\tau(m))$. For 
$k\in\{0,\dots,n\}$ let $h_t^{\sigma_{\tau(m),k}}$ be as in \eqref{Defh}, and
for 
$\nu\in\hat K$ let
\[
m_\nu(\sigma_{\tau(m),k})\in\{-1,0,1\}
\]
be defined by \eqref{multipl1}. By \eqref{reprJ} we have
\begin{equation}\label{FormJ}
\begin{split}
J(h_{t}^{\sigma_{\tau(m),k}})=&e^{-tc(\sigma_{\tau(m),k})}\sum_{\nu\in\hat{K}}m_
{\nu}(\sigma_{\tau(m),k})J(h_{t}^{\nu})\\
=&-\frac{\kappa e^{-tc(\sigma_{\tau(m),k})}}{4\pi
i}\sum_{\sigma\in\hat{M}}\dim(\sigma)\sum_{\nu\in\hat{K}}m_{\nu}(\sigma_{\tau(m)
,k})\left[\nu:\sigma\right
]\\
&\times\int_{D_{\epsilon}}{c_{\nu}(\sigma:z)^{-1}\frac{d}{dz}c_{\nu}
(\sigma:z)e^{-t(z^{2}-c(\sigma))}dz}.
\end{split}
\end{equation}
To continue with the investigation of the right hand side, we 
need the following lemma.
\begin{lem}\label{lemc}
Let $k=0,\dots,n$. For $\sigma\in\hat{M}$ let
\begin{align}\label{Def f}
f_{k,m}(z,\sigma):=\sum_{\nu\in\hat{K}}m_{\nu}(\sigma_{
\tau(m),k})\left[\nu:\sigma\right]c_{\nu}(\sigma:z)^{-1}\frac{d}{dz}c_{\nu}
(\sigma:z).
\end{align}
Then one has
\begin{align*}
f_{k,m}(z,\sigma)=&\sum_{\nu\in\hat{K}}m_{\nu}(\sigma_{\tau(m),k})\left[
\nu:\sigma\right]\sum_{j=2}^{n+1}\left(\sum_{\substack{m\leq l\leq k_j(\nu)
\\ 
\left|k_{j}(\sigma)\right|<l}}\frac{i}{iz-l-\rho_{j}}-\sum_{\substack
{m\leq l\leq k_j(\nu)\\ 
\left|k_{j}(\sigma)\right|\leq l}}\frac{i}{iz+l+\rho_{j}}\right)
\end{align*}
\end{lem}
\begin{proof}
By Proposition \ref{branching} and equation \eqref{sigmaneu}, it follows that
for every $\nu\in\hat{K}$
with $m_{\nu}(\sigma_{\tau(m),k})\neq 0$ and every $j=2,\dots,n+1$ we have
\begin{align*}
m-1\leq k_{j}(\nu), 
\end{align*}
where $(k_2(\nu),\dots,k_{n+1}(\nu))$ is the highest weight of $\nu$. 
Thus using \eqref{Glcfnktn} one can write
\begin{align}\label{Glc}
f_{k,m}(z,\sigma)=&\sum_{\nu\in\hat{K}}m_{\nu}(\sigma_{\tau(m),k})\left[
\nu:\sigma\right]\sum_{j=2}^{n+1}\left(\sum_{\substack{m\leq l\leq k_j(\nu)
\\ 
\left|k_{j}(\sigma)\right|<l}}\frac{i}{iz-l-\rho_{j}}-\sum_{\substack{m\leq
l\leq k_j(\nu)
\\ 
\left|k_{j}(\sigma)\right|\leq l}}\frac{i}{iz+l+\rho_{j}}
\right)\nonumber\\
+&\sum_{\nu\in\hat{K}}m_{\nu}(\sigma_{\tau(m),k})\left[\nu:\sigma\right]\sum_{
j=2}^{n+1}
\left(\sum_{\substack{l=1\\
l>\left|k_{j}(\sigma)\right|}}^{m-1}\frac{i}{iz-l-\rho_{j}}-\sum_{\substack{
l=0\\
l\geq\left|k_{j}(\sigma)\right|}}^{m-1}\frac{i}{iz+l+\rho_{j}}\right).
\end{align}
Now if $\sigma=\sigma_{\tau(m),k}$ or $\sigma=w_{0}\sigma_{\tau(m),k}$ the sum
in the second row of
\eqref{Glc} is
zero by \eqref{sigmaneu} and \eqref{wsigma}. On the other hand, assume that 
$\sigma\neq \sigma_{\tau(m),k}$,
$\sigma\neq w_{0}\sigma_{\tau(m),k}$. Since $R(M)$ is the free abelian
group generated by  $\sigma\in\hat{M}$, it follows from \eqref{multipl1} that
\begin{align*}
\sum_{\nu\in\hat{K}}m_{\nu}(\sigma_{\tau(m),k})\left[\nu:\sigma\right]=0.
\end{align*}
Thus in this case the sum in the second row of \eqref{Glc} is zero too.
This proves the proposition. 
\end{proof}

\begin{prop}\label{propJ}
For $s\in\C$, $\Real(s)>0$ let
\begin{align*}
\mathcal{M}J(s;\sigma_{\tau(m),k}):=\int_{0}
^\infty t^{s-1}e^{-t\lambda_{\tau(m),k}^2} 
J(h^{\sigma_{\tau(m),k}}_t)dt.
\end{align*}
Then $\mathcal{M}J(s;\sigma_{\tau(m),k})$ has a meromorphic continuation to
$\C$ with at most a simple pole at $0$ and we have
\[
\begin{split}
\frac{d}{ds}\biggr|_{s=0}\frac{\mathcal{M}J(s;\sigma_{\tau(m),k})}{\Gamma(s)}&=
-\kappa\sum_{\sigma\in\hat{M}}\sum_{\nu\in\hat{K}}m_{\nu}(\sigma_{\tau(m),k}
)\left[
\nu:\sigma\right]\dim(\sigma)\\
&\hskip30pt\cdot\sum_{j=2}^{n+1}\sum_{\substack{m\leq l\leq k_j(\nu)\\
l>\left|k_j(\sigma)\right|}}\log\left(\sqrt{\lambda_{\tau(m),k}
^2+c(\sigma_{\tau(m),k})-c(\sigma)
}+l+\rho_j\right)\\
&-\frac{\kappa}{2}\sum_{\sigma\in\hat{M}
}\sum_{\nu\in\hat{K}}m_{\nu}(\sigma_{\tau(m),k})\left[
\nu:\sigma\right]\dim(\sigma)\\
&\hskip30pt\cdot\sum_{\substack{j=2\\ \left|k_j(\sigma)\right|\geq
m}}^{n+1}\log\left(\sqrt{\lambda_{\tau(m),k}
^2+c(\sigma_{\tau(m),k})-c(\sigma)}
+\left|k_j(\sigma)\right|+\rho_j\right).
\end{split}
\]
\end{prop}
\begin{proof}
Let $\sigma\in\hat M$. By \eqref{BrKM} the highest weights of $\nu\in\hat K$ 
with $m_{\nu}(\sigma_{\tau(m),k})\neq 0$ are of the form
$\Lambda(\sigma_{\tau(m),k})
-\mu$, where $\mu\in\{0,1\}^n$. Now assume that also $[\nu\colon\sigma]\ne0$.  
Then by \cite[Theorem 8.1.4]{Knapp2} we have $k_j(\sigma_{\tau(m),k})\ge 
k_j(\sigma)$. Hence if $\sigma\in\hat M$ is such that $\left[\nu:\sigma\right]
m_{\nu}(\sigma_{\tau(m),k})\neq 0$ for some $\nu\in\hat K$, 
it follows from \eqref{csigma} that 
\begin{equation}\label{estim11}
c(\sigma_{\tau(m),k})-c(\sigma)\geq 0.
\end{equation}
Thus the proposition follows from Lemma
\ref{Ratcontr}, equation \eqref{FormJ} and Lemma \ref{lemc}.
\end{proof}

\begin{prop}\label{MJ}
Let $k\in \{0,\dots,n\}$.
There exists a constant $C$ such that for every $m$ one has
\begin{align*}
\left|\frac{d}{ds}\biggr|_{s=0}\frac{\mathcal{M}J(s;\sigma_{\tau(m),k})}{
\Gamma(s)}\right| \leq Cm^{\frac{n(n+1)}{2}}\log{m}.
\end{align*}
\end{prop}
\begin{proof}
Let $\nu\in\hat{K}$ such that $m_{\nu}(\sigma_{\tau(m),k})\neq 0$. 
Let $\sigma\in\hat{M}$ such that $\left[\nu:\sigma\right]\neq 0$. 
Then \eqref{estim11} holds as shown in the proof of the previous proposition.
Hence
\begin{align*}
m\leq\sqrt{\lambda_{\tau(m),k}^2+c(\sigma_{\tau(m),k})-c(\sigma)}\leq
\sqrt{\lambda_{\tau(m),k}^2+c(\sigma_{\tau(m),k})}.
\end{align*}
By \eqref{csigma}, \eqref{sigmaneu} and \eqref{lambdaneu} there exists a
constant $C_1$ which is independent of $\nu$ and $\sigma$ such that
for every $m$ we have
\begin{align*}
m\leq\sqrt{\lambda_{\tau(m),k}^2+c(\sigma_{\tau(m),k})-c(\sigma)}\leq C_1 m.
\end{align*}
For $\nu\in\hat K$ as above, it follows from \eqref{BrKM} and \eqref{sigmaneu} 
that for every $j\in\{2,\dots,n+1\}$ one has
\begin{align*}
k_{j}(\nu)\leq m+\tau_1.
\end{align*}
Thus there exists a constant $C_2$ which is independent of $\nu$ and $\sigma$
such that for every $m$ we have
\begin{align*}
\sum_{j=2}^{n+1}\sum_{m\leq l\leq
k_j(\nu)}\left|\log{\left(\sqrt{\lambda_{\tau(m),k}
^2+c(\sigma_{\tau(m),k})-c(\sigma)
}+l+\rho_j\right)}\right|\leq C_2\log{m}.
\end{align*}
By Proposition \ref{propJ} it follows that there exists a constant $C_3$ such 
that for every $m\in\N$ we have
\begin{align*}
\left|\frac{d}{ds}\biggr|_{s=0}\frac{\mathcal{M}J(s;\sigma_{\tau(m),k})}{
\Gamma(s)}\right|\leq & C_3
\log{m}\sum_{\nu\in\hat{K}}\left|m_{\nu}(\sigma_{\tau(m),k})\right|\sum_{
\sigma\in\hat{M}}\left
[\nu:\sigma\right]\dim(\sigma)\\=
&C_3\log{m}\sum_{\nu\in\hat{K}}\left|m_{\nu}(\sigma_
{\tau(m),k})\right|\dim(\nu).
\end{align*}
Now by \eqref{BrKM} the number of $\nu\in\hat{K}$ with
$m_{\nu}(\sigma_{\tau(m),k})\neq 0$ is bounded by $2^n$ and one has
$\left|m_{\nu}(\sigma_{\tau(m),k})\right|\leq 1$ for every $\nu\in\hat{K}$.
Let $\Lambda(\nu)\in\mathfrak{b}_{\C}^*$ be the highest weight of $\nu$
as in \eqref{Darstellungen von K}
Then by Weyl's dimension formula \cite[Theorem 4.48]{Knapp} we have 
\begin{align}\label{Dimension nu}
\dim(\nu)=&\prod_{\alpha\in\Delta^{+}(\mathfrak{k}_{\mathbb{C}},
\mathfrak{b}_{\C})}\frac{\left<\Lambda(\nu)+\rho_{K},\alpha\right>}
{\left<\rho_{K},\alpha\right>}\\
=&\prod_{i=2}^{n+1}(k_i(\nu)+\rho_i+1/2)\prod_{j=i+1}^{n+1}\frac{\left(k_{i}
(\nu)
+\rho_{i}+1/2\right)^{2}-\left(k_{j}(\nu)+\rho_{j}+1/2\right)^{2}}
{(\rho_{i}+1/2)^{2}-(\rho_{j}+1/2)^{2}}.
\end{align}
By\eqref{BrKM} the highest weights of $\nu\in\hat K$ with 
$m_{\nu}(\sigma_{\tau(m),k})\neq 0$ are of the form 
$\Lambda(\sigma_{\tau(m),k})-\mu$, where
$\mu\in\{0,1\}^n$.  Using  \eqref{sigmaneu} it follows that there exists
$C_4>0$, which is independent of $m$, such that for each
$\nu\in\hat{K}$ with $m_{\nu}(\sigma_{\tau(m),k})\neq 0$ one has
\begin{align*}
\dim(\nu)\leq C_4m^{\frac{n(n+1)}{2}}.
\end{align*}
This proves the proposition.
\end{proof}
Summarizing, we have proved the following proposition.
\begin{prop}\label{PropJ}
For $s\in\C$, $\Real(s)>0$ the integral
\begin{align*}
\mathcal{M}J(s;\tau(m)):=\int_{0}^{\infty}t^{s-1}J(k_t^{\tau(m)})dt
\end{align*}
converges and $\mathcal{M}J(s;\tau(m))$ admits a meromorphic continuation to 
$\C$ with at most simple a pole at $0$. Let
\begin{align*}
\mathcal{M}J(\tau(m)):=\frac{d}{ds}\biggr|_{s=0}\frac{\mathcal{M}J(s;\tau(m))}{
\Gamma(s)}.
\end{align*}
Then there exists a constant $C$ such that for every $m\in\N$ one has
\begin{align*}
|\mathcal{M}J(\tau(m))|\leq Cm^{\frac{n(n+1)}{2}}\log{m}.
\end{align*}
\end{prop}
\begin{proof}
By Proposition \ref{auxk} one has
\begin{align*}
\mathcal{M}J(s;\tau(m))=
\sum_{k=0}^n(-1)^{k+1}\mathcal
{M}J(s;\sigma_{\tau(m),k}).
\end{align*}
The Proposition follows from Proposition \ref{propJ} and Proposition \ref{MJ}. 
\end{proof}
Now by equation \ref{Def Tor}, equation \ref{FTK} and Proposition \ref{auxk}
we have
\[
\begin{split}
\log{T_X(\tau(m))}
=\frac{1}{2}\bigl(\mathcal{M}I(\tau(m))&+\mathcal{M}H(\tau(m))
+\mathcal{M}T(\tau(m))\\
&+\mathcal{M}\mathcal{I}(\tau(m))+\mathcal{M}
J(\tau(m))\bigr).
\end{split}
\]
Combining equation \ref{mellin} and Propositions \ref{Idkontr}, \ref{PropH},
\ref{propT},  \ref{propI} 
and \ref{PropJ}, Theorem \ref{theo1} and Theorem \ref{Theorem2} follow.

\end{document}